\documentclass[10pt,reqno]{amsart}
\usepackage{graphicx,epsfig}
\usepackage{epstopdf}
\usepackage{pdfpages}
\usepackage{amssymb}
\usepackage{empheq}
\usepackage{cases}
\usepackage{amsthm,amsmath}
\usepackage{caption,lipsum}
\usepackage{stmaryrd}
\usepackage{tabularx}
\usepackage{color}
\usepackage{empheq,float}
\DeclareGraphicsExtensions{.eps}
\usepackage{hyperref}

\newtheorem{mytheo}{Theorem}[section]

\newtheorem{mydef}[mytheo]{Definition}

\newtheorem{algo}[mytheo]{Algorithm}

\newcommand{\eps}{\epsilon}

\newcommand{\norm}[1]{\left\Vert#1\right\Vert}
\newcommand{\RR}{\mathbb{R}}
\numberwithin{equation}{section}

\begin{document}

\title[Continuous-stage symplectic adapted exponential methods]{Continuous-stage symplectic adapted exponential\\ methods for charged-particle dynamics with\\ arbitrary electromagnetic fields}

\author[T. Li]{Ting Li}\address{\hspace*{-12pt}T.~Li: School of Mathematics and Statistics, Xi'an Jiaotong University, 710049 Xi'an, China}
\email{litingmath@stu.xjtu.edu.cn}

\author[B. Wang]{Bin Wang}\address{\hspace*{-12pt}Corresponding author. B.~Wang: School of Mathematics and Statistics, Xi'an Jiaotong University, 710049 Xi'an, China}
\email{wangbinmaths@xjtu.edu.cn}
%\urladdr{http://gr.xjtu.edu.cn/web/wangbinmaths/home}

%\subjclass[2010]{Primary }
%
%\keywords{Vlasov-Poisson equation, Three dimensions, Strong magnetic field, Varying direction, Uniformly accurate method, Particle-In-Cell}

\date{}

\dedicatory{}

\begin{abstract}
This paper is devoted to the numerical  symplectic  approximation of the charged-particle
dynamics (CPD) with arbitrary electromagnetic fields. By utilizing continuous-stage methods and exponential integrators,
a general class of symplectic methods is formulated for CPD under a homogeneous magnetic field.  Based on  the derived symplectic conditions,  two practical  symplectic methods up to order four are constructed where the error estimates show that the proposed second order scheme has a uniform accuracy in the position w.r.t. the strength of the magnetic field. Moreover, the symplectic methods are extended to CPD under  non-homogeneous magnetic fields and three algorithms are formulated. Rigorous error estimates are investigated for the proposed methods and one method is proved to have a uniform accuracy in the  position w.r.t. the strength of the magnetic field.
 Numerical experiments are provided for CPD under  homogeneous and non-homogeneous magnetic fields, and the numerical results support the theoretical analysis and demonstrate the remarkable numerical behavior of our methods.
 \\ \\
{\bf Keywords:} Charged particle dynamics, Symplectic methods,  Uniform error bounds, Exponential methods, Arbitrary electromagnetic fields. \\ \\
{\bf AMS Subject Classification:} 65P10, 65L05, {65L20,} 78A35.
\end{abstract}

\maketitle

\section{Introduction}\label{sec:intro}
In this article, we are concerned with  numerical symplectic methods for the following charged-particle dynamics (CPD)
 (see {\cite{Benettin94,Hairer2018,Hairer2017,Lubich2017,Hairer2020,Lubich2020,WZ}})
 \begin{equation}\label{charged-particle sts-cons}
\begin{array}[c]{ll}
{\ddot{x}(t)=\dot{x}(t) \times  \frac{B(x(t))}{\epsilon} +F(x(t))}, \quad
x(0)=x_0,\quad \dot{x}(0)=\dot{x}_0,
\end{array}
\end{equation}
where  $x(t)\in \RR^3$  represents  the position of a particle
moving in an electro-magnetic field, the
magnetic field is described by $B(x) = \nabla_x \times A(x)$ with the vector potential
$A(x) \in \RR^3 $,  $\epsilon \in(0,1]$ is a dimensionless parameter  which is related to the strength of the magnetic field, and $F(x)$ is the nonlinear function of the form $F(x)=-\nabla U(x)$ with scalar potential $U(x)$. The initial data $x_0$ and $\dot{x}_0$ are
two given real-valued vectors independent of $\eps$.  In this paper, we denote the {magnetic field} $B(x)$ by
 $B(x)=\big(B_1(x),B_2(x),B_3(x)\big)^{\intercal}$, where $B_i(x) \in \RR $ for $i=1,2,3.$ Then it follows from  the definition of the cross
 product that
$ \dot{x} \times  B(x) =  \tilde{B}(x) \dot{x},$ where the skew symmetric
matrix $\tilde{B}(x)$ is given by
$\tilde{B}(x)=\left(
                   \begin{array}{ccc}
                     0 & B_3(x) & -B_2(x) \\
                     -B_3(x) & 0 & B_1(x) \\
                     B_2(x) & -B_1(x) & 0 \\
                   \end{array}
                 \right).
$
According to the conjugate momentum $ p =  v + \frac{A(x)}{\epsilon}=v
-\frac{\tilde{B}(x)}{2\epsilon}x$, the CPD
\eqref{charged-particle sts-cons} can be converted into a
Hamiltonian system with the non-separable Hamiltonian
 \begin{equation}\label{H}H(x,p)=\frac{1}{2} \norm{p+\frac{\tilde{B}(x)}{2\epsilon}x}^2 + U(x),\end{equation}
 where $\norm{ \cdot}$ denotes the Euclidean norm.
 For  the non-separable
Hamiltonian \eqref{H}, it is well known that its  flow is
symplectic (\cite{Hairer2002}), i.e., it preserves the differential 2-form
$\sum\limits_{i=1}^{3} dx_i\wedge dp_i.$ In this work, we are devoted to the formulation and analysis of  a novel kind of symplectic methods
for CPD.

{The CPD  \eqref{charged-particle sts-cons}   arises in many applications  such as plasma physics, astrophysics and  magnetic fusion devices} (see, e.g.
{\cite{Arnold97,Benettin94,Cary2009,Northrop63}}) {and  a large variety of methods has
been constructed and developed for numerical solutions of the system}. For the normal magnetic field $\epsilon \approx 1$,
the Boris method proposed in
\cite{Boris} is popular due to its simplicity and good properties (cf. \cite{Hairer2017, Qin2013}). Other classical schemes have also been proposed for CPD such as multistep methods  \cite{Lubich2017,wangwu2020}, splitting methods \cite{Knapp} and so on.
In the recent few {decades}, structure-preserving  methods for differential equations have received more and more attention (\cite{bao-zhao,Feng1985,E. Hairer, Hairer2002,S. MacNamara,Sanz-Serna1988,Suris}). There are various structure-preserving  methods which have been applied to CPD,   including volume-preserving methods
\cite{He2015}, energy-preserving integrators \cite{Brugnano2020, Brugnano2019, li2022,Chacon}, symmetric methods \cite{Lubich2017, wangwu2020} and symplectic methods  \cite{ S. Blanes2012, He2017,Y. Shi, Tao2016, Xiao,Xiao21,Zhang2016}. It is noted that these structure-preserving  methods are all designed for CPD with normal magnetic field {$ \epsilon \approx 1$} and no analysis has been given for  the strong regime  $0<\epsilon \ll 1$.   If one considers the  methods mentioned above for CPD with  $0<\eps \ll 1$,
the accuracy  usually depends on  $1/\eps^j$ for some $j>0$ and this makes the methods inefficient   for small $\eps$.
Concerning the methods for CPD with strong regime  $0<\epsilon \ll 1$,  to the best of my knowledge,  most existed algorithms are devoted to the accuracy or long time near conservation   \cite{VP1,Zhao,VP4,VP5,VP6,Hairer2018,Lubich2020,Lubich2022,WZ}. Structure-preserving  methods have not been considered and analysed for CPD with strong regime  $0<\epsilon \ll 1$.

From the point of Hamiltonian system,  it is well known that symplecticity is a very
important property, which has been
investigated by many researchers (\cite{Feng1985, Webb2014}). Various symplectic methods have been derived such as
symplectic Runge-Kutta  methods
\cite{Sanz-Serna1988}, symplectic  Runge-Kutta-Nystr\"{o}m  methods
\cite{Suris}, and continuous-stage symplectic methods  \cite{E. Hairer, Miyatake, Tang2018}.
However, as stated above,  symplectic methods for CPD \cite{He2017, Y. Shi, Tao2016, Xiao,Xiao21,Zhang2016}  have only been designed for
the system under normal magnetic fields and there is no numerical approximation or error analysis  for   strong magnetic fields.

In this article, we  construct a new kind of  symplectic methods for CPD with  arbitrary electromagnetic fields and  {establish} the rigorous error estimates. To formulate the methods, the ideas of exponential methods   (see, e.g. \cite{bao-cai, Frenod, Hairer2002, M.Hochbruck2017, Mei2017, wang-2016}) and continuous-stage  methods
will be employed. Moreover, the convergence of the obtained symplectic methods is rigorously researched, {and the error estimates show that some proposed methods {have}  a uniform accuracy in the position w.r.t. $\eps$. This means that symplectic methods with uniform (w.r.t. $\eps$) error bounds are achieved {in this paper, and performances superior to the standard symplectic methods are
demonstrated on some examples.}  %the good performance in accuracy and long time energy conservation.

In the analysis of symplecticity, we consider the differential 2-form for the methods. However, since the
 Hamiltonian is non-separable and the new methods are exponential type,  some challenges and
difficulties emerge in the achievement of symplectic conditions. To make the derivation go smoothly,
we  consider  some transformations of the original system and the methods, and make well use of the antisymmetry of $\tilde{B}(x)$.
{These} transformations  can keep the symplecticity of the methods, and thus
symplectic methods can be formulated by deriving the symplectic conditions for the transformed methods.
Concerning the convergence of the methods for CPD under nonhomogeneous magnetic fields,
 exponential variation-of-constants formula unfortunately does not hold and thence the local errors of exponential methods cannot be derived.
 In order to  overcome this difficulty, we design a novel system called as linearized system and establish exponential variation-of-constants formula for it. The error between this linearized system and {the original} one is deduced and based on which, we
 manage to derive the  convergence of the methods. Moreover, following this approach, a method with uniform second order error bound and very good long time energy conservation  is obtained for CPD under maximal ordering magnetic fields  {\cite{scaling1,Lubich2020,scaling2}}. For the system of this case, two filtered Boris algorithms were presented in \cite{Lubich2020} {and they} were proved to have uniform second order error bound but without long time behaviour. An energy-preserving method was constructed in \cite{WZ} and  uniform first order error estimate was established. Compared with {the methods derived in \cite{Lubich2020,WZ}}, the method proposed in this paper {has both uniform second order accuracy and symplecticity, and is} a brand-new method to CPD under maximal ordering magnetic fields.}

%They are different from normal
%exponential methods since symplectic exponential methods
%are usually implicit for general first-order system
%(\cite{Mei2017}). Moreover, the connections and
%differences between the classical exponential methods and the
%methods presented in this paper are discussed.

The rest of this paper is organised as follows.   In Section \ref{sec1: methods},  for solving CPD with homogeneous magnetic field, we propose the
scheme of continuous-stage adapted exponential methods and derive its symplectic conditions. Then two practical  symplectic methods up to order four  {are presented in
Section \ref{sec:prac meth} and   their properties including convergence and implementations are also studied there}. At the end of Section \ref{sec:prac meth}, { one
numerical experiment is provided and the results  show the favorable   behaviour of
the new methods in    comparison  with the Boris method and two symplectic Runge-Kutta methods}. Moreover,  {the extension and application of  the proposed symplectic methods to CPD with non-homogeneous magnetic fields  are discussed in Section \ref{sec:general}. Three practical methods as well as their error estimates are proposed and two numerical tests are
given to demonstrate that they are sharp.} Section \ref{sec:conclusions} is
devoted to the conclusions of this paper.

\section{{Adapted exponential methods} and symplectic conditions}\label{sec1: methods}
In this section, we shall propose the novel methods for  \eqref{charged-particle sts-cons}
in a homogeneous magnetic field $B=\big(B_1,B_2,B_3\big)^{\intercal}$ and analyse their symplectic conditions.
{We start by  the  variation-of-constants formula of the CPD \eqref{charged-particle sts-cons} which reads (see
\cite{Lubich2020})
\begin{equation}\label{VOC}
\begin{aligned}
x(t_n+h)&=x(t_n) +h\varphi_1( hM ) v(t_n)+h^2
\int_{0}^1(1-\tau) \varphi_1((1-\tau) hM ) F(x(t_n+h\tau))
d\tau,\\
v(t_n+h)&=\varphi_{0}( hM )v(t_n)+h
\int_{0}^1
 \varphi_{0}((1-\tau) hM ) F(x(t_n+h\tau))  d\tau,
\end{aligned}
\end{equation}
for any step size $h\geq0$ and $t_n=nh$, where $v(t):=\dot{x}(t)$,
  $M=\frac{\tilde{B}}{\epsilon}$ and the $\varphi$-functions are defined by (see
\cite{Hochbruck2005})
\begin{equation*}\label{phi}
 \varphi_0(z)=e^{z},\ \ \varphi_k(z)=\int_{0}^1
e^{(1-\sigma)z}\frac{\sigma^{k-1}}{(k-1)!}d\sigma, \ \
k=1,2,\ldots.
\end{equation*}

To derive the novel methods, we consider  the variation-of-constants formula  \eqref{VOC} and utilize the idea of continuous-stage methods, which leads to the following definition.

\begin{mydef}
\label{scheme 1}  \textbf{(Adapted exponential
 methods)} By denoting
the numerical approximation $x_n\approx x(t_n),\, v_n\approx v(t_n)$ and choosing  $v_0=\dot{x}_0$, a continuous-stage adapted exponential
 method for solving \eqref{charged-particle sts-cons} is
defined {as}
\begin{equation}
\begin{aligned}\label{CSAEI}
&X_{\tau}=x_{n}+ \tau h\varphi_{1}(\tau hM ) v_{n}+
h^2 \int_{0}^{1}{\alpha}_{\tau \sigma}(hM)F (X_\sigma)d\sigma,\\
&x_{n+1}=x_{n}+ h\varphi_1(hM)
v_{n}+h^2 \int_{0}^{1}\beta_{\tau}(hM)F (X_{\tau})d\tau,\ \ v_{n+1}=\varphi_0(hM)v_{n}+h
\int_{0}^{1}\gamma_{\tau}(hM)F(X_{\tau})d\tau,
\end{aligned}
\end{equation}
where $h$ is the step size, $ X_{\tau} $ is a function  about $\tau $, and
 the coefficients $ {\alpha}_{\tau \sigma}(hM) $, $ \beta_{\tau}(hM) $ and $
\gamma_{\tau}(hM) $ are functions {depending} on $ hM$.
For the  method adapted to the Hamiltonian system \eqref{H}, we
consider the continuous-stage adapted exponential
 method \eqref{CSAEI}  as well as a momentum calculation $$p_{n+1} = v_{n+1} -\frac{\tilde{B}}{2\epsilon}
x_{n+1}.$$
\end{mydef}

In what follows, we derive the symplectic conditions of the  method given in Definition \ref{scheme 1}.
When applied to the non-separable Hamiltonian \eqref{H}, the  method is called as symplectic
if it preserves the symplecticity exactly, i.e. (\cite{Hairer2002}),
\begin{equation*}\label{NMSD}
\sum\limits_{J=1}^{3}dx_{n+1}^{J} \wedge
dp_{n+1}^{J}=\sum_{J=1}^{3}dx_{n}^{J} \wedge dp_{n}^{J},
\end{equation*}
where the superscript $(\cdot)^{J}$ denotes the $J$th component of a vector.
To prove this statement, we first make some transformations of the system
\eqref{charged-particle sts-cons}  and  of the method \eqref{CSAEI}.
Since $\tilde{B}$ is a skew-symmetric matrix, it
can be expressed as $ \tilde{B}=K \Omega K^\textup{H}, $ where $K$
is a unitary matrix and $\Omega=
\textmd{diag}(-{\norm{B}}\mathrm{i},0,{\norm{B}}\mathrm{i}) $.
Now we denote the new variables as
 \begin{equation}\label{change of variable}
 \tilde{x}(t)= K^\textup{H} x(t),\quad \tilde{v}(t)= K^\textup{H} v(t).
 \end{equation}
Then the system
\eqref{charged-particle sts-cons} becomes
\begin{equation*}\label{necharged-sts-first order}
\begin{array}[c]{ll}
\frac{d}{dt }\left(
  \begin{array}{c}
    \tilde{x} \\
    \tilde{v} \\
  \end{array}
\right)  =\left(
            \begin{array}{cc}
              0 & I \\
             0  & \tilde{\Lambda} \mathrm{i} \\
            \end{array}
          \right) \left(
  \begin{array}{c}
    \tilde{x} \\
    \tilde{v} \\
  \end{array}
\right)+\left(
  \begin{array}{c}
   0 \\
    \tilde{F}(\tilde{x}) \\
  \end{array}
\right),\quad \left(
                \begin{array}{c}
                  \tilde{x}_0 \\
                  \tilde{v}_0 \\
                \end{array}
              \right)=\left(
                        \begin{array}{c}
                K^\textup{H} x_0 \\
                 K^\textup{H} \dot{x}_0\\
                        \end{array}
                      \right),
\end{array}
\end{equation*}
where $\tilde{\Lambda}=
\textmd{diag}(-\tilde{\zeta},0,\tilde{\zeta})$ with
$\tilde{\zeta}=\frac{{\norm{B}}}{\epsilon}$ and
$\tilde{F}(\tilde{x})=K^\textup{H} F(K \tilde{x})=
-\nabla_{\tilde{x}} U(K\tilde{x})$.

It is noted that    the vector $\tilde{x}$ is denoted by $\tilde{x}=(\tilde{x}^{1},\tilde{x}^{2},\tilde{x}^{3})^{\intercal}$ and all
the vectors in $\mathbb{R}^{3}$ or $\mathbb{C}^{3}$ use the same notation in this paper.
For the  variables \eqref{change of variable} and the property of $K$, we
observe that
$\tilde{x}^{1}=\bar{\tilde{x}}^{3}$,
$\tilde{v}^{1}=\bar{\tilde{v}}^{3}$, and $\tilde{x}^{2}$,
$\tilde{v}^{2} \in \mathbb{R}$.  We shall denote $\bar{\tilde{x}}^{3}$
as the conjugate of $\tilde{x}^{3}$ and the same expression is applied to  other functions.
In the light of  \eqref{change of variable}, the corresponding Hamiltonian system \eqref{H} is given by
\begin{equation}\label{new-H}
\begin{array}[c]{ll}{\dot{\tilde{x}}} =\nabla_{\tilde{p}} \tilde{H}(\tilde{x},\tilde{p})=\tilde{p}+\frac{1}{2}
\tilde{\Lambda} \mathrm{i}\tilde{x},
\ \ \ \  {\dot{\tilde{p}}}=-\nabla_{\tilde{x}} \tilde{H}(\tilde{x},\tilde{p})=-\frac{1}{2}
\big(\tilde{\Lambda} \mathrm{i}\big)^\textup{H} \big(\tilde{p}+\frac{1}{2}
\tilde{\Lambda} \mathrm{i}
\tilde{x}\big)-\nabla_{\tilde{x}}\tilde{U}(K\tilde{x}),
\end{array}
\end{equation}with
$\tilde{H}(\tilde{x},\tilde{p}) =  \frac{1}{2} \norm{\tilde{p}+\frac{1}{2}
\tilde{\Lambda} \mathrm{i}\tilde{x}}^2+\tilde{U}(K\tilde{x}).$
For this transformed Hamiltonian  system,  the  method \eqref{CSAEI} takes the {following transformed} form
\begin{equation}\label{TRM}
\begin{aligned}
&\tilde{X}_{\tau}=\tilde{x}_{n}+ h\tau \varphi_1\big(\tau W\big) \tilde{v}_{n}+
h^2 \int_{0}^{1}{\alpha}_{\tau \sigma}\big(W\big)\tilde{F} \big(\tilde{X}_\sigma\big)d\sigma,\\
& \tilde{x}_{n+1}=\tilde{x}_{n}+ h\varphi_1\big(W\big)
\tilde{v}_{n}+h^2 \int_{0}^{1}\beta_{\tau}\big(W\big)\tilde{F} \big(\tilde{X}_{\tau}\big)d\tau,\\
&\tilde{v}_{n+1}=\varphi_0\big(W\big)\tilde{v}_{n}+h
\int_{0}^{1}\gamma_{\tau}\big(W\big)\tilde{F}\big(\tilde{X}_{\tau}\big)d\tau,\ \ \tilde{p}_{n+1} = \tilde{v}_{n+1} -\frac{1}{2}
\tilde{\Lambda}\mathrm{i}\tilde{x}_{n+1},
\end{aligned}
\end{equation}
where $W=h\tilde{\Lambda}\mathrm{i}$.

Now we present symplectic conditions of the {transformed method \eqref{TRM} and then get the symplecticity of the method} given in Definition \ref{scheme 1}.

\begin{mytheo}\label{thm: symp}  \textbf{(Symplecticity)}
{(a)} For solving {the transformed Hamiltonian {system \eqref{new-H}}, the map $(\tilde{x}_n,\tilde{p}_n ) \rightarrow (\tilde{x}_{n+1},\tilde{p}_{n+1})$ determined by the transformed method \eqref{TRM}} is
symplectic  if the coefficients satisfy
\begin{equation}\label{17}
\begin{aligned}
 (i)\quad &\gamma_{\tau}(W) -W\beta_{\tau}(W)
=d_{\tau} I, \ \ d_\tau\in \mathbb{C},\\
(ii)\quad &\gamma_{\tau}(W)  ( \bar{\varphi}_{1}(W)-\tau\bar{\varphi}_{1}(\tau W) )
=\beta_{\tau}(W)  ( e^{-W}+W
\bar{\varphi}_{1}(W)-\tau W \bar{\varphi}_{1}(\tau W) ),\\
(iii)\quad &\bar{\beta}_{\sigma}(W) \gamma_{\tau}(W) - W\bar{\beta}_{\sigma}(W)
\beta_{\tau}(W)/2 -\bar{{\alpha}}_{\tau \sigma}(W) ( \gamma_{\tau}(W)
  -W\beta_{\tau}(W) )  \\
  \ &=
\beta_{\tau}(W)  \bar{ \gamma}_{\sigma}(W) +W \beta_{\tau}(W)
\bar{\beta}_{\sigma}(W)/2 -{\alpha}_{\sigma \tau}(W)  ( \bar{\gamma}_{\sigma}(W)
  +W\bar{\beta}_{\sigma}(W)    ).
\end{aligned}
\end{equation}
{(b)} {If the transformed method \eqref{TRM} is symplectic, the map $(x_n,p_n ) \rightarrow (x_{n+1},p_{n+1})$ determined by Definition \ref{scheme 1} is
symplectic for   solving the Hamiltonian system \eqref{H}.}
\end{mytheo}

\label{necharged-sts-first order}

\begin{proof}
{ We first prove the second statement {(b)}.
 By the notation of differential
2-form and with the help of $K$,} we have
\begin{equation}
\begin{aligned}
&\sum\limits_{J=1}^{3}dx_{n+1}^{J} \wedge
dp_{n+1}^{J}=\sum\limits_{J=1}^{3}d\bar{x}_{n+1}^{J} \wedge
dp_{n+1}^{J}=\sum\limits_{J=1}^{3}d\big(\bar{K}\bar{\tilde{x}}_{n+1}\big)^{J}
\wedge
d\big(K\tilde{p}_{n+1}\big)^{J}\\
=& \sum\limits_{J=1}^{3}\bigg({ d}
\sum\limits_{i=1}^{3}\big(\bar{K}_{Ji}\bar{\tilde{x}}_{n+1}^{i}\big)\bigg)\wedge
 \bigg( {
d} \sum\limits_{\iota=1}^{3}\big(K_{J\iota}\tilde{p}_{n+1}^{\iota}\big)\bigg)
= \sum\limits_{J=1}^{3}\bigg( \sum\limits_{i=1}^{3}\big(\bar{K}_{Ji}{
d}\bar{\tilde{x}}_{n+1}^{i}\big)\bigg)\wedge
 \bigg(  \sum\limits_{\iota=1}^{3}\big(K_{J\iota}{ d}\tilde{p}_{n+1}^{\iota}\big)\bigg)\\
= & \sum\limits_{J=1}^{3}\sum\limits_{i=1}^{3}\sum\limits_{\iota=1}^{3} \bar{K}_{Ji}  K_{J\iota} \big({
d}\bar{\tilde{x}}_{n+1}^{i} \wedge  { d}\tilde{p}_{n+1}^{\iota}\big)
 =  \sum\limits_{i=1}^{3}  {
d}\bar{\tilde{x}}_{n+1}^{i} \wedge  { d}\tilde{p}_{n+1}^{i}=
\sum\limits_{J=1}^{3}  { d}\bar{\tilde{x}}_{n+1}^{J} \wedge  {
d}\tilde{p}_{n+1}^{J}.\nonumber
\end{aligned}\label{td}
\end{equation}
Likewise, one gets $ \sum\limits_{J=1}^{3}dx_{n}^{J}
\wedge dp_{n}^{J}= \sum\limits_{J=1}^{3}  {
d}\bar{\tilde{x}}_{n}^{J} \wedge { d}\tilde{p}_{n}^{J}.$ Hence {the second statement is hold.

 For the first statement {(a)},
we need} to prove
$ \sum\limits_{J=1}^{3}  { d}\bar{\tilde{x}}_{n+1}^{J} \wedge  {
d}\tilde{p}_{n+1}^{J}= \sum\limits_{J=1}^{3}  {
d}\bar{\tilde{x}}_{n}^{J} \wedge { d}\tilde{p}_{n}^{J},
$
i.e.,
\begin{equation}\label{imp res}
\begin{aligned} \sum\limits_{J=1}^{3}  { d}\bar{\tilde{x}}_{n+1}^{J} \wedge  {
d}\tilde{v}_{n+1}^{J}-\frac{1}{2} \sum\limits_{J=1}^{3}  {
d}\bar{\tilde{x}}_{n+1}^{J} \wedge  { d}
\big(\tilde{\Lambda}^{J}\mathrm{i}\tilde{x}_{n+1}^{J}\big)=
\sum\limits_{J=1}^{3}  { d}\bar{\tilde{x}}_{n}^{J} \wedge {
d}\tilde{v}_{n}^{J}-\frac{1}{2}  \sum\limits_{J=1}^{3}  {
d}\bar{\tilde{x}}_{n}^{J} \wedge { d}
\big(\tilde{\Lambda}^{J}\mathrm{i}\tilde{x}_{n}^{J}\big).
\end{aligned}
\end{equation}

 \textbf{(I) Computation of the left hand side of
\eqref{imp res}.}
 It is noted that
$ \overline{{ d}\bar{\tilde{x}}_{n}^{J} \wedge { d}
\tilde{v}_{n}^{J}} ={ d} \tilde{x}_{n}^{J} \wedge { d}
\bar{\tilde{v}}_{n}^{J}$ and $ { d}\bar{\tilde{x}}_{n}^{J} \wedge {
d}  \tilde{x}_{n}^{J} \in \mathrm{i} \mathbb{R}$, and
denote the  method by{
\begin{equation}
\begin{aligned}
&\tilde{X}^{J}_{\tau}=\tilde{x}^{J}_{n}+ h\tau \varphi_1\big(\tau{W}^{J}\big) \tilde{v}^{J}_{n}+
h^2 \int_{0}^{1}{\alpha}_{\tau \sigma}\big({W}^{J}\big)\tilde{F}^{J}_{\sigma}d\sigma,\\
&\tilde{x}^{J}_{n+1}=\tilde{x}^{J}_{n}+
h\varphi_1({W}^{J})
\tilde{v}^{J}_{n}+h^2 \int_{0}^{1}\beta_{\tau}\big({W}^{J}\big)\tilde{F}^{J}_{\tau}d\tau,\ \ \tilde{v}^{J}_{n+1}=e^{{W}^{J}}\tilde{v}^{J}_{n}+h
\int_{0}^{1}\gamma_{\tau}\big({W}^{J}\big)\tilde{F}^{J}_{\tau}d\tau,
\end{aligned}\label{a}
\end{equation}}where the superscript $(\cdot)^{J}$ is the $J$th component of a vector
or a matrix and $\tilde{F}^{J}_{\tau}$ is  the $J$th component of
$\tilde{F}\big(\tilde{X}_{\tau}\big)$.
 Based on  the above  {expression}, in what follows we
  calculate the left hand side of
\eqref{imp res}.

Inserting the scheme of \eqref{a} into \eqref{imp res}, it arrives that
\begin{equation}
\begin{aligned}\label{aaa1}
&\sum\limits_{J=1}^{3}\bigg(d\bar{\tilde{x}}_{n+1}^{J}\wedge d\tilde{v}_{n+1}^{J}-\frac{1}{2}d\bar{\tilde{x}}_{n+1}^{J}\wedge d\big(\tilde{\Lambda}^{J}\mathrm{i}\tilde{x}_{n+1}^{J}\big)\bigg)
= \sum\limits_{J=1}^{3}\bigg(e^{{W}^{J}} d\bar{\tilde{x}}_{n}^{J} \wedge
d\tilde{v}_{n}^{J}
+h\int_{0}^{1}\Big(\gamma_{\tau}\big({W}^{J}\big)
d\bar{\tilde{x}}_{n}^{J} \wedge  d\tilde{F}_{\tau}^{J}\Big)d\tau\\
&+he^{{W}^{J}} \bar{\varphi}_{1}\big({W}^{J}\big)  d\bar{\tilde{v}}_{n}^{J}  \wedge   d\tilde{v}_{n}^{J}
+h^{2}\int_{0}^{1}\Big(\bar{\varphi}_{1}\big({W}^{J}\big)\gamma_{\tau}\big({W}^{J}\big)  d\bar{\tilde{v}}_{n}^{J} \wedge  d\tilde{F}_{\tau}^{J}\Big)d\tau
+h^{2}  e^{{W}^{J}} \int_{0}^{1}\Big(\bar{\beta}_{\tau}\big({W}^{J}\big)
d\bar{\tilde{F}}_{\tau}^{J} \wedge  d\tilde{v}_{n}^{J}\Big)d\tau\\
&+h^{3}
\int_{0}^{1}\int_{0}^{1}\Big(\bar{\beta}_{\tau}\big({W}^{J}\big)   \gamma_{\sigma}\big({W}^{J}\big)
d\bar{\tilde{F}}_{\tau}^{J} \wedge d\tilde{F}_{\sigma}^{J}\Big)d\tau d\sigma
- \frac{1}{2}d\bar{\tilde{x}}_{n}^{J}\wedge  d\big(\tilde{\Lambda}^{J}\mathrm{i}\tilde{x}_{n}^{J}\big)
-\frac{1}{2} {W}^{J}  \varphi_{1}\big({W}^{J}\big) d\bar{\tilde{x}}_{n}^{J}\wedge   d\tilde{v}_{n}^{J}\\
&-\frac{1}{2}h {W}^{J}  \int_{0}^{1}\Big(\beta_{\tau}\big({W}^{J}\big)
d\bar{\tilde{x}}_{n}^{J}\wedge d\tilde{F}_{\tau}^{J}\Big)d\tau
-\frac{1}{2}{W}^{J}\bar{\varphi}_{1}\big({W}^{J}\big)    d\bar{\tilde{v}}_{n}^{J}  \wedge d\tilde{x}_{n}^{J}
-\frac{1}{2}h {W}^{J}\bar{\varphi}_{1}\big({W}^{J}\big) \varphi_{1}\big({W}^{J}\big)
 d\bar{\tilde{v}}_{n}^{J} \wedge   d\tilde{v}_{n}^{J}\\
 &-\frac{1}{2}h^{2} {W}^{J}  \varphi_{1}\big({W}^{J}\big) \int_{0}^{1}\Big(\bar{\beta}_{\tau}\big({W}^{J}\big)
 d\bar{\tilde{F}}_{\tau}^{J} \wedge d\tilde{v}_{n}^{J}\Big)d\tau
 -\frac{1}{2}h^{3} {W}^{J} \int_{0}^{1}\int_{0}^{1}\Big(\bar{\beta}_{\tau}\big({W}^{J}\big)   \beta_{\sigma}\big({W}^{J}\big)
 d\bar{\tilde{F}}_{\tau}^{J} \wedge d\tilde{F}_{\sigma}^{J}\Big)d\tau d\sigma \\
 &-\frac{1}{2}h^{2} {W}^{J}\bar{\varphi}_{1}\big({W}^{J}\big)\int_{0}^{1}\Big(\beta_{\tau}\big({W}^{J}\big)   d\bar{\tilde{v}}_{n}^{J}  \wedge   d\tilde{F}_{\tau}^{J}\Big)d\tau
 -\frac{1}{2}h{W}^{J} \int_{0}^{1}\Big(\bar{\beta}_{\tau}\big({W}^{J}\big)
 d\bar{\tilde{F}}_{\tau}^{J} \wedge d\tilde{x}_{n}^{J}\Big)d\tau\bigg).
\end{aligned}
\end{equation}
Considering the properties of coefficient functions and  exterior product, we obtain
\begin{equation*}
\begin{aligned}
&-\sum\limits_{J=1}^{3}\frac{1}{2}h
{W}^{J}\bar{\varphi}_{1}\big({W}^{J}\big) d\bar{\tilde{v}}_{n}^{J}  \wedge d\tilde{x}_{n}^{J}=-\sum\limits_{J=1}^{3}\frac{1}{2}h{W}^{J}
\varphi_{1}\big({W}^{J}\big)    d\bar{\tilde{x}}_{n}^{J}  \wedge d\tilde{v}_{n}^{J},\\
&-\sum\limits_{J=1}^{3}\frac{1}{2}h{W}^{J} \int_{0}^{1}\Big(\bar{\beta}_{\tau}\big({W}^{J}\big)
d\bar{\tilde{F}}_{\tau}^{J} \wedge d\tilde{x}_{n}^{J}\Big)d\tau
=-\sum\limits_{J=1}^{3}\frac{1}{2}h{W}^{J} \int_{0}^{1}\Big(\beta_{\tau}\big({W}^{J}\big)
d\bar{\tilde{x}}_{n}^{J}\wedge d\tilde{F}_{\tau}^{J}\Big)d\tau,\\
&-\sum\limits_{J=1}^{3}h^{2}\int_{0}^{1}\Big(\bar{\beta}_{\tau}\big({W}^{J}\big)e^{{W}^{J}} d\bar{\tilde{F}}_{\tau}^{J} \wedge d\tilde{v}_{n}^{J}\Big)d\tau
=\sum\limits_{J=1}^{3}h^{2}\int_{0}^{1}\Big(\beta_{\tau}\big({W}^{J}\big)e^{-{W}^{J}} d\bar{\tilde{v}}_{n}^{J} \wedge d\tilde{F}_{\tau}^{J}\Big)d\tau,\\
&-\sum\limits_{J=1}^{3}\frac{1}{2}h^{2} {W}^{J} \varphi_{1}\big({W}^{J}\big) \int_{0}^{1}\Big(\bar{\beta}_{\tau}\big({W}^{J}\big)
d\bar{\tilde{F}}_{\tau}^{J} \wedge d\tilde{v}_{n}^{J}\Big)d\tau
=-\sum\limits_{J=1}^{3}\frac{1}{2}h^{2} {W}^{J} \bar{\varphi}_{1}\big({W}^{J}\big) \int_{0}^{1}\Big(\beta_{\tau}\big({W}^{J}\big) d\bar{\tilde{v}}_{n}^{J} \wedge d\tilde{F}_{\tau}^{J}\Big)d\tau.\\
\end{aligned}
\end{equation*}
Thus \eqref{aaa1} can be simplified as{
\begin{equation}
\begin{aligned}
&\sum\limits_{J=1}^{3}\bigg(d\bar{\tilde{x}}_{n+1}^{J}\wedge d\tilde{v}_{n+1}^{J}-\frac{1}{2}d\bar{\tilde{x}}_{n+1}^{J}\wedge d\big(\tilde{\Lambda}^{J}\mathrm{i}\tilde{x}_{n+1}^{J}\big)\bigg)\\
=&\sum\limits_{J=1}^{3}\bigg(d\bar{\tilde{x}}_{n}^{J}  \wedge d\tilde{v}_{n}^{J}
 - \frac{1}{2}d\bar{\tilde{x}}_{n}^{J}\wedge  d\big(\tilde{\Lambda}^{J}\mathrm{i}\tilde{x}_{n}^{J}\big)
+ h\int_{0}^{1}\Big(  \gamma_{\tau}\big({W}^{J}\big)-{W}^{J} \beta_{\tau}\big({W}^{J}\big)           \Big)
\big(d\bar{\tilde{x}}_{n}^{J} \wedge  d\tilde{F}_{\tau}^{J}\big)d\tau\\
&+\Big(he^{{W}^{J}} \bar{\varphi}_{1}\big({W}^{J}\big)
-\frac{1}{2}h {W}^{J}\bar{\varphi}_{1}\big({W}^{J}\big) \varphi_{1}\big({W}^{J}\big)\Big)
d\bar{\tilde{v}}_{n}^{J}  \wedge   d\tilde{v}_{n}^{J}\\
&+h^{2}\int_{0}^{1}\Big(\bar{\varphi}_{1}\big({W}^{J}\big)   \gamma_{\tau}\big({W}^{J}\big)
- \beta_{\tau}\big({W}^{J}\big)  e^{-{W}^{J}}
 -W^J \bar{\varphi}_{1}\big({W}^{J}\big) \beta_{\tau}\big({W}^{J}\big)\Big)
\big(d\bar{\tilde{v}}_{n}^{J} \wedge  d\tilde{F}_{\tau}^{J}\big)d\tau\\
&+h^{3}\int_{0}^{1}\int_{0}^{1}\Big(\bar{\beta}_{\tau}\big({W}^{J}\big)
\gamma_{\sigma}\big({W}^{J}\big) -\frac{1}{2}{W}^{J}
\bar{\beta}_{\tau}\big({W}^{J}\big)
\beta_{\sigma}\big({W}^{J}\big)\Big)
\big(d\bar{\tilde{F}}_{\tau}^{J} \wedge d\tilde{F}_{\sigma}^{J}\big)d\tau d\sigma\bigg),\\
\end{aligned}\label{21}
\end{equation}}where we have used the fact  $e^{{W}^{J}} -{W}^{J}\varphi_{1}\big({W}^{J}\big)=I$
which is given by the definition of $\varphi$-functions.
{Further, it follows  from the first formula of \eqref{a} that
$d\tilde{x}^{J}_{n} =d\tilde{X}^{J}_{\tau}-\tau h \varphi_1\big(\tau{W}^{J}\big)d\tilde{v}^{J}_{n}-h^2
\int_{0}^{1} \alpha_{\tau \sigma}\big({W}^{J}\big)d\tilde{F}_{\sigma}^{J}d\sigma. $
Then $
d\bar{\tilde{x}}^{J}_{n}\wedge d\tilde{F}^{J}_{\tau}
=d\bar{\tilde{X}}^{J}_{\tau}\wedge d\tilde{F}^{J}_{\tau} -
  \tau h \bar{\varphi}_1\big(\tau{W}^{J}\big)d\bar{\tilde{v}}^{J}_{n}
\wedge d\tilde{F}^{J}_{\tau}-h^2
\int_{0}^{1}\big(\bar{{\alpha}}_{\tau \sigma}\big({W}^{J}\big)
d\bar{\tilde{F}}_{\sigma}^{J}\wedge d\tilde{F}^{J}_{\tau}\big)d\sigma. $
Therefore, the formula  \eqref{21} can be rewritten as
\begin{align}
&\sum\limits_{J=1}^{3}d\bar{\tilde{x}}_{n+1}^{J}\wedge d\tilde{v}_{n+1}^{J}-\frac{1}{2}\sum\limits_{J=1}^{3}d\bar{\tilde{x}}_{n+1}^{J}\wedge d\big(\tilde{\Lambda}^{J}\mathrm{i}\tilde{x}_{n+1}^{J}\big)
=\sum\limits_{J=1}^{3} d\bar{\tilde{x}}_{n}^{J}  \wedge
d\tilde{v}_{n}^{J}
- \frac{1}{2}\sum\limits_{J=1}^{3} d\bar{\tilde{x}}_{n}^{J}\wedge  d\big(\tilde{\Lambda}^{J}\mathrm{i}\tilde{x}_{n}^{J}\big)\nonumber \\
&\qquad \qquad \quad
+ h\sum\limits_{J=1}^{3}\int_{0}^{1} \Big( \gamma_{\tau}\big({W}^{J}\big)
-{W}^{J}\beta_{\tau}\big({W}^{J}\big)            \Big)
\Big(d\bar{\tilde{X}}^{J}_{\tau}\wedge d\tilde{F}^{J}_{\tau}\Big)d\tau \label{c} \\
&\qquad \qquad  \quad
+\sum\limits_{J=1}^{3} \Big(he^{{W}^{J}} \bar{\varphi}_{1}\big({W}^{J}\big)
-\frac{1}{2}h {W}^{J}\bar{\varphi}_{1}\big({W}^{J}\big) \varphi_{1}\big({W}^{J}\big)\Big)
d\bar{\tilde{v}}_{n}^{J}  \wedge   d\tilde{v}_{n}^{J} \label{d} \\
&\qquad \qquad \quad
+h^{2}\sum\limits_{J=1}^{3}\int_{0}^{1} \bigg(\bar{\varphi}_{1}\big({W}^{J}\big)    \gamma_{\tau}\big({W}^{J}\big)
- \beta_{\tau}\big({W}^{J}\big)  e^{-{W}^{J}}-{W}^{J} \bar{\varphi}_{1}\big({W}^{J}\big) \beta_{\tau}\big({W}^{J}\big) \nonumber \\
&\qquad \qquad \quad
-\tau \bar{\varphi}_1\big(\tau{W}^{J}\big)\Big(\gamma_{\tau}\big({W}^{J}\big)
  -{W}^{J}\beta_{\tau}\big({W}^{J}\big)\Big)\bigg)
\Big(d\bar{\tilde{v}}_{n}^{J} \wedge  d\tilde{F}_{\tau}^{J}\Big)d\tau\label{e}\\
&\qquad \qquad \quad
+h^{3}\sum\limits_{J=1}^{3}\int_{0}^{1}\int_{0}^{1}\bigg(\bar{\beta}_{\sigma}\big({W}^{J}\big)   \gamma_{\tau}\big({W}^{J}\big)
-\frac{1}{2}{W}^{J} \bar{\beta}_{\sigma}\big({W}^{J}\big)   \beta_{\tau}\big({W}^{J}\big) \nonumber \\
&\qquad \qquad \quad
-\bar{{\alpha}}_{\tau\sigma}\big({W}^{J}\big)\Big(  \gamma_{\tau}\big({W}^{J}\big) -{W}^{J} \beta_{\tau}\big({W}^{J}\big)   \Big)
\bigg)\Big(d\bar{\tilde{F}}_{\sigma}^{J} \wedge d\tilde{F}_{\tau}^{J}\Big)d\tau d\sigma\label{f}.
\end{align}\label{21a}

\textbf{(II) Proof of the results.}
Next, we   prove that all the formulae
\eqref{c}--\eqref{f} are equal to 0 under the conditions \eqref{17}.}

 $\bullet$ \textbf{Prove that the term $\eqref{c}$ is
zero.}
In the light of  the first condition
of \eqref{17}, ${F}(x)=-\nabla_{x} U(x)$ and \eqref{a}, we find $\sum\limits_{J=1}^{3}d\bar{\tilde{X}}^{J}_{\tau}\wedge d\tilde{F}^{J}_{\tau}=\sum\limits_{J=1}^{3} d
X^{J}_{\tau}\wedge dF^{J}_{\tau}$ which further implies
\begin{equation}
\begin{aligned}
&\sum\limits_{J=1}^{3} \Big( \gamma_{\tau}\big({W}^{J}\big)
-{W}^{J}\beta_{\tau}\big({W}^{J}\big)            \Big)
d\bar{\tilde{X}}^{J}_{\tau}\wedge d\tilde{F}^{J}_{\tau}
=d_{\tau}\sum\limits_{J=1}^{3}   d\bar{\tilde{X}}^{J}_{\tau}\wedge
d\tilde{F}^{J}_{\tau} =d_{\tau}\sum\limits_{J=1}^{3}   d X^{J}_{\tau}\wedge
dF^{J}_{\tau}\\
=&-d_{\tau}\sum\limits_{J=1}^{3}  dF^{J}_{\tau} \wedge d X^{J}_{\tau}
= -d_{\tau} \sum\limits_{J=1}^{3} d\bigg(-\frac{\partial{ U} }
{\partial{x} }(X_\tau) \bigg)^{J}\wedge d X_{\tau}^{J}
=d_{\tau} \sum\limits_{J,I=1}^{3}\frac{\partial^{2}{ U(X_\tau)} }
{\partial{x^{J}} \partial{x^{I}}} d X_{\tau}^{I}  \wedge d
X_{\tau}^{J}=0.
\end{aligned}\nonumber
\end{equation}

$\bullet$  \textbf{Prove that the term $\eqref{d}$ is
zero.} Moreover, it can be checked that
\begin{equation}
\begin{aligned}
&\sum\limits_{J=1}^{3}\Big(he^{{W}^{J}} \bar{\varphi}_{1}\big({W}^{J}\big)
-\frac{1}{2}h {W}^{J}\bar{\varphi}_{1}\big({W}^{J}\big) \varphi_{1}\big({W}^{J}\big)\Big)
d\bar{\tilde{v}}_{n}^{J}  \wedge   d\tilde{v}_{n}^{J}\\
=&\Big(he^{{W}^{1}}
\bar{\varphi}_{1}\big({W}^{1}\big) -\frac{1}{2}h
{W}^{1}\bar{\varphi}_{1}\big({W}^{1}\big)
\varphi_{1}\big({W}^{1}\big)\Big)
d\bar{\tilde{v}}_{n}^{1}  \wedge   d\tilde{v}_{n}^{1}\\
&+\Big(he^{{W}^{2}} \bar{\varphi}_{1}\big({W}^{2}\big)
-\frac{1}{2}h {W}^{2}\bar{\varphi}_{1}\big({W}^{2}\big) \varphi_{1}\big({W}^{2}\big)\Big)
d\bar{\tilde{v}}_{n}^{2}  \wedge   d\tilde{v}_{n}^{2}\\
&+\Big(he^{{W}^{3}} \bar{\varphi}_{1}\big({W}^{3}\big)
-\frac{1}{2}h {W}^{3}\bar{\varphi}_{1}\big({W}^{3}\big) \varphi_{1}\big({W}^{3}\big)\Big)
d\bar{\tilde{v}}_{n}^{3}  \wedge   d\tilde{v}_{n}^{3}.\\
\end{aligned}\nonumber
\end{equation}
According to the  property  of $\tilde{v}_{n}$, it yields
\begin{equation}
 d\bar{\tilde{v}}_{n}^{1}  \wedge   d\tilde{v}_{n}^{1}=-d\bar{\tilde{v}}_{n}^{3}  \wedge   d\tilde{v}_{n}^{3}, \  \ \ d\bar{\tilde{v}}_{n}^{2}  \wedge   d\tilde{v}_{n}^{2}=0, \nonumber
\end{equation}
and
\begin{equation}
 he^{{W}^{1}} \bar{\varphi}_{1}\big({W}^{1}\big)
-\frac{1}{2}h
{W}^{1}\bar{\varphi}_{1}\big({W}^{1}\big)
\varphi_{1}\big({W}^{1}\big)
=he^{{W}^{3}}
\bar{\varphi}_{1}\big({W}^{3}\big) -\frac{1}{2}h
{W}^{3}\bar{\varphi}_{1}\big({W}^{3}\big)
\varphi_{1}\big({W}^{3}\big).\nonumber
 \end{equation}
Thus, we obtain
  \begin{equation}
  \sum\limits_{J=1}^{3}\Big(he^{{W}^{J}} \bar{\varphi}_{1}\big({W}^{J}\big)
-\frac{1}{2}h
{W}^{J}\bar{\varphi}_{1}\big({W}^{J}\big)
\varphi_{1}\big({W}^{J}\big)\Big) d\bar{\tilde{v}}_{n}^{J}
\wedge   d\tilde{v}_{n}^{J}=0.\nonumber
\end{equation}

$\bullet$  \textbf{Prove that the terms
\eqref{e}--\eqref{f} are zero.} From the second and third formulae
of \eqref{17}, the last two terms \eqref{e} and \eqref{f} vanish.
\vspace{0.2cm} In the light of the above analysis, it is arrived at
\begin{equation}
\begin{aligned}
\sum\limits_{J=1}^{3}  { d}\bar{\tilde{x}}_{n+1}^{J} \wedge  {
d}\tilde{v}_{n+1}^{J}-\frac{1}{2} \sum\limits_{J=1}^{3}  {
d}\bar{\tilde{x}}_{n+1}^{J} \wedge  { d}
\big(\tilde{\Lambda}^{J}\mathrm{i}\tilde{x}_{n+1}^{J}\big)=
\sum\limits_{J=1}^{3}  { d}\bar{\tilde{x}}_{n}^{J} \wedge {
d}\tilde{v}_{n}^{J}-\frac{1}{2}  \sum\limits_{J=1}^{3}  {
d}\bar{\tilde{x}}_{n}^{J} \wedge { d}
\big(\tilde{\Lambda}^{J}\mathrm{i}\tilde{x}_{n}^{J}\big).
\end{aligned}\nonumber
\end{equation}
 Therefore the  method with the coefficients satisfying
\eqref{17} is symplectic. The proof is complete.
%\hfill $\blacksquare$
\end{proof}

\section{{Practical {symplectic methods} and numerical tests}}\label{sec:prac meth}
We are now ready to consider the construction of practical {symplectic methods}. In this section, we {shall} propose  second-order and fourth-order continuous-stage symplectic adapted exponential methods  based on the symplectic conditions \eqref{17}  and {then rigorously study their error bounds and implementations. Finally, one numerical test is given to show the performance of the obtained methods.}
\subsection{Construction of the practical methods}\label{sec3.1}
In order to construct  practical continuous-stage symplectic adapted exponential methods, we first present  a class of
 method \eqref{TRM} whose  coefficients satisfy the symplectic conditions \eqref{17}. Then
 symplectic methods {of} the form \eqref{CSAEI} are obtained since the transformed method \eqref{TRM} shares the same symplecticity with   \eqref{CSAEI}.
\begin{mytheo}\label{thm: csxmthod}% For numerical integration of system \eqref{charged-particle
%sts-cons}, i
If the coefficients of  \eqref{TRM} satisfy
 \begin{align}
 &\beta_{\tau}( W )=(1-\tau)\varphi_{1}\big((1-\tau)W\big),\qquad \gamma_{\tau}( W )=\varphi_{0}\big((1-\tau)W\big),\label{coebc}\\
 & \bar{\alpha}_{\tau \sigma}(W)-\alpha_{\sigma\tau}(W)
=(\tau-\sigma)\varphi_{1}\big(-(\tau-\sigma) W\big),\label{coea}
\end{align}
  then the   method \eqref{TRM} is symplectic, {i.e.,} its coefficients satisfy the symplectic conditions \eqref{17}.
\end{mytheo}

\begin{proof}
 By inserting    \eqref{coebc} into the first  formula of \eqref{17}, we have
\begin{equation*}
 \begin{aligned}
  &\gamma_{\tau}( W )-W\beta_{\tau}( W )
 =\varphi_{0}\big((1-\tau)W\big)-W(1-\tau)\varphi_{1}\big((1-\tau)W\big)\\
=&e^{(1-\tau)W}-W(1-\tau)\big(e^{(1-\tau)W}-I\big)/\big((1-\tau)W\big)=I.
\end{aligned}
\end{equation*}
It follows from the second formula of \eqref{17} that
\begin{equation*}
 \begin{aligned}
  &\gamma_{\tau}( W )\big(\bar{\varphi}_{1}(W)-\tau\bar{\varphi}_{1}(\tau W)\big)-
\beta_{\tau}( W )\big(e^{-W}+W\bar{\varphi}_{1}(W)-\tau W\bar{\varphi}_{1}(\tau W)\big)\\
 =&\varphi_{0}\big((1-\tau)W)\big)\big(\varphi_{1}(-W)-\tau\varphi_{1}(-\tau W)\big)-
(1-\tau)\varphi_{1}\big((1-\tau)W\big)\\
&\big(e^{-W}+W\varphi_{1}(-W)-\tau W\varphi_{1}(-\tau W)\big)\\
%=&e^{(1-\tau)M}\big(\frac{e^{-M}-I}{-M}-\tau \frac{e^{-\tau M}-I}{-\tau M} \big)-
%(1-\tau)\varphi_{1}((1-\tau)M)\big(e^{-M}+M \frac{e^{-M}-I}{-M}-\tau M\frac{e^{-\tau M}-I}{-\tau M}\big)\\
%=& e^{(1-\tau)hM}(I-e^{-hM}+e^{-\tau hM}-I)/(hM)-(e^{(1-\tau)hM}-I)(e^{-hM}-e^{-hM}+I+e^{-\tau hM}-I)/(hM)\\
=& e^{(1-\tau)W}\big(e^{-\tau W}-e^{-W}\big)/W-e^{-\tau W}\big(e^{(1-\tau)W}-I\big)/W=0.\\
%=& \big(e^{(1-\tau)M-\tau M}-e^{(1-\tau)M- M}-e^{(1-\tau)M-\tau M}+e^{-\tau M}\big)/M\\
\end{aligned}
\end{equation*}
Furthermore, we substitute  the coefficients \eqref{coebc} into the third formula of symplectic conditions \eqref{17} and compute the left hand side to get
\begin{equation*}
 \begin{aligned}
 &(1-\sigma)\varphi_{1}\big(-(1-\sigma) W\big)\varphi_{0}\big((1-\tau) W\big)
 -\bar{\alpha}_{\tau \sigma}(W)\\
 &-W(1-\sigma)\varphi_{1}\big(-(1-\sigma) W\big)(1-\tau)\varphi_{1}\big((1-\tau) W\big)/2\\
%=&e^{(1-\tau) W}\big(-e^{-(1-\sigma) W}+I\big)/W-\bar{\alpha}_{\tau \sigma}(W)\\
%&-\big(-e^{-(1-\sigma) W}+I\big)\big(e^{(1-\tau) W}-I\big)/(2W)\\
=&\big(-e^{(\sigma-\tau)W}+e^{(1-\tau)W}-e^{-(1-\sigma)W}+I\big)/(2W)-\bar{\alpha}_{\tau \sigma}(W).
\end{aligned}
\end{equation*}
Similarly, the right hand side of the third  symplectic condition {in \eqref{17}} can be simplified as
\begin{equation*}
 \begin{aligned}
 &(1-\tau)\varphi_{1}\big((1-\tau) W\big)\varphi_{0}\big(-(1-\sigma) W\big)-\alpha_{\sigma\tau}(W)\\
 &+W(1-\tau)\varphi_{1}\big((1-\tau) W\big)(1-\sigma)\varphi_{1}\big(-(1-\sigma) W\big)/2\\
% =&e^{-(1-\sigma) W}\big(e^{(1-\tau) W}-I\big)/W-\alpha_{\sigma\tau}(W)\\
% &+\big(e^{(1-\tau) W}-I\big)\big(-e^{(1-\sigma) W}+I\big)/(2W)\\
=&\big(e^{(\sigma-\tau)W}-e^{-(1-\sigma)W}+e^{(1-\tau)W}-I\big)/(2W)-\alpha_{\sigma\tau}(W).
\end{aligned}
\end{equation*}
%\begin{equation*}
% \begin{aligned}
% &(1-\tau)\varphi_{1}\big((1-\tau) hM\big)\varphi_{0}\big(-(1-\sigma) hM\big)
% +hM(1-\tau)\varphi_{1}\big((1-\tau) hM\big)(1-\sigma)\varphi_{1}\big(-(1-\sigma) hM\big)/2\\
% &-\alpha_{\sigma\tau}(hM)\\
%=&e^{-(1-\sigma) hM}\big(e^{(1-\tau) hM}-I\big)/(hM)+\big(e^{(1-\tau) hM}-I\big)\big(-e^{(1-\sigma) hM}+I\big)/(2hM)-\alpha_{\sigma\tau}(hM)\\
%=&\big(e^{(\sigma-\tau)hM}-e^{-(1-\sigma)hM}+e^{(1-\tau)hM}-I\big)/(2hM)-\alpha_{\sigma\tau}(hM).
%\end{aligned}
%\end{equation*}
Combining the above two results, we get
%$$\bar{\alpha}_{\tau \sigma}(M)-\alpha_{\sigma\tau}(M)=(-e^{(\sigma-\tau)M}+I)/M=(\tau-\sigma)\varphi_{1}(-(\tau-\sigma) M).$$  \hfill $\blacksquare$
\begin{equation*}
 \begin{aligned}
 \bar{\alpha}_{\tau \sigma}(W)-\alpha_{\sigma\tau}(W)
%=&\frac{-e^{(\sigma-\tau)M}+e^{(1-\tau)M}-e^{-(1-\sigma)M}+I-e^{(\sigma-\tau)M}+e^{-(1-\sigma)M}-e^{(1-\tau)M}+I}{2M}
%=&\frac{-2e^{(\sigma-\tau)M}+2I}{2M}
=\big(-e^{(\sigma-\tau)W}+I\big)/W
%=(\tau-\sigma)(e^{-(\tau-\sigma)M}-I)/(-(\tau-\sigma)M)
=(\tau-\sigma)\varphi_{1}\big(-(\tau-\sigma) W\big).
\end{aligned}
\end{equation*}
From the above analysis, {it is known} that the coefficients  \eqref{coebc} and  \eqref{coea} satisfy all the symplectic conditions \eqref{17}, which completes the proof of this theorem.  %\hfill  $\blacksquare$
\end{proof}

As stated {in the beginning of this section}, once  a
  symplectic method determined by \eqref{TRM} with the coefficients \eqref{coebc} and  \eqref{coea} is given, a
 symplectic method  \eqref{CSAEI} is derived immediately {whose coefficients are obtained by} replacing $W$ with $hM$ {in \eqref{coebc} and  \eqref{coea}}.
In what follows, we present two practical continuous-stage symplectic adapted exponential  methods  \eqref{CSAEI} up to order four.
Actually,  the result of $\beta_{\tau}( hM )$ and $\gamma_{\tau}( hM )$ is described by \eqref{coebc} of Theorem \ref{thm: csxmthod}. Therefore, we  only need to make the choice of the coefficient
$\alpha_{\tau\sigma}(hM)$. The first case is given by the following  algorithm which {will be proved to be of order two.}

\begin{algo}\label{algo3.2}
\textbf{(Second-order  method)} Define a practical continuous-stage symplectic  adapted exponential algorithm \eqref{CSAEI} with the following coefficients
\begin{equation}\label{coe2}
 \begin{aligned}
 &\alpha_{\tau\sigma}(hM)=\frac{\tau-\sigma}{2}\varphi_{1}((\tau-\sigma) hM),\ \
    \beta_{\tau}( hM )=(1-\tau)\varphi_{1}((1-\tau)hM),\ \ \gamma_{\tau}( hM )=\varphi_{0}((1-\tau)hM).
\end{aligned}
\end{equation}
It is easily to  check that the coefficient $\alpha_{\tau\sigma} $  satisfies \eqref{coea}, which means that
the  method with the coefficients \eqref{coe2} is  {symplectic.}
\end{algo}

In order to improve the  accuracy, we consider the following choice of $\alpha_{\tau\sigma}(hM)$ which leads to a
 fourth-order scheme.
\begin{algo}\label{algo3.3}
\textbf{(Fourth-order  method)} If we take the following coefficients
\begin{equation}
 \begin{aligned}\label{coe4}
 &\alpha_{\tau\sigma}(hM)=\Big(\frac{1}{6}+\frac{\tau-\sigma}{2}\Big)\varphi_{1}((\tau-\sigma) hM),\\
&  \beta_{\tau}( hM )=(1-\tau)\varphi_{1}((1-\tau)hM),\quad\gamma_{\tau}( hM )=\varphi_{0}((1-\tau)hM),
\end{aligned}
\end{equation}
then we get a  continuous-stage symplectic  adapted exponential  method.  Similarly, we can prove that the coefficients \eqref{coe4}  satisfy all the symplectic conditions of  {Theorem \ref{thm: csxmthod}.}
\end{algo}

% \subsection{Second-order  method}
%Let us start with a   method whose coefficients have the
%following form
%\begin{equation}\label{coe}
% \begin{aligned}
% &\alpha_{\tau\sigma}(K)=\frac{\tau-\sigma}{2}\varphi_{1}((\tau-\sigma) M),\ \ \beta_{\tau}( M )=(1-\tau)\varphi_{1}((1-\tau)M),\ \  \gamma_{\tau}( M )=\varphi_{0}((1-\tau)M).
%\end{aligned}
%\end{equation}
%Moreover, it can be checked  the coefficients  $ \alpha_{\tau \sigma}( M ) $  \eqref{coe} that
%\begin{equation*}
% \begin{aligned}
% \bar{\alpha}_{\tau \sigma}(M)-\alpha_{\sigma\tau}(M)
%=\frac{\tau-\sigma}{2}\varphi_{1}(-(\tau-\sigma) M)-\frac{\sigma-\tau}{2}\varphi_{1}((\sigma-\tau) M)
%=(\tau-\sigma)\varphi_{1}(-(\tau-\sigma) M).
%\end{aligned}
%\end{equation*}
%Thus the method with the coefficients \eqref{coe} satisfying  \eqref{thm: symp} is symplectic.  It can be checked that the coefficients of  this method satisfy all the second-order conditions   \eqref{order} with $r=2$.
%
%\subsection{Fourth-order  method}
%We now consider a   method with the following coefficients:
%\begin{equation}
% \begin{aligned}\label{coe2}
% &\alpha_{\tau\sigma}(K)=(\frac{1}{6}+\frac{\tau-\sigma}{2})\varphi_{1}((\tau-\sigma) M),\ \ \beta_{\tau}( M )=(1-\tau)\varphi_{1}((1-\tau)M),\ \  \gamma_{\tau}( M )=\varphi_{0}((1-\tau)M).
%\end{aligned}
%\end{equation}
%According to the analysis in the section 6.1, it can be checked easily that the coefficients \eqref{coe2} of  this  method satisfy all the symplecticity conditions \eqref{thm: symp} and fourth-order conditions   \eqref{order} with $r=4$.

{
 \subsection{Convergence}
 In this section, we shall study the convergence of the proposed two {algorithms}.
\begin{mytheo}\label{order condition} \textbf{(Convergence)}
It is assumed that $ F $ is locally Lipschitz-continuous with the Lipschitz constant $L$.
There {exist  constants $h_0>0$ and $\widehat{C}>0$} independent of $\eps$, such that  if the stepsize $h$ satisfies $ h \leq h_0$   and $ h \leq \widehat{C} \epsilon$,  the global errors are estimated as
\begin{equation*}
\begin{aligned}&Algorithm\ \textmd{\ref{algo3.2}}:\ \ \norm{x(t_{n})-x_n}\leq C  h^2,\ \ \ \ \ \ \
\norm{v(t_{n})-v_n}\leq C h^2/\epsilon,\\
&Algorithm\ \textmd{\ref{algo3.3}}:\ \ \norm{x(t_{n})-x_n}\leq C  h^4/\eps^2,\ \ \
\norm{v(t_{n})-v_n}\leq C h^4/\epsilon^3,
\end{aligned}
\end{equation*}
 for $nh\leq T$, where
$C>0$ is a generic constant independent of $\epsilon$ or $h$  or $n$  but depends on $\widehat{C}, L, T$ and  $\norm{ \frac{d^s}{dx^s}F(x) }$ with $s=1,2$ for Algorithm \ref{algo3.2}  and $s=1,2,3,4$ for Algorithm \ref{algo3.3}.
%If the $C_{\tau}( hK )$ are chosen
%as $C_{\tau}( hK )=\tau \varphi_{1}({\tau}^2  hK ) $, then $ m=r $ in the result.
\end{mytheo}

\begin{proof}
%We divide the proof into two parts.\\
$\bullet$ \textbf{Local errors.}
Local errors of the {method \eqref{CSAEI} are defined}
by inserting the exact solution
\eqref{VOC} into \eqref{CSAEI}, which leads to
\begin{equation*}
\begin{aligned}\label{PCSEEP1}
&x(t_{n}+\tau h)=x(t_{n})+\tau h\varphi_{1}(\tau hM )v(t_{n})+h^{2}\int_{0}^{1}
\alpha_{\tau \sigma}( hM )
\hat{F}(t_{n}+\sigma h)d \sigma +\triangle^x_{\tau},\\
&x(t_{n+1})=x(t_{n})+h\varphi_{1}( hM )v(t_{n})+h^{2}\int_{0}^{1}
\beta_{\tau}( hM )
\hat{F}(t_{n}+\tau h)d \tau+\delta^x_{n+1},\\
&v(t_{n+1})=\varphi_{0}( hM )v(t_{n})+h\int_{0}^{1}\gamma_{\tau}( hM )
\hat{F}(t_{n}+\tau h)d\tau+\delta^v_{n+1},
\end{aligned}
\end{equation*}
where $ \hat{F}(t):= F(x(t))$ and $ \triangle^x_{\tau}$, $ \delta^x_{n+1}$, $\delta^v_{n+1} $ are the discrepancies.

 Combining with the variation-of-constants formula  \eqref{VOC} and using Taylor series, we obtain
\begin{equation*}
\begin{aligned}
\triangle^x_{\tau}
&= {\tau}^2 h^2\int_{0}^{1}(1-z)\varphi_{1}(\tau(1-z) hM )\hat{F}(t_{n}+h \tau z)dz
-h^{2}\int_{0}^{1}\alpha_{\tau\sigma}( hM )\hat{F}(t_{n}+\sigma h)d\sigma\\
&=\sum\limits_{j=0}^{r-3}h^{j+2}\bigg({\tau}^{j+2}\int_{0}^{1}(1-\sigma)\varphi_{1}
(\tau(1-\sigma) hM )\frac{{\sigma}^{j}}{j!}d\sigma
-\int_{0}^{1}\alpha_{\tau\sigma}( hM )\frac{{\sigma}^{j}}{j!}d \sigma\bigg)\hat{F}^{j}(t_{n})+\mathcal{O}(h^{r}/\epsilon^{r-3})\\
&=\sum\limits_{j=0}^{r-3}h^{j+2}\bigg({\tau}^{j+2}\varphi_{j+2}(\tau hM )
-\int_{0}^{1}\alpha_{\tau\sigma}( hM )\frac{{\sigma}^{j}}{j!}d \sigma\bigg)\hat{F}^{j}(t_{n})+\mathcal{O}(h^{r}/\epsilon^{r-3}),
%&=\sum\limits_{j=0}^{r-3}\mathcal{O}( h^{r}\gamma_j/\epsilon^{j-1}) + \mathcal{O}(h^{r}/\epsilon^{r-3}),
\end{aligned}
\end{equation*}
and similarly
\begin{equation*}
\begin{aligned}
\rho_{n+1}=&\sum\limits_{j=0}^{r-2}h^{j+2}\bigg(\varphi_{j+2}( hM )
-\int_{0}^{1}\beta_{\tau}( hM )\frac{{\tau}^{j}}{j!}d \tau\bigg)\hat{F}^{j}(t_{n})
+{\mathcal{O}(h^{r+1}/\epsilon^{r-2})},\\
{\rho}^{'}_{n+1}=&\sum\limits_{j=0}^{r-1}h^{j+1}\bigg(\varphi_{j+1}( hM )
-\int_{0}^{1}\gamma_{\tau}( hM )\frac{{\tau}^{j}}{j!}d \tau\bigg)\hat{F}^{j}(t_{n})
+{\mathcal{O}(h^{r+1}/\epsilon^{r-1})}.
\end{aligned}
\end{equation*}
where $ \hat{F}^{j}(t) $ denotes the $j$th order derivative of $ F(x(t)) $ with respect to $ t $ and $r$ is a positive integer which has different value for different algorithm.

Based on the coefficient functions chosen in Algorithms \ref{algo3.2}-\ref{algo3.3}, we have the following results for the  discrepancies stated above
\begin{equation}\label{local error}
\begin{aligned}&Algorithm\ \textmd{\ref{algo3.2}}:\ \ r=2,\ \ \norm{\triangle^x_{\tau}}\leq Ch^{2},
\ \ \ \  \
\norm{\delta^x_{n+1}}\leq Ch^{3}, \ \  \ \ \ \ \norm{\delta^v_{n+1}}\leq Ch^{3}/\eps,\\
&Algorithm\ \textmd{\ref{algo3.3}}:\ \ r=4,\ \ \norm{\triangle^x_{\tau}}\leq Ch^{4}/\eps,
\ \
\norm{\delta^x_{n+1}}\leq Ch^{5}/\eps^2, \ \   \norm{\delta^v_{n+1}}\leq Ch^{5}/\eps^3.
\end{aligned}
\end{equation}

$\bullet$ \textbf{Global errors.}  Denote the global errors of \eqref{CSAEI} by
\begin{equation*}
e_{n}^{x}=x(t_{n})-x_{n},\ \ \ e_{n}^{v}=v(t_{n})-v_{n},
\ \ \ E_{\tau}=x(t_{n}+\tau h)-X_{\tau},
\end{equation*}
and the error system is
\begin{equation}
\begin{aligned}\label{error}
&E_{\tau}=e_{n}^{x}+\tau h \varphi_1(\tau hM )e_{n}^{v}
+h^2\int_{0}^{1}\alpha_{\tau\sigma}( hM )\big(F(x(t_{n}+\sigma h))-F(X_{\sigma})\big)d\sigma+\Delta^x_{\tau},\\
&e_{n+1}^{x}=e_{n}^{x}+h \varphi_{1}( hM ) e_{n}^{v}
+h^{2}\int_{0}^{1}\beta_{\tau}( hM )\big(F(x(t_{n}+\tau h))-F(X_{\tau})\big)d\tau+\delta^x_{n+1},\\
&e_{n+1}^{v}=\varphi_{0}( hM ) e_{n}^{v}
+h\int_{0}^{1}\gamma_{\tau}( hM )\big(F(x(t_{n}+\tau h))-F(X_{\tau})\big)d\tau+\delta^v_{n+1},
\end{aligned}
\end{equation}
where the initial conditions are $ e_{0}^{x}=0 $, $ e_{0}^{v}=0 $.
With the uniform bound of the coefficients,   the first equation of \eqref{error} is bounded by
\begin{equation*}
\norm{E_{\tau}} \leq  \norm{e_{n }^{x}}+ \tau h \norm{e_{n }^{v}}+h^2 C L\norm{E_{\tau}}_{c}+{\norm{\Delta^x_{\tau}}_{c}},
\end{equation*}
which yields
\begin{equation*}
\norm{E_{\tau}}_{c} \leq  \norm{e_{n }^{x}}+h \norm{e_{n }^{v}}+h^2 {C} L \norm{E_{\tau}}_{c}+\norm{\Delta^x_{\tau}}_{c},
\end{equation*}
where  $\|\cdot\|_{c}$ denotes the maximum norm
$
\norm{E_{\tau}}_{c}=\max \limits_{\tau\in[0,1]} \norm{E_{\tau}}
$
for a continuous  function $E_{\tau}$ on [0,1].
%Here  the constant $ \widetilde{C}$ is referred to the bound of the coefficients of the integrator and the constant $C_{4}$ depends on the constant symbolized by $\mathcal{O} $ in the first formula of \eqref{error}.
Under the condition that
$ h \leq \sqrt{\frac{1}{2 C L}}$ is satisfied, we have
\begin{equation*}\label{con6}
\norm{E_{\tau}}_{c} \leq  2\big(\norm{e_{n }^{x}}+h \norm{e_{n }^{v}}\big)+2\norm{\Delta_{\tau}^{x}}_{c}.
\end{equation*}
From this result and the last two  equations of \eqref{error}, it is deduced that
  \begin{equation*}
\begin{aligned}
\norm{e_{n+1}^{x}}
&\leq \norm{e_{n}^{x}}+h\norm{e_{n}^{v}}+Ch^2\big(\norm{e_{n}^{x}}+h\norm{e_{n}^{v}}+\norm{\Delta_{\tau}^{x}}_{c}\big)
+\norm{\delta^x_{n+1}},\\
\norm{e_{n+1}^{v}}
&\leq \norm{e_{n}^{v}}+Ch\big(\norm{e_{n}^{x}}+h\norm{e_{n}^{v}}+\norm{\Delta_{\tau}^{x}}_{c}\big)
+\norm{\delta^v_{n+1}}.
 \end{aligned}
  \end{equation*}
These {results, the local errors derived in  \eqref{local error}} and Gronwall inequality immediately lead to the statement of this theorem, which completes the proof.
\end{proof}}

\subsection{{Implementation issues}}\label{sec:four}
%It is known that the explicit method  is widely used for its simpleness and little memory and it is  more attractive than implicit methods for high efficiency in the numerical computations. In this section, we are devoted to  the construction of
%practical explicit continuous-stage symplectic adapted exponential methods by using some numerical quadrature formulae for the
%implicit algorithms: Algorithm \ref{algo3.2} and Algorithm \ref{algo3.3}.

It is noted that   Algorithms \ref{algo3.2}--\ref{algo3.3} fail to be practical unless the integrals appearing in \eqref{CSAEI} are computed exactly or by using some numerical quadrature formulae. For most cases, those integrals cannot be solved exactly and usually quadrature formulae are needed in practical computations. In this section, we pay attention to this point and discuss the implementations of the obtained algorithms.

%\subsubsection{The implementation of Algorithm \ref{algo3.2}}\label{sec:algo1}
%Denote the coefficients of one-stage diagonal implicit RK methods
% by a Butcher tableau:
%\[%
%\begin{tabular}
%[c]{l}%
%\begin{tabular}
%[c]{c|c}%
%$c_{1}$ & $a_{11}$\\\hline
%& $\raisebox{-1.3ex}[0.5pt]{$ b_{1}$}$%
%\end{tabular}
%\end{tabular}%
%\]
%In the light of the symplecticity conditions, it can be obtained that this
%method is symplectic if
%\begin{equation}\label{sym1}
%b_{1}a_{11}+b_{1}a_{11}=b_{1}^{2}.
%\end{equation}
%On the other hand, on the basis of the order
%conditions,  a  one-stage diagonal implicit RK method is  of order
%two if
%\begin{equation}\label{order1}
%a_{11}=c_{1},\ \ \ \ \ b_{1}=1,  \ \ \ \ \  b_{1}c_{1}=\frac{1}{2}.
%\end{equation}
%By \eqref{sym1} and \eqref{order1}, we get $b_{1}=1,
%a_{11}=\frac{1}{2}, c_{1}=\frac{1}{2}$. This method is implicit
%midpoint rule which is denoted by RK1s2.  Then from \eqref{aeicoe}, the
%AEI method \eqref{AEI}  is  given by
%\begin{equation}\label{1s2}
%\left\{\begin{aligned}
%&X_{1}=x_{n}+ \frac{h}{2}\varphi_1(\frac{1}{2}\frac{h}{\epsilon} \tilde{B}) v_{n},\\
%&x_{n+1}=x_{n}+ h\varphi_1(\frac{h}{\epsilon} \tilde{B}) v_{n}+\frac{h^2}{2}
%\varphi_1(\frac{1}{2}\frac{h}{\epsilon}\tilde{B})F (X_1),\\
%&v_{n+1}=\varphi_0(\frac{h}{\epsilon} \tilde{B})v_{n}+h
%e^{\frac{1}{2}\frac{h}{\epsilon}\tilde{B}}F (X_1).
%\end{aligned}\right.
%\end{equation}

%It can be checked easily that these coefficients of \eqref{1s2}
%satisfy the symplecticity conditions \eqref{17}. Therefore, this is
%an explicit symplectic AEI method of order two and we
%denoted it by AEI1s2.

Let us start with the implementation of Algorithm \ref{algo3.2}.
In the computation, we consider the four-point Gaussian quadrature with  the weights $b_{i}$ and abscissae $c_{i}$
and then derive the  scheme as follows
\begin{equation*}
\begin{aligned}
&X_{c_{i}}=x_{n}+ c_{i} h\varphi_1(c_{i}hM) v_{n}
+h^2\sum_{j=1}^{4}\frac{b_{j}(c_{i}-c_{j})}{2}\varphi_1((c_{i}-c_{j})hM)F (X_{c_{j}}),\ \ i=1,2,\ldots,4,\\
&x_{n+1}=x_{n}+ h\varphi_1(hM) v_{n}+h^2\sum_{i=1}^{4}b_{i}(1-c_{i})\varphi_1((1-c_{i})hM)F (X_{c_{i}}),\\
&v_{n+1}=\varphi_0(hM)v_{n}+h\sum_{i=1}^{4}b_{i}\varphi_0((1-c_{i})hM)F (X_{c_{i}}),
\end{aligned}
\end{equation*}
where
\begin{equation*}
\begin{array}[c]{ll}
&\breve{a}=\frac{2\sqrt{30}}{35}, \qquad \ \ \ \ \   \breve{b}=\frac{\sqrt{30}}{36},\qquad \ \ \ \ \  \ b_{1}=b_{4}=\frac{1/2-\breve{b}}{2},\ \ \ \ \  \  b_{2}=b_{3}=\frac{1/2+\breve{b}}{2},\\
&c_{1}=\frac{1+\sqrt{3/7+\breve{a}}}{2},\ \ \ c_{2}=\frac{1+\sqrt{3/7-\breve{a}}}{2}, \ \ \
 c_{3}=\frac{1-\sqrt{3/7-\breve{a}}}{2}, \ \ \ \ \ \ \  c_{4}=\frac{1-\sqrt{3/7+\breve{a}}}{2}.\\
\end{array}
\end{equation*}
We shall refer to this algorithm by SC1O2.

%\begin{equation*}
%\left\{\begin{aligned}
%&X_{\hat{c}_{i}}=x_{n}+ h\hat{c}_{i}\varphi_1(\hat{c}_{i}hM) v_{n}
%+h^2\sum_{j=1}^{4}\frac{\hat{b}_{j}(\hat{c}_{i}-\hat{c}_{j})}{2}\varphi_1((\hat{c}_{i}-\hat{c}_{j})hM)F (X_{\hat{c}_{j}}),\\
%&x_{n+1}=x_{n}+ h\varphi_1(hM) v_{n}+h^2\sum_{i=1}^{4}\hat{b}_{i}(1-\hat{c}_{i})\varphi_1((1-\hat{c}_{i})hM)F (X_{\hat{c}_{i}}),\\
%&v_{n+1}=\varphi_0(hM)v_{n}+\sum_{i=1}^{4}\hat{b}_{i}\varphi_0((1-\hat{c}_{i})hM)F (X_{\hat{c}_{i}}).
%\end{aligned}\right.
%\end{equation*}
%where
%\begin{equation*}
%\begin{array}[c]{ll}
%&\breve{a}=\frac{2\sqrt{30}}{35}, \ \ \  \breve{b}=\frac{\sqrt{30}}{36},\ \ \ \hat{b}_{1}=\hat{b}_{4}=\frac{1/2-\breve{b}}{2},\ \ \  \hat{b}_{2}=\hat{b}_{3}=\frac{1/2+\breve{b}}{2},\\
%&\hat{c}_{1}=\frac{1+\sqrt{3/7+\breve{a}}}{2},\ \ \ \hat{c}_{2}=\frac{1+\sqrt{3/7-\breve{a}}}{2}, \ \ \
% \hat{c}_{3}=\frac{1-\sqrt{3/7-\breve{a}}}{2}, \ \ \  \hat{c}_{4}=\frac{1-\sqrt{3/7+\breve{a}}}{2}.\\
%\end{array}
%\end{equation*}
%We shall refer to this algorithm by SC1O2.

Obviously, this scheme is implicit and to get an explicit one, we consider the one-point Gaussian quadrature with
 the weight $b_{1}=1$  and abscissa  $c_{1}=\frac{1}{2}$. This yields
\begin{equation*}
\begin{aligned}\label{vv}
&X_{c_{1}}=x_{n}+ \frac{h}{2}\varphi_1(hM/2) v_{n},\ \ x_{n+1}=x_{n}+ h\varphi_1(hM) v_{n}+\frac{h^2}{2}
\varphi_1(hM/2)F (X_{c_{1}}),\\ &v_{n+1}=\varphi_0(hM)v_{n}+h\varphi_{0}(hM/2)F (X_{c_{1}}).
\end{aligned}
\end{equation*}
%\begin{equation}
%\left\{\begin{aligned}\label{vv}
%&X_{c_{1}}=x_{n}+ hc_{1}\varphi_1(c_{1}hM) v_{n}
%=x_{n}+ \frac{h}{2}\varphi_1(\frac{1}{2}hM) v_{n},\\
%&x_{n+1}=x_{n}+ h\varphi_1(hM)v_{n}+h^2\cdot b_{1}(1-c_1)\varphi_1((1-c_1)hM)F (X_{c_{1}})\\
%&\ \ \ \quad=x_{n}+ h\varphi_1(hM) v_{n}+\frac{h^2}{2}
%\varphi_1(\frac{1}{2}hM)F (X_{c_{1}}),\\
%&v_{n+1}=\varphi_0(hM)v_{n}+h\cdot b_{1}
%\varphi_{0}((1-c_{1})hM)F (X_{c_{1}})=
%\varphi_0(hM)v_{n}+h
%\varphi_{0}({\frac{1}{2}hM})F (X_{c_{1}}).
%\end{aligned}\right.
%\end{equation}
From this formulation, it is clear that this is
an explicit scheme and we
denoted it by SC2O2.

%\subsubsection{The implementation of Algorithm \ref{algo3.3}}\label{sec:algo2}
For the implementation of Algorithm \ref{algo3.3}, the same  quadrature as used in SC1O2 is chosen here and the corresponding algorithm is referred as  SC1O4. {To get the explicit scheme, we  consider}  the following scheme
\begin{equation*}
\begin{aligned}
&X_{c_{i}}=x_{n}+ c_ih\varphi_1(c_{i} hM) v_{n}
+h^2 \sum\limits_{j=1}^{i-1}a_{ij}(c_i-c_j)h\varphi_1((c_i-c_j)hM)F (X_j),\ \ i=1,2,3,\\
&x_{n+1}=x_{n}+ h\varphi_1(hM) v_{n}
+h^2\sum\limits_{i=1}^{3}b_i(1-c_i)\varphi_1((1-c_i)hM)F (X_i),\\
&v_{n+1}=\varphi_0(hM)v_{n}
+h\sum\limits_{i=1}^{3}b_i\varphi_0((1-c_i)hM)F (X_i),
\end{aligned}
\end{equation*}
where
\begin{equation*}
\begin{array}[c]{ll}
&a_{21}=\frac{4+2\sqrt[3]{2}+\sqrt[3]{4}}{6}, \qquad   a_{31}=b_{1}=b_3=a_{21},\qquad   a_{32}=\frac{-1-2\sqrt[3]{2}-\sqrt[3]{4}}{3},\\
&b_{2}=\frac{-1-2\sqrt[3]{2}-\sqrt[3]{4}}{3},\ \ \ \ \ \ c_{1}=b_{1}/2, \ \ \ c_{2}=\frac{1}{2}, \qquad c_{3}=1-c_1.\\
\end{array}
\end{equation*}
We shall call this  method as SC2O4.

\subsection{Numerical test}\label{sec: experiment}
One numerical test is shown in this section to test the efficiency of the new methods
compared with some existing methods in the literature.
 The   methods  for comparison  are chosen as follows:
\begin{itemize}
 \item BORIS: the Boris method of order two presented in \cite{Boris};
 % \item EIO2: the exponential  method of order two presented in \cite{M.Hochbruck2017};
       \item RKO2: a  symplectic Runge-Kutta method of order two (implicit
midpoint rule) presented in \cite{Hairer2002};
  \item SC1O2: the  continuous-stage symplectic adapted exponential  method of order two derived in  Section \ref{sec:four};
  \item SC2O2: the explicit continuous-stage symplectic adapted exponential  method  of order two derived in  Section \ref{sec:four};
   \item RKO4: a  symplectic Runge-Kutta method of order four presented in
\cite{Sanz-Serna91};
  \item SC1O4: the  continuous-stage symplectic adapted exponential  method of order four derived in  Section \ref{sec:four};
   \item SC2O4: the explicit continuous-stage symplectic adapted exponential  method of order four derived in  Section \ref{sec:four}.
   \end{itemize}
For implicit methods, we choose fixed-point iteration and set $ 10^{-16} $ as the
error tolerance  and 5  as the maximum number of  each iteration.
To test the performance of all the methods, we compute the global errors: $error:=\frac{\norm{x_{n}-x(t_n)}}{\norm{x(t_n)}}+\frac{\norm{v_{n}-v(t_n)}}{\norm{v(t_n)}}$ and the energy  error
$e_{H}:=\frac{|H(x_{n},v_n)-H(x_0,v_0)|}{|H(x_0,v_0)|}$  in {the} numerical experiment.

\noindent\vskip3mm \noindent\textbf{Problem 1. (Homogeneous magnetic field)}
This problem is devoted to the charged-particle system  \eqref{charged-particle sts-cons} of \cite{Lubich2017} with an additional factor $1/\epsilon$ and a homogeneous magnetic field.
We take the scalar potential $U(x)=\frac{1}{100\sqrt{x_{1}^2+x_{2}^2}}$, the homogeneous magnetic field
$\frac{1}{\epsilon}B=\frac{1}{\epsilon}(0,0,1)^{\intercal}$ and the initial conditions as $x(0)=(0,0.2,0.1)^{\intercal}, v(0)=(0.09, 0.05, 0.2)^{\intercal}.$

\textbf{Order behaviour.} This system is integrated on $[0,1]$ with different   $\epsilon$  and step sizes $h$ to
show  the global {errors $error$} in Fig. \ref{fig:problem11}.  According to the result, it can be seen   that  the global error lines of  SC1O2 and SC2O2 are  nearly parallel to the line  with slope 2, which verifies  that these two methods  have second-order  accuracy.
Similarly, the fourth-order  accuracy can be observed for SC1O4 and SC2O4. Moreover, we can see that the global errors of our methods are smaller than the other three methods especially with  small $\epsilon$. This  demonstrates that our methods  have  better accuracy than others when solving CPD in a strong magnetic field.

\textbf{Uniform errors.} To display the influence of $\eps$ on the global errors,
we show  the error $error2:=\frac{\norm{x_{n}-x(t_n)}}{\norm{x(t_n)}}+\frac{\eps \norm{v_{n}-v(t_n)}}{\norm{v(t_n)}}$ for the second order methods in Fig. \ref{fig:problem11new1}  and $error4:=\frac{\eps^2\norm{x_{n}-x(t_n)}}{\norm{x(t_n)}}+\frac{\eps^3\norm{v_{n}-v(t_n)}}{\norm{v(t_n)}}$ for the fourth order ones in Fig. \ref{fig:problem11new2}. It can be observed from the results that our methods demonstrate the convergence stated in Theorem \ref{order condition} and behave much better than the others.

\textbf{Energy conservation.} Then we solve the problem
with    $h=\frac{1}{100}$ on the interval $ [0,1000]$.  Fig. \ref{fig:problem12} displays the results of energy conservation {$err_H$}.
It can be observed that for the normal  magnetic field, all the methods have good energy conservation over long times and they behave similarly. For the case that $\eps$ is small, the energy error of those  two Runge-Kutta methods (RKO2 and RKO4) increases slightly when $t$ goes large while the other methods have a good performance. Moreover, our methods display better accuracy in the energy preservation than the Boris method.

\textbf{Long-time behavior.}  Finally,  the problem is solved on $[0, 1000]$  with
$h=\frac{1}{2}$ and $\eps=0.1$. Fig. \ref{fig:problem11new3} presents the trajectory  for the second order methods (BORIS, RKO2, SC1O2 and SC2O2)  in [x y z] space and the results show that  the methods SC1O2 and SC2O2 perform uniformly better  than the Boris  and  RKO2 methods.
%It can  be observed from the numerical results that  our methods  are more stable than  the others two methods.
 Meanwhile,  we have noticed that the RKO2 method performs worse than expected. The reason
is that   for  the RKO2 method when solving CPD  with  small $\epsilon$, there is a
very strict requirement of the stepsize and the method usually does
not behave well.

\begin{figure}[t!]
\centering\tabcolsep=0.4mm
\begin{tabular}
[c]{ccc}%
\includegraphics[width=4.7cm,height=4.45cm]{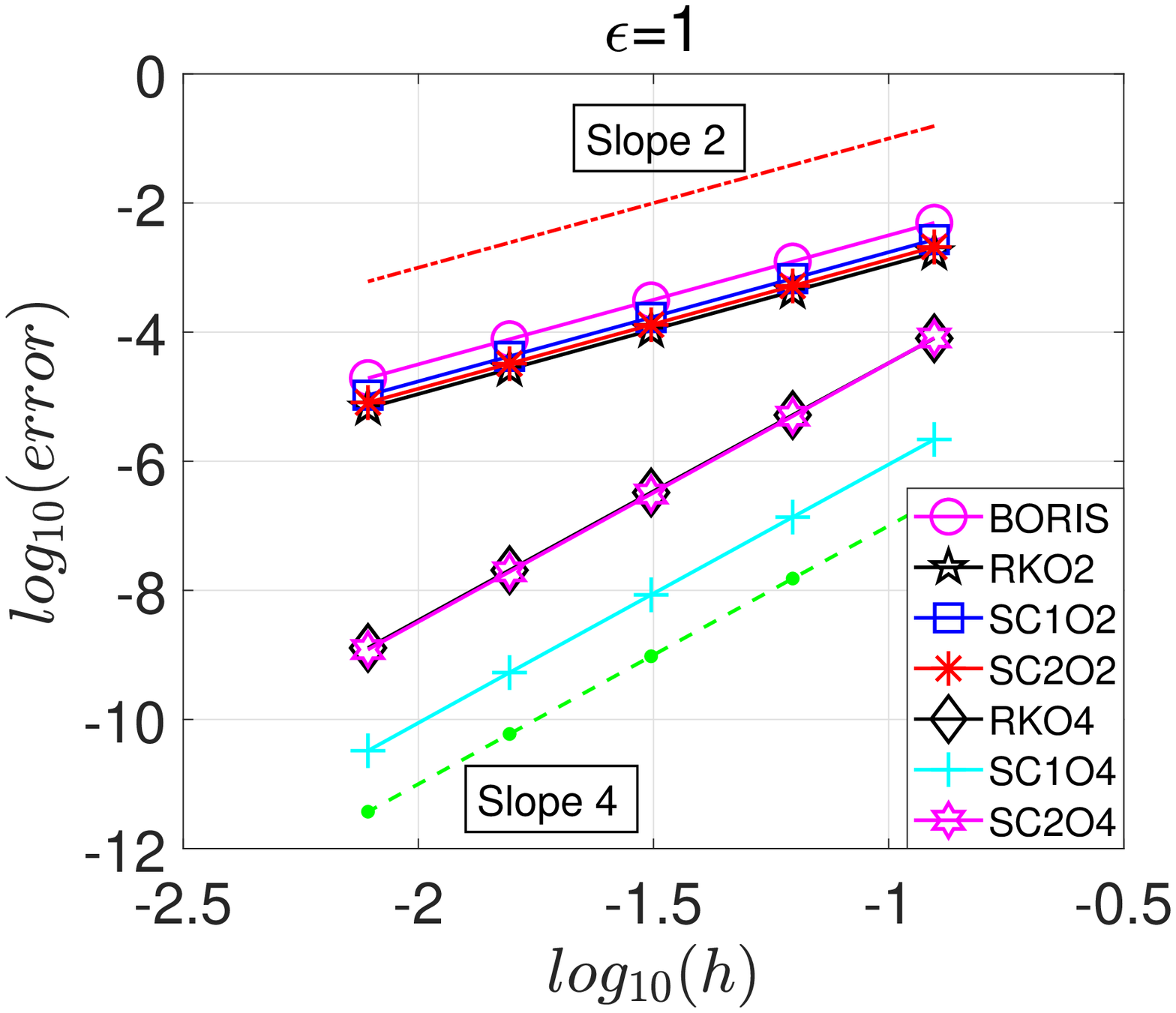} & \includegraphics[width=4.7cm,height=4.45cm]{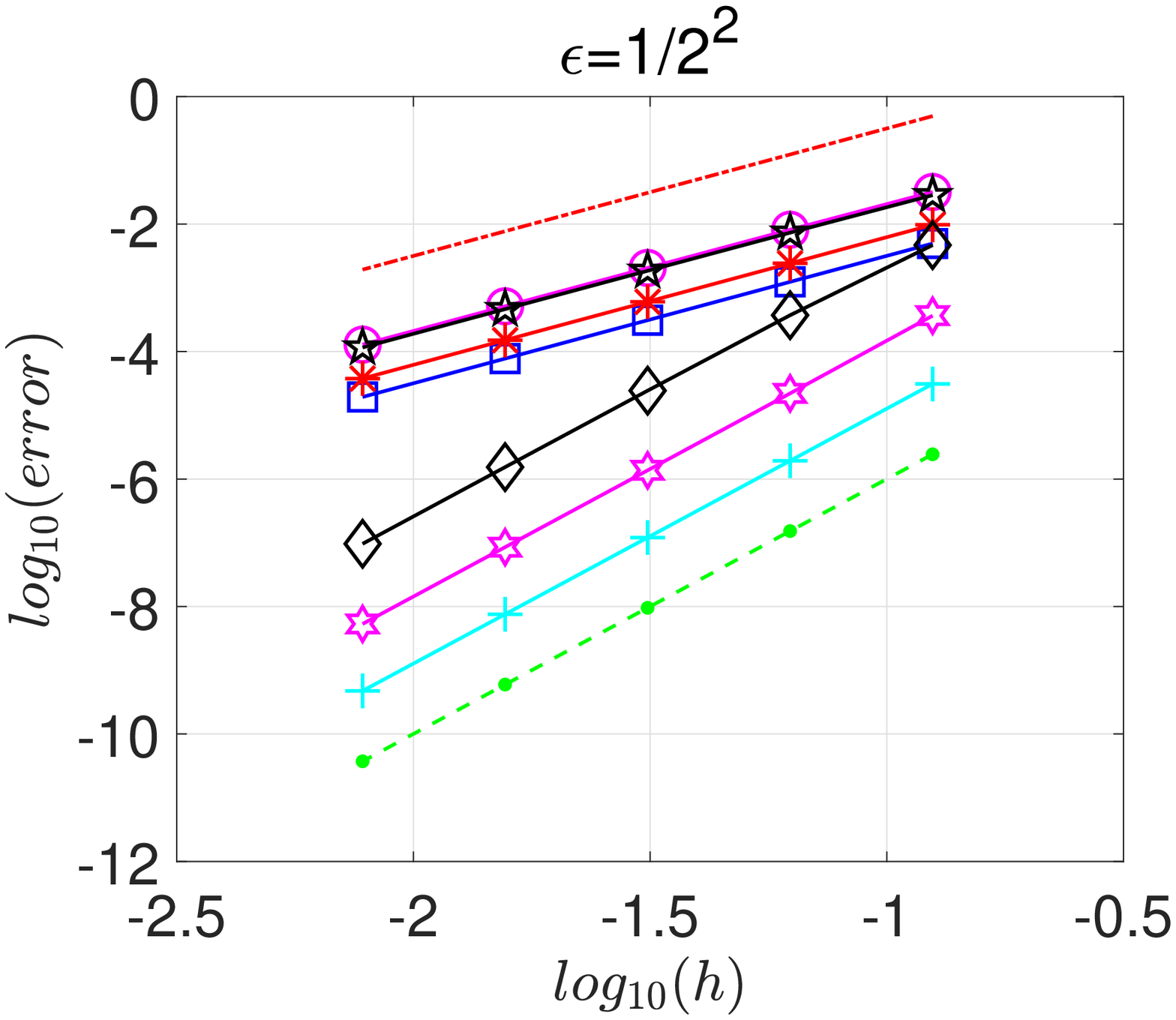}&\includegraphics[width=4.7cm,height=4.45cm]{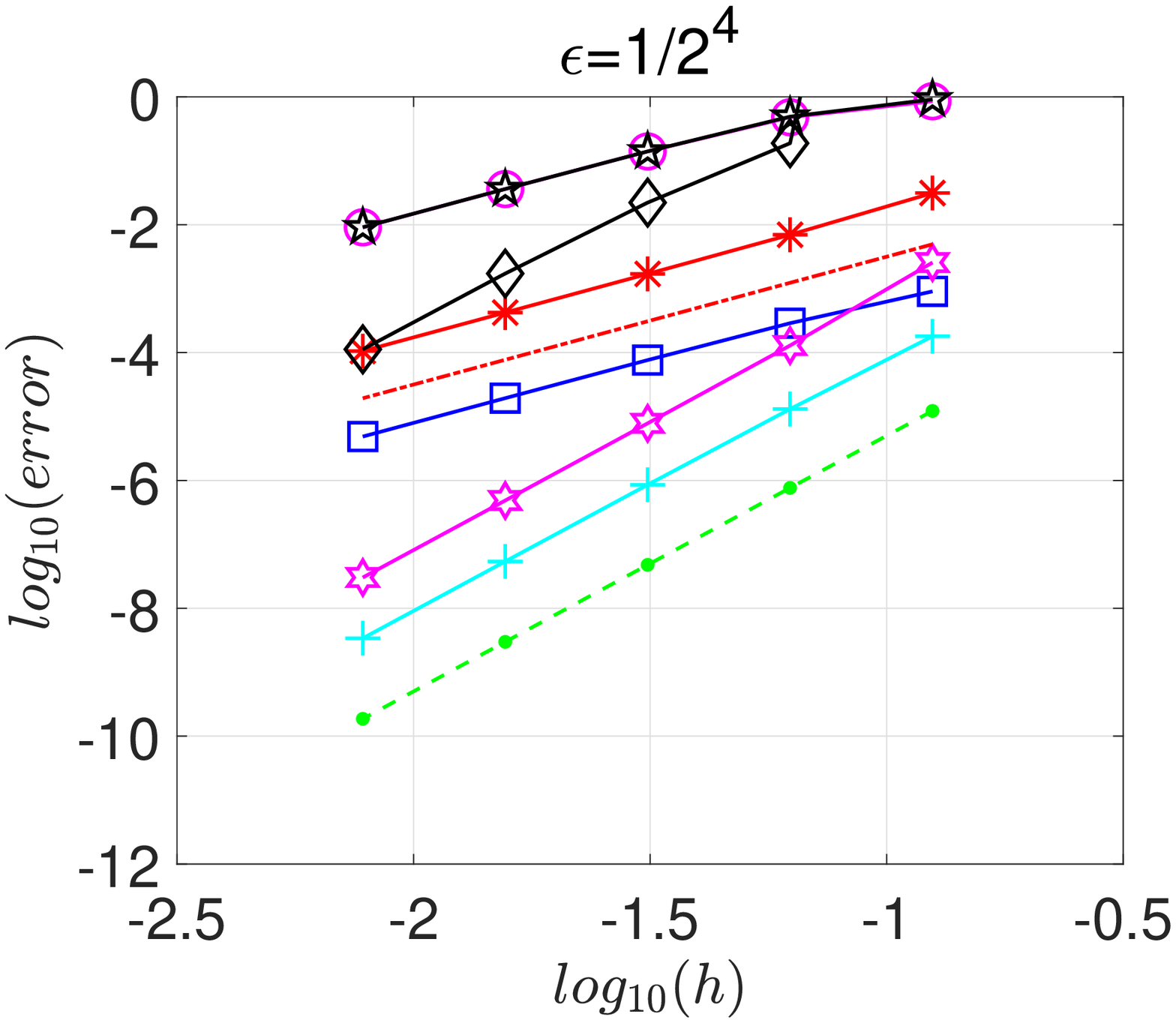}
\end{tabular}
\caption{Problem 1. The global errors $error:=\frac{\norm{x_{n}-x(t_n)}}{\norm{x(t_n)}}+\frac{\norm{v_{n}-v(t_n)}}{\norm{v(t_n)}}$ with $t=1$ and $h=1/2^{k}$ for $k=3,4,\ldots,7$ under different $\epsilon$. }
\label{fig:problem11}
\end{figure}

\begin{figure}[t!]
\centering\tabcolsep=0.4mm
\begin{tabular}
[c]{cccc}%
\includegraphics[width=3.8cm,height=4.45cm]{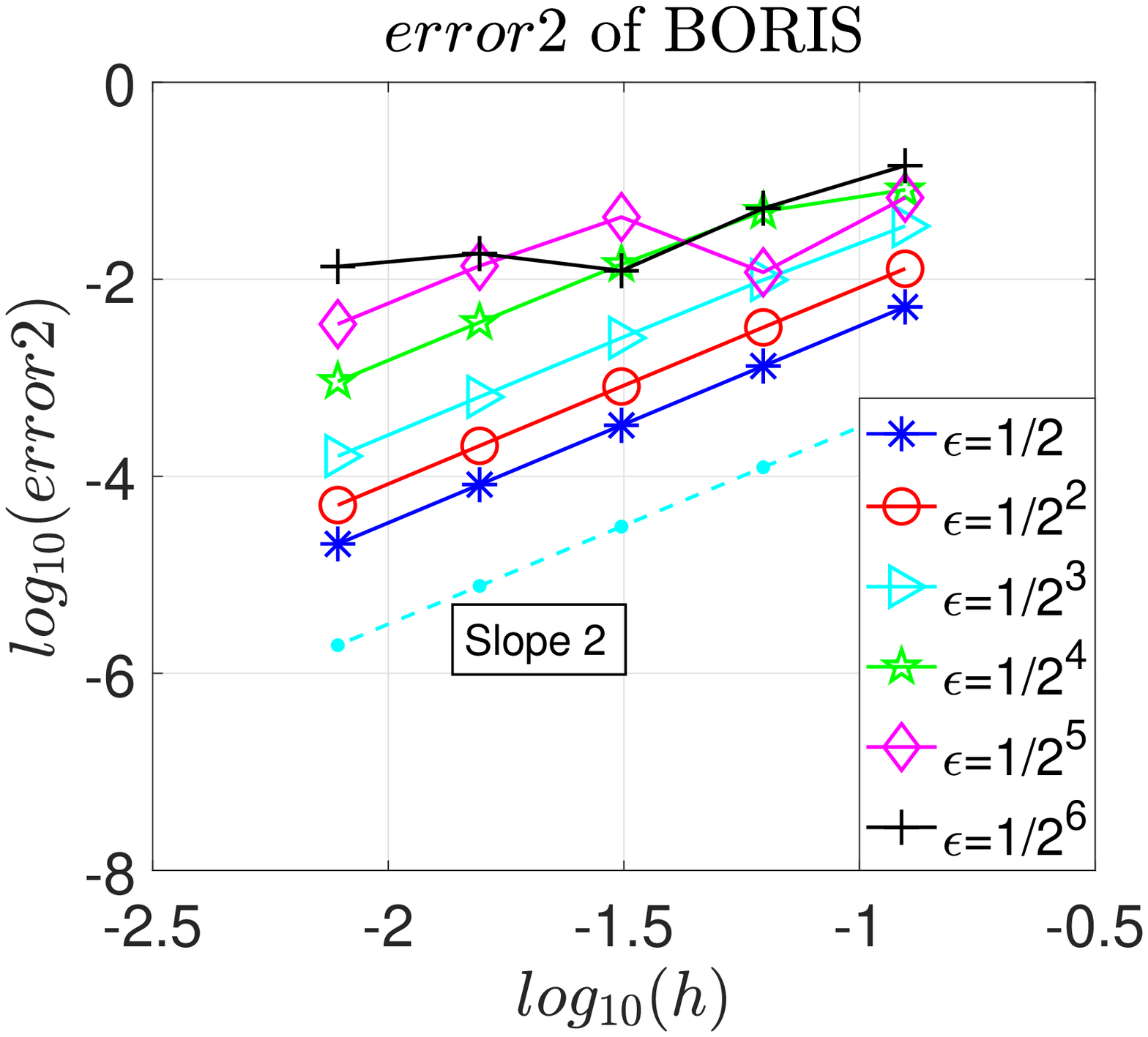} &\includegraphics[width=3.8cm,height=4.45cm]{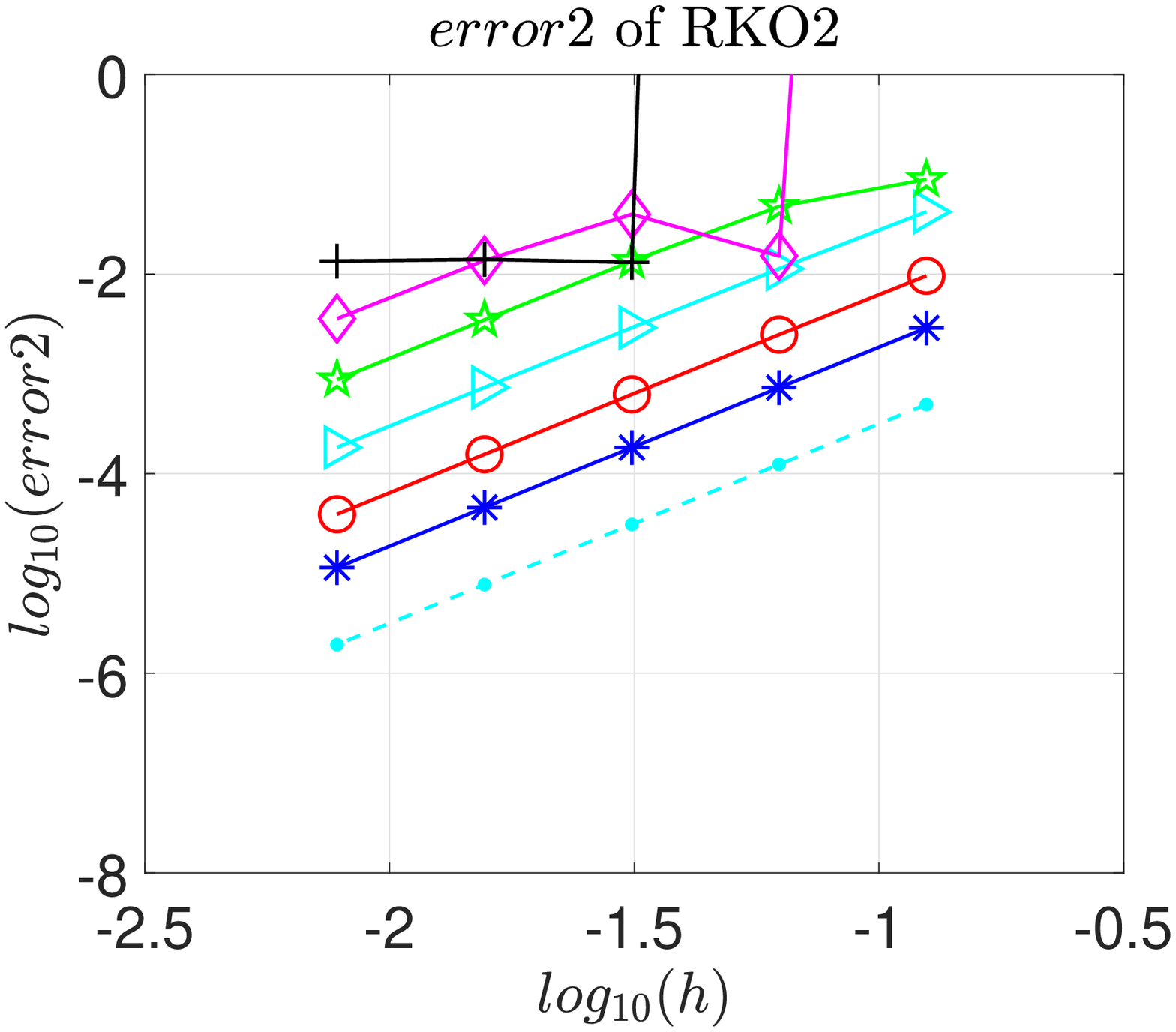} & \includegraphics[width=3.8cm,height=4.45cm]{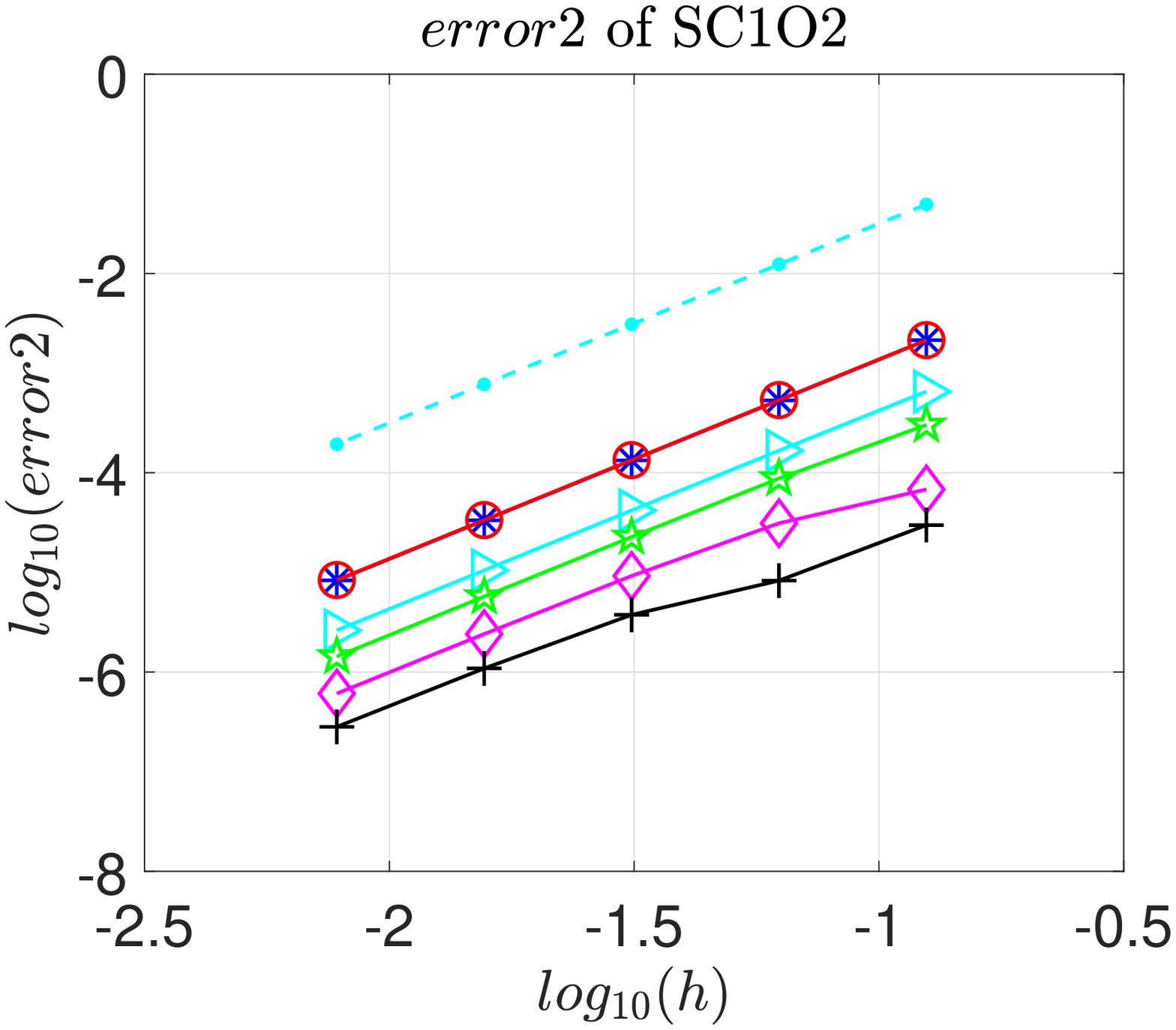} & \includegraphics[width=3.8cm,height=4.45cm]{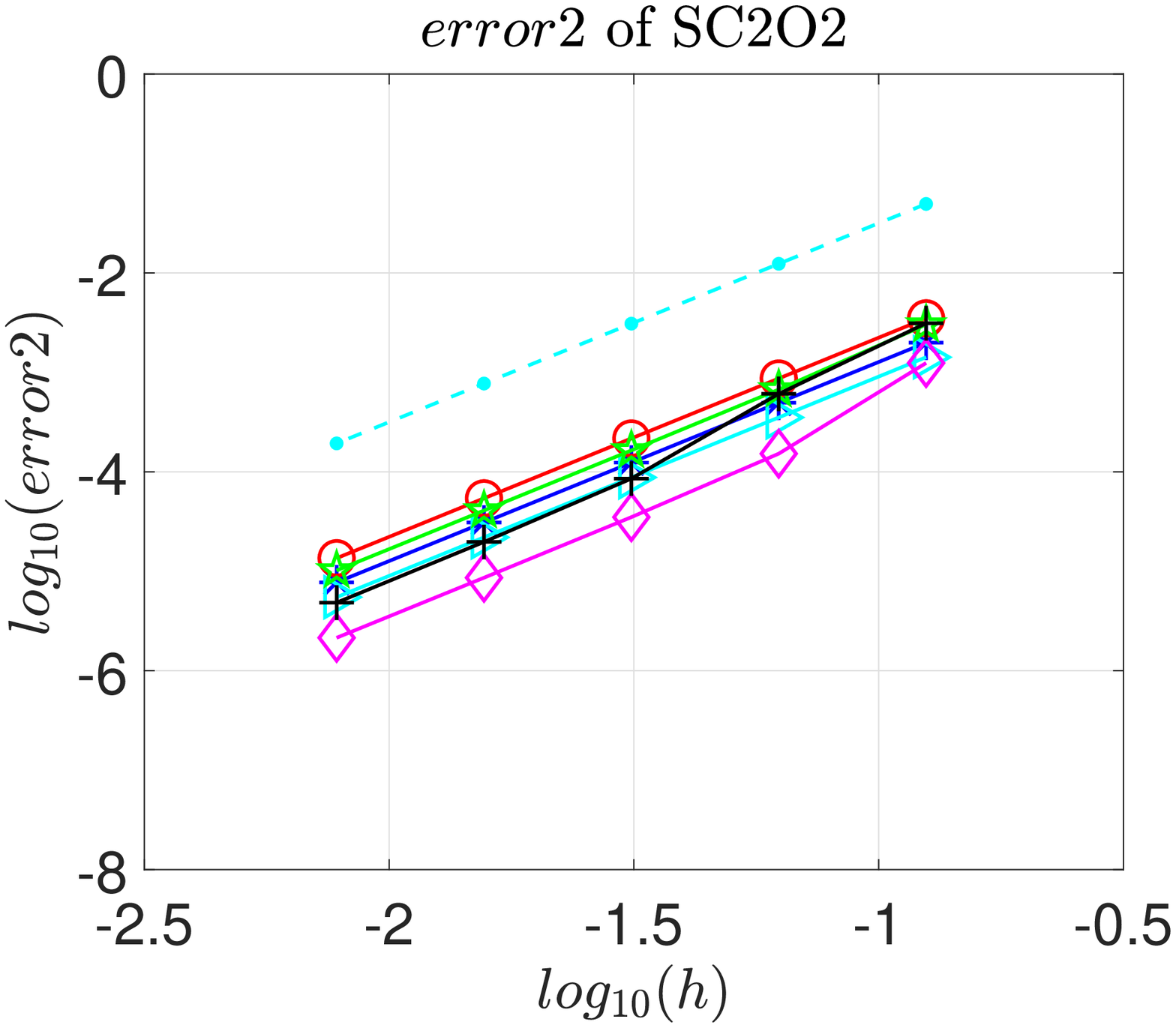}
\end{tabular}
\caption{Problem 1.  The errors {$error2:=\frac{\norm{x_{n}-x(t_n)}}{\norm{x(t_n)}}+\frac{\eps \norm{v_{n}-v(t_n)}}{\norm{v(t_n)}}$} of {second order methods (BORIS, RKO2, SC1O2 and SC2O2)} with $t=1$ and $h=1/2^{k}$ for $k=3,4,\ldots,7$ under different $\epsilon$. }
\label{fig:problem11new1}
\end{figure}

\begin{figure}[t!]
\centering\tabcolsep=0.4mm
\begin{tabular}
[c]{cccc}%
 \includegraphics[width=4.7cm,height=4.45cm]{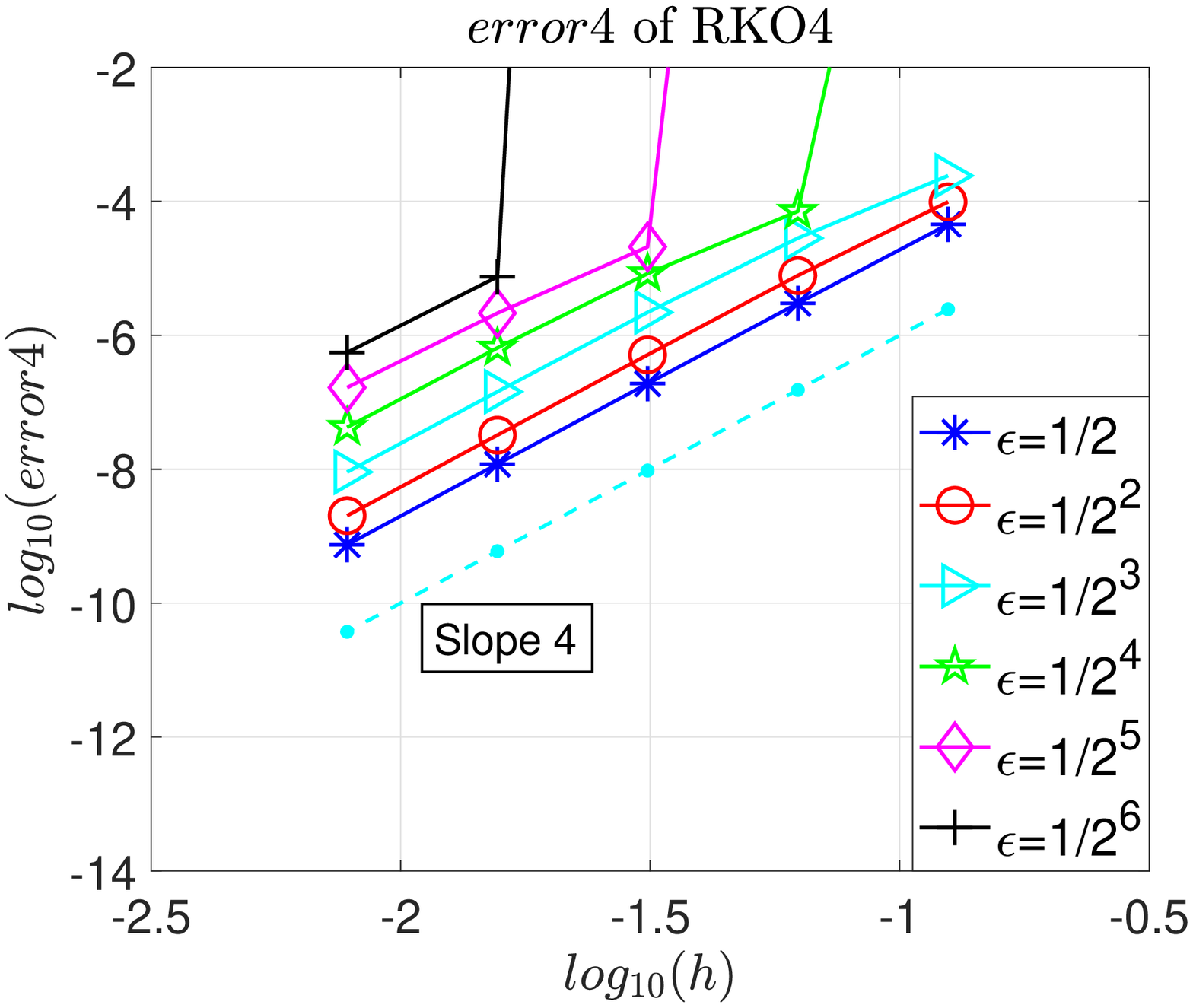} &\includegraphics[width=4.7cm,height=4.45cm]{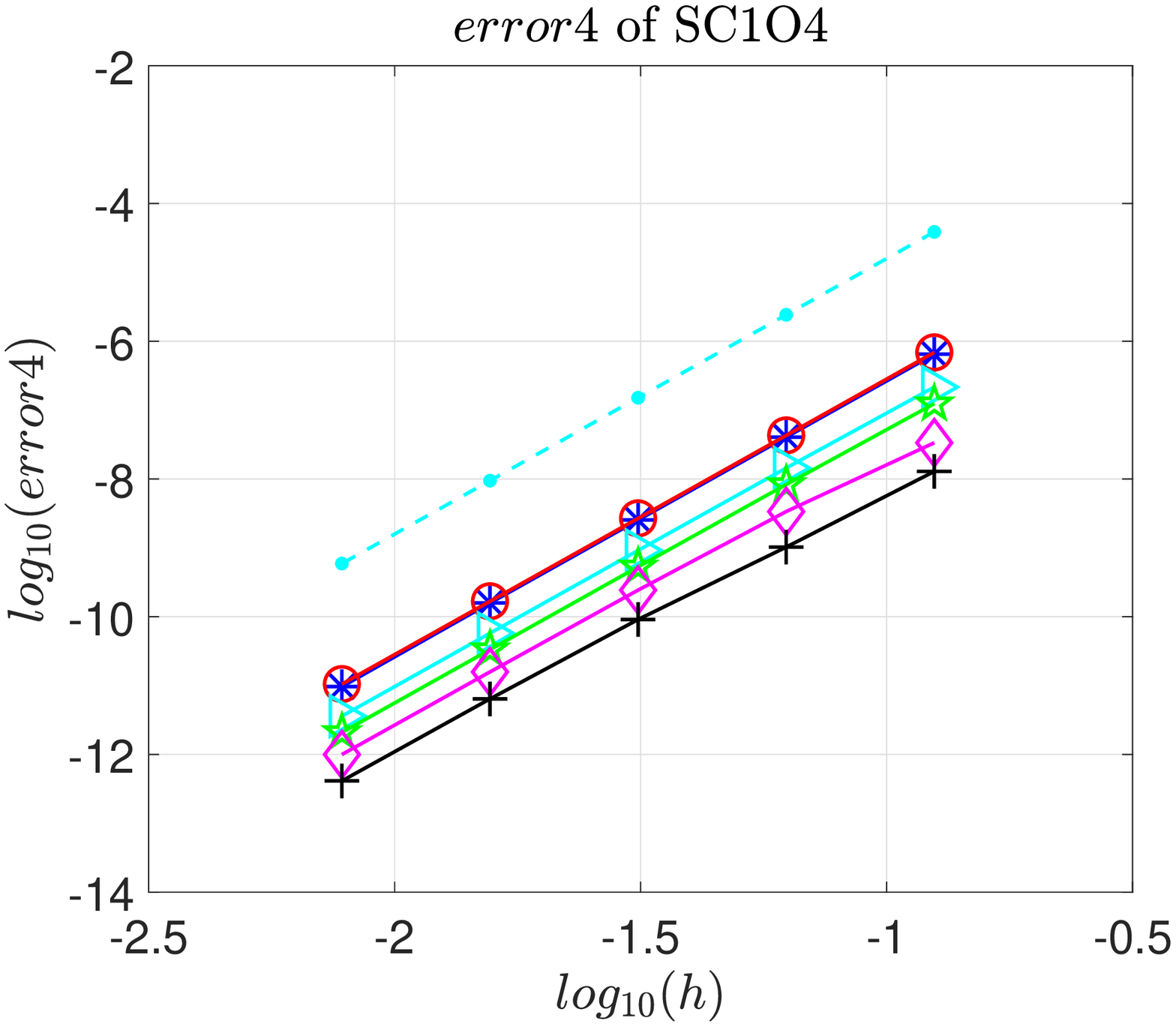} & \includegraphics[width=4.7cm,height=4.45cm]{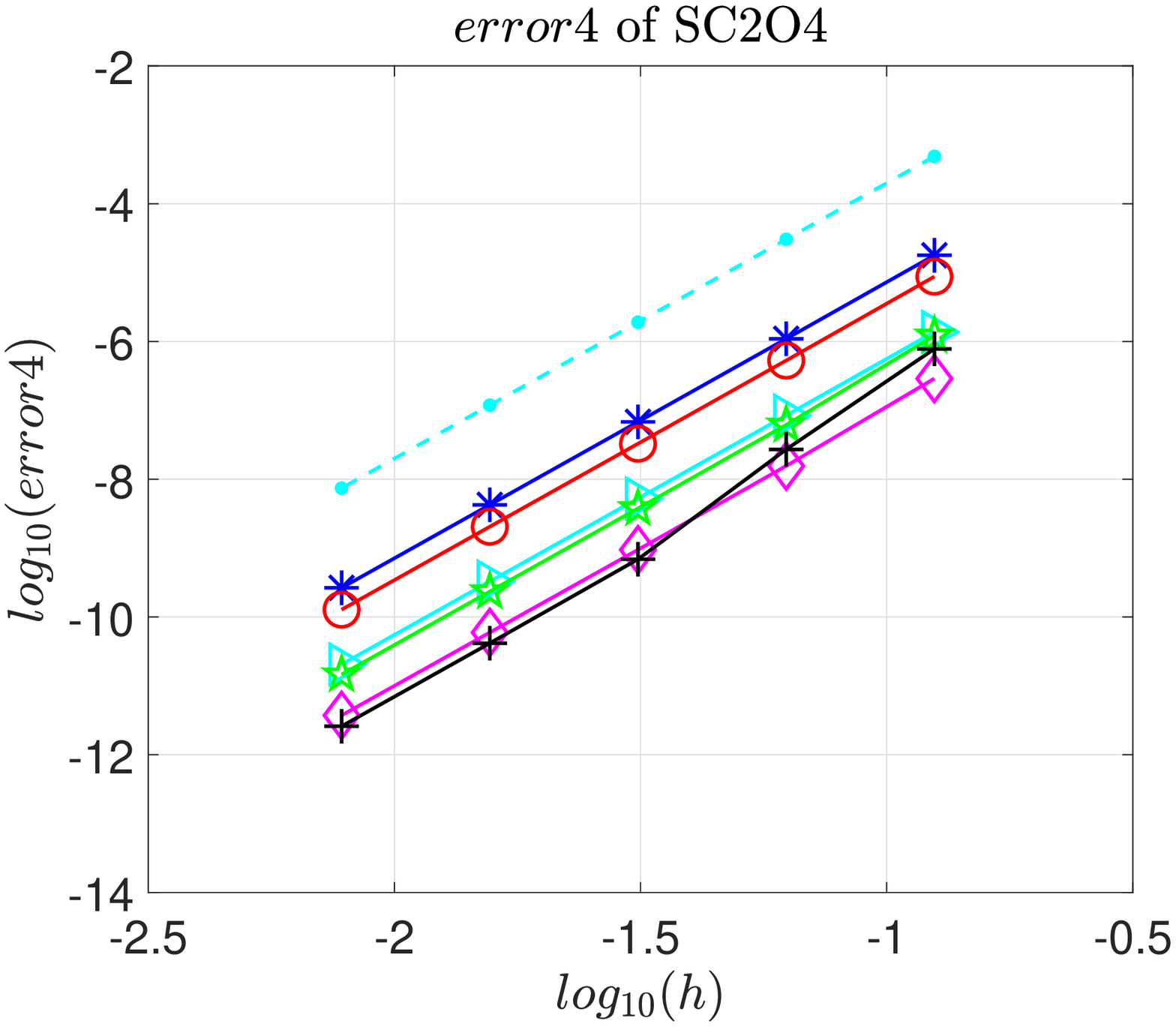}
\end{tabular}
\caption{Problem 1.  The errors {$error4:=\frac{\eps^2\norm{x_{n}-x(t_n)}}{\norm{x(t_n)}}+\frac{\eps^3\norm{v_{n}-v(t_n)}}{\norm{v(t_n)}}$} of {fourth order methods (RKO4, SC1O4 and SC2O4)} with $t=1$ and $h=1/2^{k}$ for $k=3,4,\ldots,7$ under different $\epsilon$. }
\label{fig:problem11new2}
\end{figure}

\begin{figure}[t!]
\centering\tabcolsep=0.4mm
\begin{tabular}
[c]{ccc}%
\includegraphics[width=4.7cm,height=4.45cm]{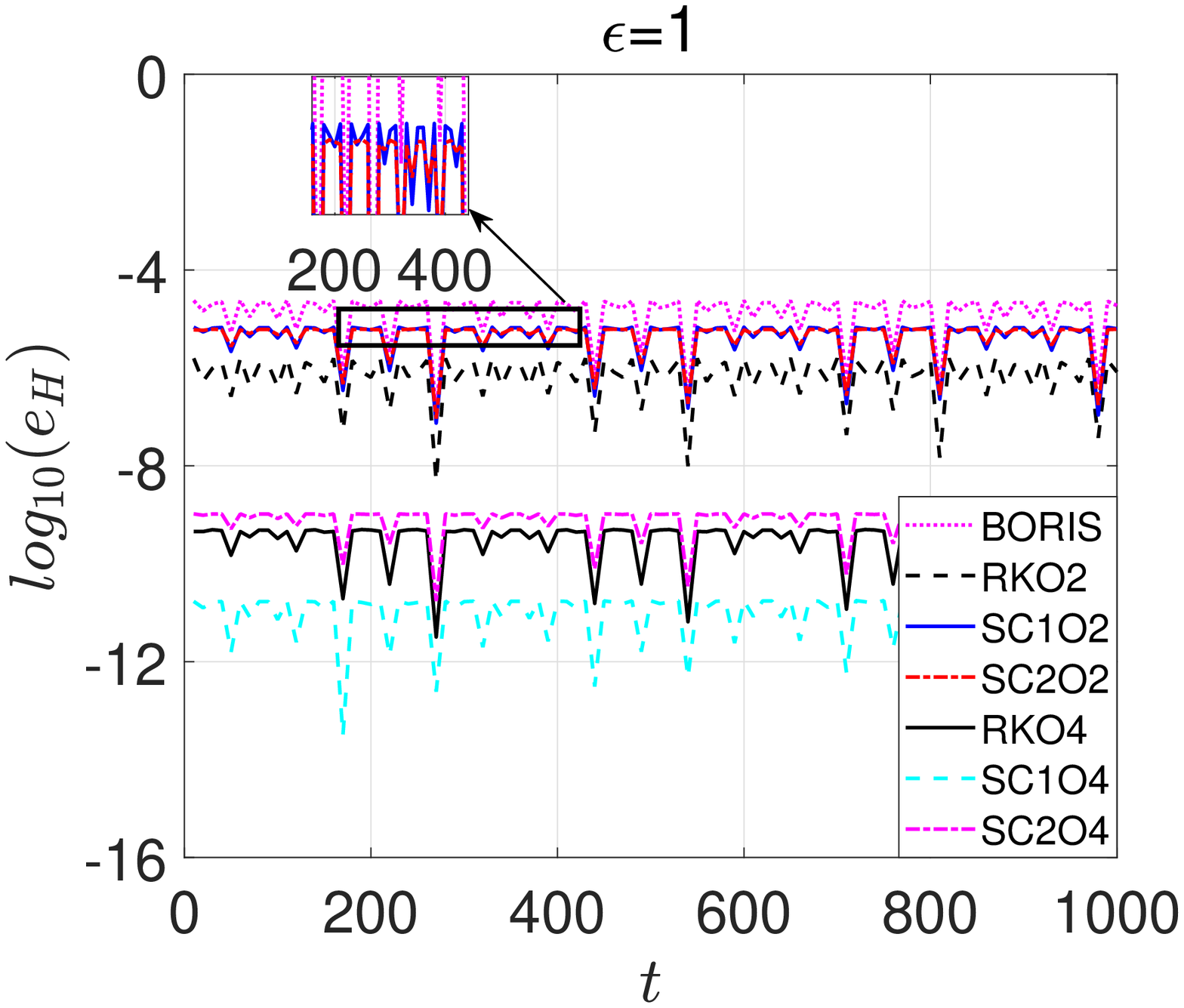} & \includegraphics[width=4.7cm,height=4.45cm]{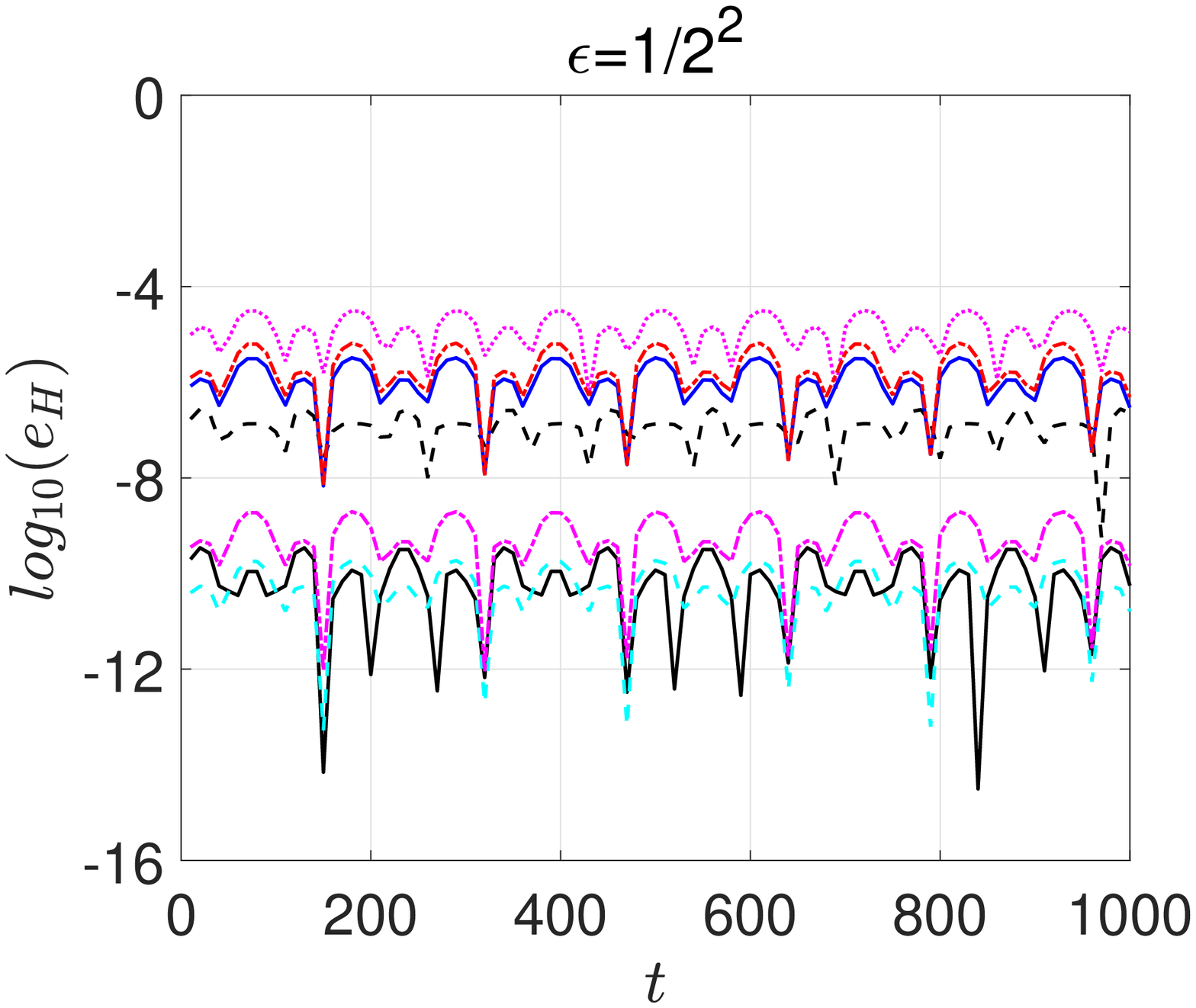} & \includegraphics[width=4.7cm,height=4.45cm]{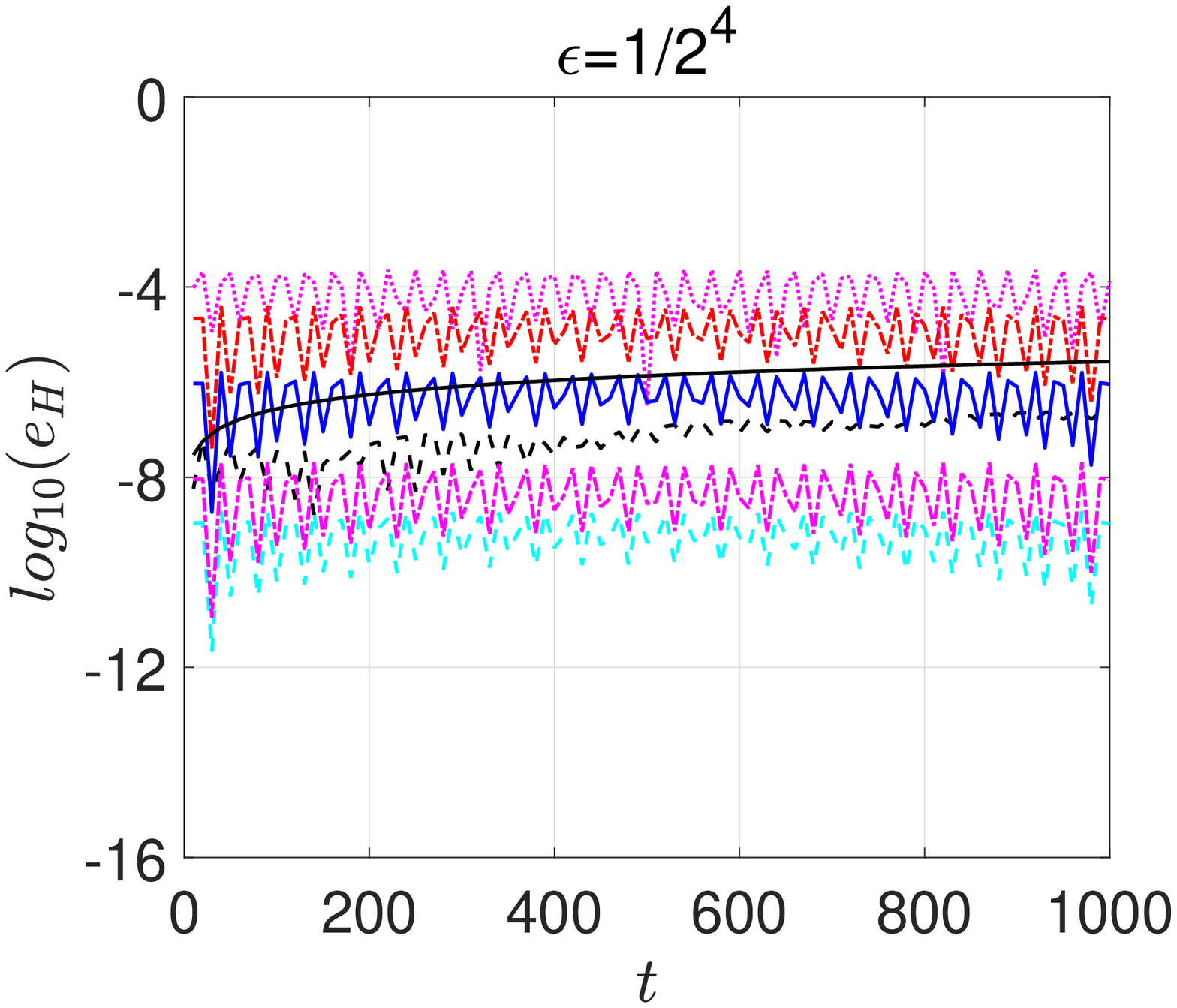}
\end{tabular}
\caption{Problem 1. Evolution of the energy error $e_{H}:=\frac{|H(x_{n},v_n)-H(x_0,v_0)|}{|H(x_0,v_0)|}$ as function of time  $t=nh$. }
\label{fig:problem12}
\end{figure}

%\begin{figure}[t!]
%\centering\tabcolsep=0.4mm
%\begin{tabular}
%[c]{cccc}%
%\includegraphics[width=3.8cm,height=4.45cm]{Boris-1tra2} &\includegraphics[width=3.8cm,height=4.45cm]{RKO2-1tra2} & \includegraphics[width=3.8cm,height=4.45cm]{SC1O2-1tra2} & \includegraphics[width=3.8cm,height=4.45cm]{SC2O2-1tra2}
%\end{tabular}
%\caption{Problem 1.  The trajectory of second order methods (BORIS, RKO2, SC1O2 and SC2O2) in [x y] space with $t=1000$, $h=1/2$ and $\epsilon=0.1$. }
%\label{fig:problem11new3}
%\end{figure}

\begin{figure}[t!]
\centering\tabcolsep=0.4mm
\begin{tabular}
[c]{cccc}%
\includegraphics[width=3.8cm,height=4.45cm]{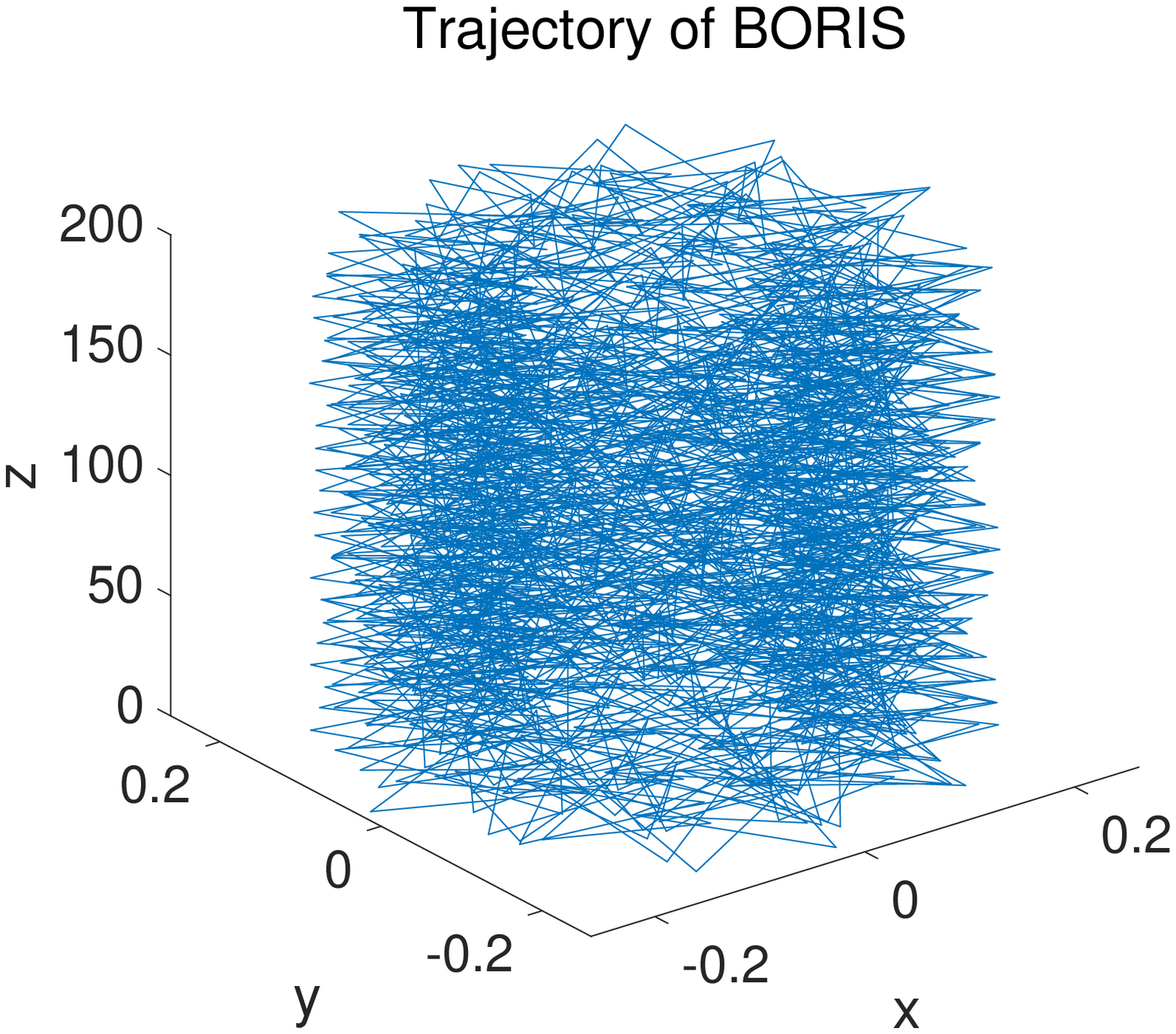} &\includegraphics[width=3.8cm,height=4.45cm]{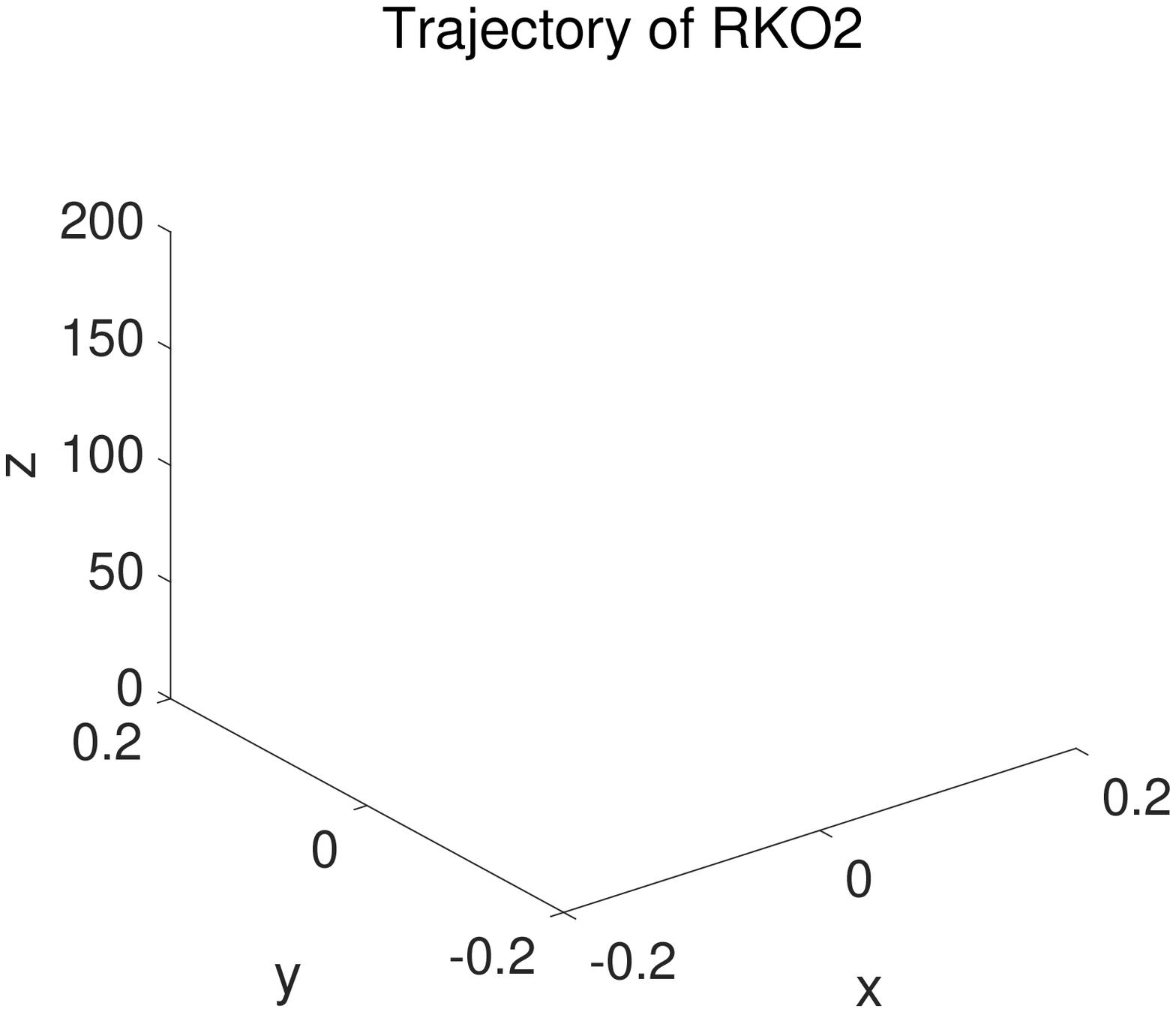} & \includegraphics[width=3.8cm,height=4.45cm]{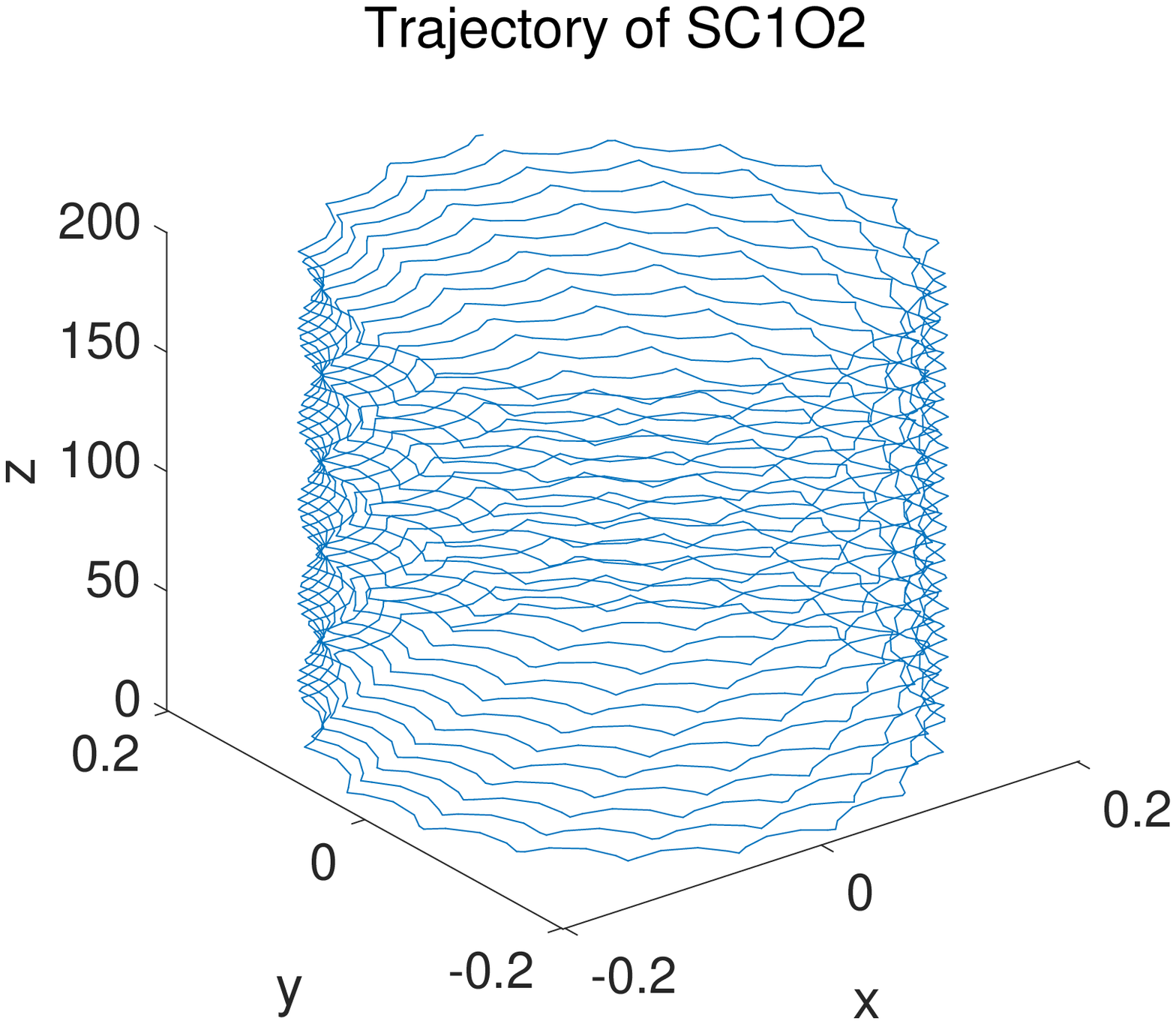} & \includegraphics[width=3.8cm,height=4.45cm]{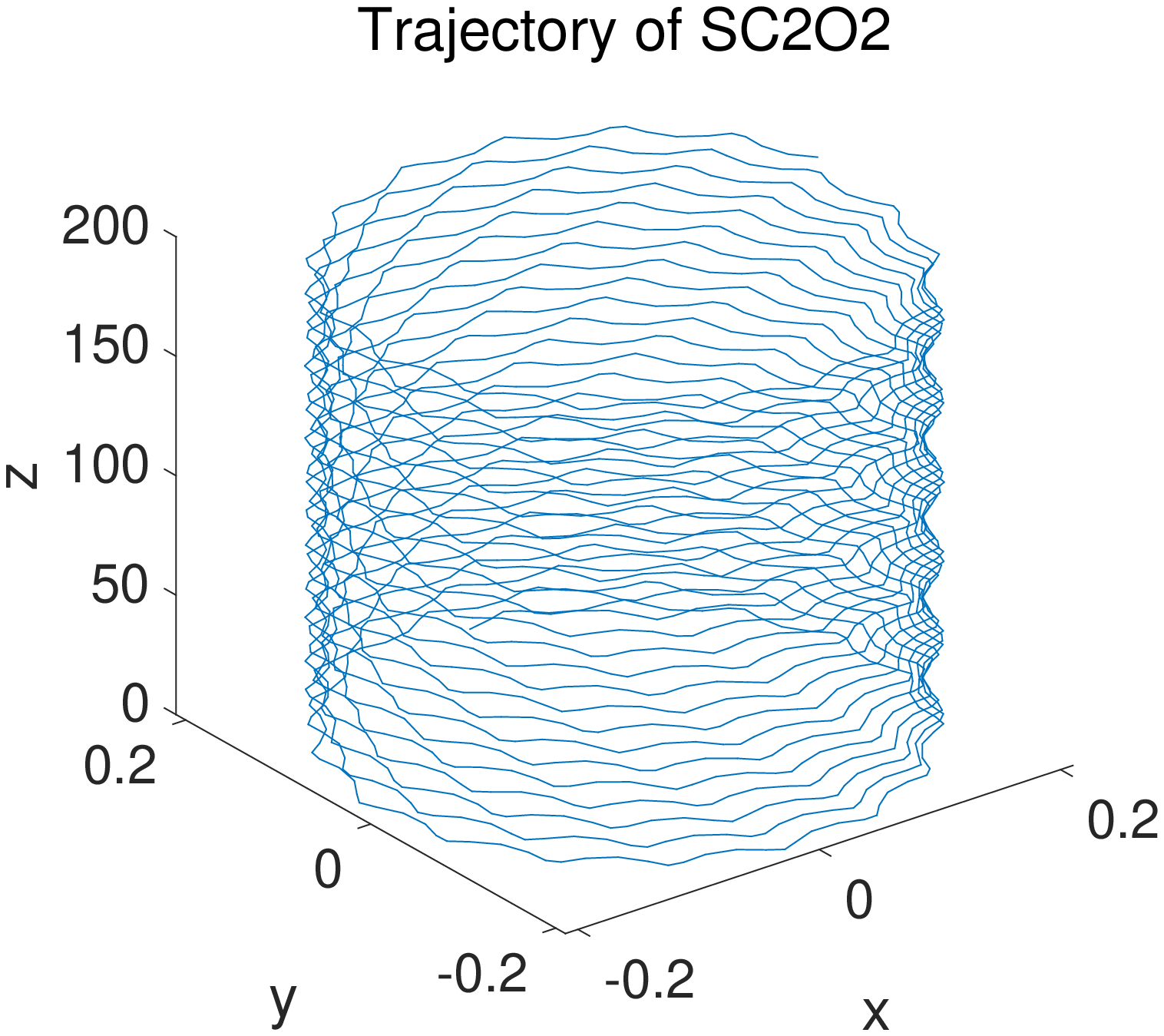}
\end{tabular}
\caption{Problem 1.  The trajectory of second order methods (BORIS, RKO2, SC1O2 and SC2O2) in [x y z] space with $t=1000$, $h=1/2$ and $\epsilon=0.1$. }
\label{fig:problem11new3}
\end{figure}

\section{Extension to CPD with non-homogeneous magnetic field}\label{sec:general}
In this section, we consider  {the extension of the above symplectic methods and corresponding error estimates}  to  the following CPD with a non-homogeneous  magnetic field  $B(x)$:
\begin{equation}\label{charged-particle}
\begin{array}[c]{ll}
\ddot{x}=\dot{x} \times \frac{B(x)}{\epsilon}  +F(x), \quad
x(0)=x_0,\quad \dot{x}(0)=\dot{x}_0.
\end{array}
\end{equation}
We remark that the above methods cannot be directly applied and some novel modifications are needed in the construction of the new methods.

\subsection{Methods and their convergence}
\begin{algo} \label{alg:EPSX}  \textbf{(Adapted exponential  methods)}
For solving the CPD \eqref{charged-particle}  in a non-homogeneous  magnetic field $B (x)$, define the following continuous-stage
 adapted exponential  method  (denoted by SG1O1)
{\begin{equation}\label{SG1O1}
\begin{aligned}
&X_{\tau}=x_{n}+ \tau h\varphi_{1}(\tau hM_n) v_{n}+
h^2 \int_{0}^{1}\frac{\tau-\sigma}{2}\varphi_{1}((\tau-\sigma) hM_n)F (X_\sigma)d\sigma,\\
&x_{n+1}=x_{n}+ h\varphi_1(hM_n)
v_{n}+h^2 \int_{0}^{1}(1-\tau)\varphi_{1}((1-\tau)hM_n)F (X_{\tau})d\tau,\\
&v_{n+1}=\varphi_0(hM_n)v_{n}+h
\int_{0}^{1}\varphi_{0}((1-\tau)hM_n)F(X_{\tau})d\tau,
\end{aligned}
\end{equation} with $M_n=\frac{1}{\epsilon}\tilde{B}(x_{n})$.

The second scheme is obtained by considering another approximation of $\tilde{B}$  which leads to}
\begin{equation}\label{SG1O2}
\begin{aligned}
&X_{\tau}=x_{n}+ \tau h\varphi_{1}(\tau h\tilde{M}_n) v_{n}+
h^2 \int_{0}^{1}\frac{\tau-\sigma}{2}\varphi_{1}((\tau-\sigma) h\tilde{M}_n)F (X_\sigma)d\sigma,\\
&x_{n+1}=x_{n}+ h\varphi_1(h\tilde{M}_n)
v_{n}+h^2 \int_{0}^{1}(1-\tau)\varphi_{1}((1-\tau)h\tilde{M}_n)F (X_{\tau})d\tau,\\
&v_{n+1}=\varphi_0(h\tilde{M}_n)v_{n}+h
\int_{0}^{1}\varphi_{0}((1-\tau)h\tilde{M}_n)F(X_{\tau})d\tau,\\
\end{aligned}
\end{equation}
where  $\tilde{M}_n=\frac{1}{\epsilon}\tilde{B}(\frac{x_{n}+x_{n+1}}{2})$. We shall call it as SG1O2.
		
The third method is formulated as follows. Considering a Triple Jump splitting method \cite{Hairer2002} and denoting  	SG1O2  by $\Upsilon_{h}$, then we get a splitting scheme:
 $	\Phi_h=\Upsilon_{\kappa_1h}\circ \Upsilon_{\kappa_2h}\circ\Upsilon_{\kappa_3h}$,
where  $\kappa_1=\kappa_3=\frac{1}{2-\sqrt[3]{2}}$ and $\kappa_2=-\frac{\sqrt[3]{2}}{2-\sqrt[3]{2}}$. We denote this  method by SG1O4.
\end{algo}

The following theorem states the convergence of the above three methods.

\begin{mytheo}\label{order condition-B(x)} \textbf{(Convergence)}
For the CPD
\eqref{charged-particle}, we assume that its solution is sufficiently smooth, and the function  $ F$ is locally Lipschitz-continuous and
sufficient differentiable.
Under these conditions and  $ h \leq C \epsilon$, the global errors of SG1O1,  SG1O2 and  SG1O4
are respectively given by
\begin{subequations}\label{RR}
\begin{align}&\textmd{SG1O1}:\quad\|x(t_{n+1})-x_{n+1}\|\leq C h,\qquad \qquad  \|v(t_{n+1})-v_{n+1}\| \leq C h/\epsilon,\label{R0}\\
&\textmd{SG1O2}:\quad\|x(t_{n+1})-x_{n+1}\|\leq C h^2/\epsilon,\qquad  \ \    \|v(t_{n+1})-v_{n+1}\| \leq C h^2/\epsilon^2,\label{R1}\\
&\textmd{SG1O4}:\quad\|x(t_{n+1})-x_{n+1}\|\leq C h^4/\epsilon^3,\qquad \ \|v(t_{n+1})-v_{n+1}\|\leq C h^4/\epsilon^4, \label{R2} \end{align}
\end{subequations} where $0<n<T/h$ and the
generic constant  $C>0$ is independent of $\epsilon,\ h,\ n$.
\end{mytheo}

\begin{proof} {The proof is firstly given for the errors of SG1O2 and then the results of the other methods can be derived with
 similar procedure.}

 { $\bullet$  \textbf{Proof of \eqref{R1}.}}
\textbf{Linearized system.} The first step is to rewrite   the system \eqref{charged-particle} as
\begin{equation}\label{charged-particle:B(x)}
\begin{aligned}
 \dot{x}=v,\ \ \dot{v}=v \times \frac{B(x)}{\epsilon}  +F(x),\ \
 x(0)=x_0,\ \ v(0)=v_0.
\end{aligned}
\end{equation}
For some $t=t_n+s$ with $n\geq 0$, we consider the following  linearized system of \eqref{charged-particle:B(x)}
\begin{equation}\label{linearized system}
\begin{aligned}
\dot{\tilde{x}}_n(s)=\dot{\tilde{v}}_n(s),\ \
\dot{\tilde{v}}_n(s)=\frac{1}{\epsilon}\tilde{B}(x_{n+\frac{1}{2}})\tilde{v}_n(s) +F({\tilde{x}}_n(s)),\ \
{\tilde{x}}_n(0)=x(t_n),\ \ {\tilde{v}}_n(0)=v(t_n),\ \ \ 0<s\leq h, %\mathfrak{h},
\end{aligned}
\end{equation}
where {$x_{n+\frac{1}{2}}=%x\big(\frac{t_{n}+t_{n+1}}{2}\big)
\frac{x (t_{n})+x (t_{n+1})}{2} $}.

\textbf{Local errors between the original and linearized systems.}
Then we denote the local errors between the above two systems by
 $$\xi^{x}_{n}(s)=x(t_{n}+s)-\tilde{x}_{n}(s), \ \ \xi^{v}_{n}(s)=v(t_{n}+s)-\tilde{v}_{n}(s),\ \ \ \ 0\leq n <\frac{T}{h}.$$
It is easily obtained from  \eqref{charged-particle:B(x)} and  \eqref{linearized system} that
\begin{equation}\label{er1}
 \begin{aligned}
&\dot{\xi}^{x}_{n}(s)=\xi^{v}_{n}(s),\ \
 \dot{\xi}^{v}_{n}(s)=\frac{1}{\epsilon} \tilde{B}(x_{n+\frac{1}{2}})\xi^{v}_{n}(s) +F(x(t_n+s))-F(\tilde{x}_{n}(s))
+\varsigma^{0}_{n}(s).
\end{aligned}
\end{equation}
Here we denote $\varsigma^{0}_{n}(s)=\frac{1}{\epsilon}\big(\tilde{B}\big(x(t_{n}+s)\big)-\tilde{B}\big(x_{n+\frac{1}{2}}\big)\big)v(t_{n}+s)$,
and using the Taylor expansion yields
%\begin{equation*}
%\begin{aligned}
%\tilde{B}(x(t_{n}+s))-\tilde{B}\big(x_{n+\frac{1}{2}}\big)
%%=\tilde{B}\big(x(t_n)+\dot{x}(t_n)s+\mathcal{O}(s^2/\epsilon)\big)-\tilde{B}\big((x(t_n)+x(t_{n}+h\dot{x}(t_n)+\mathcal{O}(h^2/\epsilon))/2)\big)\\
%&=\tilde{B}(x(t_n))+\frac{d\tilde{B}}{dx}\big(\dot{x}(t_n)s+\mathcal{O}(s^2/\epsilon)\big)+{\mathcal{O}(s^3/\epsilon)}
%-\tilde{B}(x(t_n))-\mathcal{O}(h^2/\epsilon)\\
%&\quad \   -\frac{d\tilde{B}}{dx}\big(\dot{x}(t_n)h/2+\mathcal{O}(h^2/\epsilon)\big)
%=\tilde{B}^{(1)}_n(s-h/2)+\mathcal{O}(h^2/\epsilon).
%\end{aligned}
%\end{equation*}
\begin{equation*}
\begin{aligned}
&\tilde{B}(x(t_{n}+s))-\tilde{B}\big(x_{n+\frac{1}{2}}\big)
%=\tilde{B}\big(x(t_n)+\dot{x}(t_n)s+\mathcal{O}(s^2/\epsilon)\big)-\tilde{B}\big((x(t_n)+x(t_{n}+h\dot{x}(t_n)+\mathcal{O}(h^2/\epsilon))/2)\big)\\
=\tilde{B}(x(t_n))+\frac{d\tilde{B}}{dx}(\dot{x}(t_n)s+\mathcal{O}(s^2/\epsilon))+{\mathcal{O}(s^3/\epsilon)}
-\tilde{B}(x(t_n))\\
&\ \   -\frac{d\tilde{B}}{dx}(\dot{x}(t_n)h/2+\mathcal{O}(h^2/\epsilon))-\mathcal{O}(h^2/\epsilon)
=\tilde{B}^{(1)}_n(s-h/2)+\mathcal{O}(h^2/\epsilon),
\end{aligned}
\end{equation*}
where $\tilde{B}^{(i)}_n=\frac{d^i}{dt^i}\tilde{B}(x(t))\mid_{t=t_n},\ \ i=1,2$.
Likewise, we get
$$v(t_n+s)=v(t_n)+\dot{v}(t_n)s+\mathcal{O}(s^2/\epsilon^2).$$
Combining the above two formulae, one arrives at
\begin{equation}\label{err0}
\begin{aligned}
\varsigma^{0}_{n}(s)= \frac{1}{\epsilon}\tilde{B}_{n}^{(1)}(s-h/2)v(t_n)+\frac{1}{\epsilon}\tilde{B}_{n}^{(1)} s (s-h/2)\dot{v}(t_n)+\mathcal{O}(h^2/\epsilon^2).\\
\end{aligned}
\end{equation}
Then  taking  the variation-of-constant formula of \eqref{er1} into account, we have that
\begin{equation*}
\begin{aligned}
&\xi^{x}_{n}(h)=\int_{0}^{h}\xi^{v}_{n}(s)ds,\ \
 \xi^{v}_{n}(h)=\int_{0}^{h}e^{\frac{h-s}{\epsilon}\tilde{B}(x_{n+\frac{1}{2}})}
\bigg(\int_{0}^{1}\nabla F(x(t_n+s)+(\sigma-1)\xi^{x}_{n}(s))\xi^{x}_{n}(s)d\sigma+\varsigma^{0}_{n}(s)\bigg)ds,
\end{aligned}
\end{equation*}
which further gives
\begin{equation}\label{xi}
\begin{aligned}
&\xi^{x}_{n}(h)=\int_{0}^{h}\int_{0}^{s}e^{\frac{s-\delta}{\epsilon}\tilde{B}(x_{n+\frac{1}{2}})}
\int_{0}^{1}\nabla F(x(t_n+\delta)+(\sigma-1)\xi^{x}_{n}(\delta))\xi^{x}_{n}(s)d\sigma d\delta ds\\
&\ \ \ \  \qquad+\int_{0}^{h}\int_{0}^{s}e^{\frac{s-\delta}{\epsilon}\tilde{B}(x_{n+\frac{1}{2}})}\varsigma^{0}_{n}(\delta)d\delta ds.\\
\end{aligned}
\end{equation}
Inserting the formula \eqref{err0} into the second part of \eqref{xi}, we get
\begin{equation*}
\begin{aligned}
&\int_{0}^{h}\int_{0}^{s}e^{\frac{s-\delta}{\epsilon}\tilde{B}(x_{n+\frac{1}{2}})}\varsigma^{0}_{n}(\delta)d\delta ds
= \underbrace{\frac{1}{\epsilon}\int_{0}^{h}\int_{0}^{s}e^{\frac{s-\delta}{\epsilon}\tilde{B}(x_{n+\frac{1}{2}})}B_{n}^{(1)}
 (\delta-h/2)v(t_n)d\delta ds}_{\textmd{part}\ I}\\
& \quad +\underbrace{\frac{1}{\epsilon}\int_{0}^{h}\int_{0}^{s}e^{\frac{s-\delta}{\epsilon}\tilde{B}(x_{n+\frac{1}{2}})}B_{n}^{(1)}
\delta(\delta-h/2)\dot{v}(t_n) d\delta ds}_{\textmd{part}\ II}+\mathcal{O}(h^4/\epsilon^2).
%&+\underbrace{\frac{1}{\epsilon}\int_{0}^{h}\int_{0}^{s}e^{\frac{s-\delta}{\epsilon}\tilde{B}(x_{n+\frac{1}{2}})}B_{n}^{(2)}
% (\delta^2-h^2/4)v(t_n)d\delta ds}_{\textmd{part}\ II}
\end{aligned}
\end{equation*}
With some calculations, it is easy to have that
$$\textmd{part}\ I=\mathcal{O}(h^3/\epsilon),\quad\textmd{part}\ II=\mathcal{O}(h^5/\epsilon^2).$$
%III=\mathcal{O}(h^4/\epsilon^2),\quad\textmd{part}\ .$$
Considering further that $ h \leq C \epsilon$, we get
$
\int_{0}^{h}\int_{0}^{s}e^{\frac{s-\delta}{\epsilon}\tilde{B}(x_{n+\frac{1}{2}})}\varsigma^{0}_{n}(\delta)d\delta ds=\mathcal{O}(h^3/\epsilon).$
Hence, the local error $\xi^{x}_{n}(h)$ is bounded by
\begin{equation}\label{xix}
\xi^{x}_{n}(h)=\mathcal{O}(h^3/\epsilon).
\end{equation}
In the same way, it is arrived that
\begin{equation*}
\begin{aligned}
&\int_{0}^{h}e^{\frac{h-s}{\epsilon}\tilde{B}(x_{n+\frac{1}{2}})}\varsigma^{0}_{n}(s) ds
=\frac{1}{\epsilon}\int_{0}^{h}e^{\frac{h-s}{\epsilon}\tilde{B}(x_{n+\frac{1}{2}})}v(t_n)
 (s-h/2) B_{n}^{(1)}ds\\
 &\quad+\frac{1}{\epsilon}\int_{0}^{h}e^{\frac{h-s}{\epsilon}\tilde{B}(x_{n+\frac{1}{2}})}\dot{v}(t_n)
 s(s-h/2) B_{n}^{(1)}ds+\mathcal{O}(h^3/\epsilon^2)=\mathcal{O}(h^3/\epsilon^2),
% +\frac{1}{\epsilon}\int_{0}^{h}e^{\frac{h-s}{\epsilon}\tilde{B}(x_{n+\frac{1}{2}})}v(t_n)
% (s^2-h^2/4) B_{n}^{(2)}ds\\
\end{aligned}
\end{equation*}
and based on which, we thus have
\begin{equation}\label{xiv}
\xi^{v}_{n}(h)=\mathcal{O}(h^3/\epsilon^2).
\end{equation}%Similarly, we have

\textbf{Global errors.}
With the above analysis, we shall  estimate the global errors of the method \eqref{SG1O2}, i.e.,
$$e^{x}_{n+1}:=x(t_{n+1})-x_{n+1},\ \ \ e^{v}_{n+1}:=v(t_{n+1})-v_{n+1}.$$
These errors can be split into two parts
\begin{equation}\label{Global error}
e^{x}_{n+1}=\tilde{e}^{x}_{n}+\xi^{x}_{n}(h),\ \ \ e^{v}_{n+1}=\tilde{e}^{v}_{n}+\xi^{v}_{n}(h),
\end{equation}
and then converted to  evaluate $$\tilde{e}^{x}_{n}:=\tilde{x}_{n}(h)-x_{n+1},\ \ \ \tilde{e}^{v}_{n}:=\tilde{v}_{n}(h)-v_{n+1}.$$
%Since  the bounds of $\xi^{x}_{n}(h),\ \xi^{v}_{n}(h)$ have been derived,   one only needs to estimate $\tilde{e}^{x}_{n}$  and
%$\tilde{e}^{v}_{n}$, which can be done in a similar way of \cite{li2022}. The corresponding result is
%$$\tilde{e}^{x}_{n}=\mathcal{O}(h^3/\epsilon),\ \ \ \tilde{e}^{v}_{n}=\mathcal{O}(h^3/\epsilon^2).$$

On the basis of the method \eqref{SG1O2}, we present the local  error $\varphi_n^x$ and $\varphi_n^v$ as follows
\begin{equation}\label{truncation error}
\begin{aligned}
&\tilde{x}_{n}(\tau h)=\tilde{x}_{n}(0)+ \tau h\varphi_{1}\big(\tau h\tilde{P}\big)\tilde{v}_{n}(0)+
h^2 \int_{0}^{1}\frac{\tau-\sigma}{2}\varphi_{1}\big((\tau-\sigma) h\tilde{P}\big)
F (\tilde{x}_{n}(\sigma h))d\sigma+\varphi_n^\tau,\\
&\tilde{x}_{n}(h)=\tilde{x}_{n}(0)+ h\varphi_{1}\big(h\tilde{P}\big)
\tilde{v}_{n}(0)+h^2 \int_{0}^{1}(1-\tau)\varphi_{1}\big((1-\tau)h \tilde{P}\big)
F (\tilde{x}_{n}(\tau h))d\tau+\varphi_n^x,\\
&\tilde{v}_{n}(h)=\varphi_{0}\big( h\tilde{P}\big)\tilde{v}_{n}(0)+h
\int_{0}^{1}\varphi_{0}\big((1-\tau)h\tilde{P}\big)F(\tilde{x}_{n}(\tau h))d\tau+\varphi_n^v,\\
\end{aligned}
\end{equation}
where  $\tilde{P}=\frac{1}{\epsilon}\tilde{B}\Big(\frac{\tilde{x}_{n}(0)+\tilde{x}_{n}(h)}{2}\Big)$.
It follows from the variation-of-constant formula of the linearized system \eqref{linearized system} that
\begin{equation}\label{VOC of linearized system}
\begin{aligned}
&\tilde{x}_{n}(\tau h)=\tilde{x}_{n}(0)+ \tau h\varphi_{1}(\tau hP)
\tilde{v}_{n}(0)+h^2 \int_{0}^{\tau}(\tau-\sigma)\varphi_{1}((\tau-\sigma)h P)
F (\tilde{x}_{n}(\sigma h))d\sigma,\\
&\tilde{x}_{n}(h)=\tilde{x}_{n}(0)+ h\varphi_{1}(hP)
\tilde{v}_{n}(0)+h^2 \int_{0}^{1}(1-\tau)\varphi_{1}((1-\tau)h P)
F (\tilde{x}_{n}(\tau h))d\tau,\\
&\tilde{v}_{n}(h)=\varphi_{0}( hP)\tilde{v}_{n}(0)+h
\int_{0}^{1}\varphi_{0}((1-\tau)hP)F(\tilde{x}_{n}(\tau h))d\tau,\\
\end{aligned}
\end{equation}
where  $P=\frac{1}{\epsilon}\tilde{B}\Big(\frac{x(t_{n})+x(t_{n+1})}{2}\Big)$.
Considering the difference between   \eqref{truncation error} and  \eqref{VOC of linearized system}, we get
\begin{equation*}\label{}
\begin{aligned}
\varphi_n^\tau%\tau h\big(\varphi_{1}(\tau hP)-\varphi_{1}(\tau h\tilde{P})\big){v}(t_{n})
%+h^2 \int_{0}^{\tau}(\tau-\sigma)\varphi_{1}\big((\tau-\sigma)h P\big)F \big(\tilde{x}_{n}(\sigma h)\big)d\sigma\\
%&\quad-h^2 \int_{0}^{1}\frac{(\tau-\sigma)}{2}\varphi_{1}\big((\tau-\sigma)h \tilde{P}\big)F \big(\tilde{x}_{n}(\sigma h)\big)d\sigma\\
&=\tau h\big(\varphi_{1}(\tau hP)-\varphi_{1}(\tau h\tilde{P})\big){v}(t_{n})
+\tau^2 h^2 \int_{0}^{1}(1-z)\varphi_{1}\big(\tau(1-z)h P\big)F \big(\tilde{x}_{n}( h \tau z)\big)dz\\
&\quad-h^2 \int_{0}^{1}\frac{(\tau-\sigma)}{2}\varphi_{1}\big((\tau-\sigma)h \tilde{P}\big)F \big(\tilde{x}_{n}(\sigma h)\big)d\sigma,\\
\varphi_n^x&=h\big(\varphi_{1}(hP)-\varphi_{1}(h\tilde{P})\big){v}(t_{n})
+h^2 \int_{0}^{1}(1-\tau)\big(\varphi_{1}\big((1-\tau)h P\big)-\varphi_{1}\big((1-\tau)h \tilde{P}\big)\big)
F \big(\tilde{x}_{n}(\tau h)\big)d\tau,\\
\varphi_n^v&=\big(\varphi_{0}(hP)-\varphi_{0}(h\tilde{P})\big){v}(t_{n})+h \int_{0}^{1}\big(\varphi_{0}\big((1-\tau)h P\big)-\varphi_{0}\big((1-\tau)h \tilde{P}\big)\big)
F \big(\tilde{x}_{n}(\tau h)\big)d\tau.
\end{aligned}
\end{equation*}
Similarly as the analysis of Theorem \ref{order condition},   it follows that
\begin{equation}\label{phixv}
\norm{\varphi_n^\tau}=\mathcal{O}(h^2),\qquad \   \norm{\varphi_n^x}=\mathcal{O}(h^3/\epsilon),\qquad
 \ \norm{\varphi_n^v}=\mathcal{O}(h^3/\epsilon).
\end{equation}

Subtracting   \eqref{SG1O2} from    \eqref{truncation error} and by further using
\eqref{Global error}, we have
\begin{equation*}\label{}
\begin{aligned}
&{e}^{x}_{n+1}={e}^{x}_{n}+h\varphi_{1}(h\tilde{P}){e}^{v}_{n}+\varphi_n^x+\xi^{x}_{n}(h)+\gamma_{n}^{x},\ \ \ \ {e}^{v}_{n+1}=\varphi_{0}(h\tilde{P}){e}^{v}_{n}+\varphi_n^v+\xi^{v}_{n}(h)+\gamma_{n}^{v},
\end{aligned}
\end{equation*}
where
\begin{equation*}\label{}
\begin{aligned}
&\gamma_{n}^{x}=h\big(\varphi_{1}(h\tilde{P})-\varphi_1(h\tilde{M}_n)\big)v_n
+h^2\int_{0}^{1}(1-\tau)\varphi_{1}\big((1-\tau)h \tilde{P}\big)\big(F (\tilde{x}_{n}(\tau h))-F (X_{\tau})\big)d\tau\\
&\qquad+h^2\int_{0}^{1}(1-\tau)\big(\varphi_{1}\big((1-\tau)h \tilde{P}\big)-\varphi_{1}\big((1-\tau)h\tilde{M}_n\big)\big)F (X_{\tau})d\tau,\\
&\gamma_{n}^{v}=\big(\varphi_{0}(h\tilde{P})-\varphi_0(h\tilde{M}_n)\big)v_n
+h\int_{0}^{1}\varphi_{0}\big((1-\tau)h \tilde{P}\big)\big(F (\tilde{x}_{n}(\tau h))-F (X_{\tau})\big)d\tau\\
&\qquad+h\int_{0}^{1}\big(\varphi_{0}\big((1-\tau)h \tilde{P}\big)-\varphi_{0}\big((1-\tau)h\tilde{M}_n\big)\big)F (X_{\tau})d\tau.\\
\end{aligned}
\end{equation*}
Considering the same process of deduction  in Theorem \ref{order condition}, one gets
\begin{equation*}
\begin{aligned}
&\norm{ \gamma_{n}^{x}}\leq \frac{h^2}{\epsilon}(\|e_n^x\|+h\|e_n^v\|+\|e_{n+1}^x\|+\|\xi_n^x(h)\|),
&\norm{ \gamma_{n}^{v}}\leq \frac{h}{\epsilon}(\|e_n^x\|+h\|e_n^v\|+\|e_{n+1}^x\|+\|\xi_n^x(h)\|).
\end{aligned}
\end{equation*}
We then obtain
\begin{equation*}\label{}
\begin{aligned}
\norm{ e_{n+1}^{x}}
&\leq \norm{ e_{n}^{x}}+h \norm{ e_{n}^{v}}+\norm{ \varphi_{n}^{x}}+\norm{ \xi_{n}^{x}(h)}+\norm{ \gamma_{n}^{x}}\\
%&\leq \norm{ e_{n}^{x}}+h \norm{ e_{n}^{v}}+\norm{ \varphi_{n}^{x}}+\norm{ \xi_{n}^{x}(h)}+
%\frac{h^2}{\epsilon}(\|e_n^x\|+h\|e_n^v\|+\|e_{n+1}^x\|+\|\xi_n^x(h)\|)\\
& \leq \norm{ e_{n}^{x}}+h \big(\norm{ e_{n}^{x}}+\norm{ e_{n}^{v}}+\norm{ e_{n+1}^{x}}\big) +
 \norm{ \varphi_{n}^{x}}+\norm{ \xi_{n}^{x}(h)},
\end{aligned}
\end{equation*}
and
\begin{equation*}\label{}
\begin{aligned}
\norm{ e_{n+1}^{v}}
&\leq \norm{ e_{n}^{v}}+\norm{ \varphi_{n}^{v}}+\norm{ \xi_{n}^{v}(h)}+\norm{ \gamma_{n}^{v}}\\
%&\leq \norm{ e_{n}^{v}}+\norm{ \varphi_{n}^{v}}+\norm{ \xi_{n}^{v}(h)}+\norm{ \gamma_{n}^{v}}+
%\frac{h}{\epsilon}(\|e_n^x\|+h\|e_n^v\|+\|e_{n+1}^x\|+\|\xi_n^x(h)\|)\\
&\leq \norm{ e_{n}^{v}}+\frac{h}{\epsilon} \big(\norm{ e_{n}^{x}}+\norm{ e_{n}^{v}}+\norm{ e_{n+1}^{x}}\big) +
 \norm{ \varphi_{n}^{v}}+\norm{ \xi_{n}^{v}(h)}.
\end{aligned}
\end{equation*}
Combining $\norm{ e_{n+1}^{x}}$ and $\epsilon \norm{ e_{n+1}^{v}}$, we have
\begin{equation*}\label{}
\begin{aligned}
&\norm{ e_{n+1}^{x}}+\epsilon \norm{ e_{n+1}^{v}}- \norm{ e_{n}^{x}}- \epsilon \norm{ e_{n}^{v}}\\
\leq & h(\|e_n^x\|+\|e_n^v\|+\|e_{n+1}^x\|)+\norm{ \varphi_{n}^{x}}+\epsilon \norm{ \varphi_{n}^{v}}
+\norm{ \xi_{n}^{x}(h)}+\epsilon \norm{ \xi_{n}^{v}(h)}.\\
\end{aligned}
\end{equation*}
Further,  using recursion formula  and  $e_{0}^{x}=e_{0}^{v}=0$, it is arrived that
\begin{equation*}\label{}
\begin{aligned}
\norm{ e_{n+1}^{x}}+\epsilon \norm{ e_{n+1}^{v}}
\leq h\sum_{m=0}^{n}(\|e_m^x\|+\|e_m^v\|+\|e_{m+1}^x\|)+\sum_{m=0}^{n}(\norm{ \varphi_{m}^{x}}+\epsilon \norm{ \varphi_{m}^{v}}+\norm{ \xi_{m}^{x}(h)}+\epsilon \norm{ \xi_{m}^{v}(h)}).\\
\end{aligned}
\end{equation*}
By inserting the errors  \eqref{xix}, \eqref{xiv} and \eqref{phixv} into the above formula and using the fact that $nh <T$,  one has
\begin{equation*}\label{}
\begin{aligned}
\norm{ e_{n+1}^{x}}+\epsilon \norm{ e_{n+1}^{v}}
\leq h\sum_{m=0}^{n}(\|e_m^x\|+\|e_m^v\|+\|e_{m+1}^x\|)+Ch^2/\epsilon.
\end{aligned}
\end{equation*}
Using the Gronwall's inequality, we  get
\begin{equation*}\label{}
\begin{aligned}
\norm{ e_{n+1}^{x}}+\epsilon \norm{ e_{n+1}^{v}}
\leq Ch^2/\epsilon.
\end{aligned}
\end{equation*}
Therefore,  the global error \eqref{R1} of  the method \eqref{SG1O2} is obtained immediately.

 $\bullet$  \textbf{Proof of \eqref{R0}.}
This proof  is  also divided into  three parts as the analysis of \eqref{R1}.

Firstly, we denote the linearized system of \eqref{charged-particle:B(x)} as
\begin{equation}\label{linearized system-R0}
\begin{aligned}
\dot{\tilde{x}}_n(s)=\dot{\tilde{v}}_n(s),\
\dot{\tilde{v}}_n(s)=\frac{1}{\epsilon}\tilde{B}(x(t_n))\tilde{v}_n(s) +F({\tilde{x}}_n(s)),\
{\tilde{x}}_n(0)=x(t_n),\  {\tilde{v}}_n(0)=v(t_n),\  0<s\leq h. %\mathfrak{h},
\end{aligned}
\end{equation}

Then  the local errors between the original \eqref{charged-particle:B(x)}  and linearized systems \eqref{linearized system-R0} are given by
$$\xi^{x}_{n}(h)=\mathcal{O}(h^2),\qquad \xi^{v}_{n}(h)=\mathcal{O}(h^2/\epsilon),$$
which can be obtained in the same way by the proof of \eqref{R1}.

Finally, according to  the above analysis, we get the final conclusion as follows
$$\norm{ e_{n+1}^{x}}+\epsilon \norm{ e_{n+1}^{v}}
\leq Ch,$$
which proves  the result of \eqref{R0}.

  $\bullet$  \textbf{Proof of \eqref{R2}.}
 It is noted that the error bound between the system \eqref{charged-particle:B(x)} and linearized system \eqref{linearized system} is $\mathcal{O}(h^2)$. Thus higher order methods can not be obtained by considering algorithms with higher order for the  linearized system \eqref{linearized system}. That is the reason why we  use the idea of Triple Jump splitting   method to construct the fourth-order method SG1O4. The convergence of this splitting can be derived by the standard analysis (see, e.g.,  \cite{Hairer2002}) and we omit it for brevity.

The proof of the theorem is finished.%\hfill  $\blacksquare$
\end{proof}
For the special but important magnetic field,  the so-called \emph{maximal ordering scaling} \cite{scaling1,Lubich2020,scaling2}:
\begin{equation}\label{ma0-charged-particle:B(x)}
\begin{aligned}
 \dot{x}=v,\ \ \dot{v}=v \times \frac{ B(\epsilon x)}{\epsilon}  +F(x),\ \
 x(0)=x_0,\ \ v(0)=v_0,
\end{aligned}
\end{equation}
the improved convergent results can be obtained which are shown as follows.

\begin{mytheo}\label{order conditionN}  \textbf{(Improved error bounds)}
Under the conditions of Theorem \ref{order condition-B(x)}, if  the magnetic field has  the   \emph{maximal ordering scaling}, i.e., $B=B(\eps x)$, we have the following improved error bounds
\begin{subequations}\label{RRNEW}
\begin{align}&\textmd{SG1O1}: \quad \norm{x(t_{n+1})-x_{n+1}}\leq C h,\quad \quad \ \ \   \norm{v(t_{n+1})-v_{n+1}} \leq C h,\label{R0N}\\
&\textmd{SG1O2}: \quad \norm{x(t_{n+1})-x_{n+1}}\leq C h^2,\qquad  \ \   \norm{v(t_{n+1})-v_{n+1}} \leq C h^2/\epsilon,\label{R1N}\\
&\textmd{SG1O4}: \quad \norm{x(t_{n+1})-x_{n+1}}\leq C h^4/\epsilon^2,\quad \ \norm{v(t_{n+1})-v_{n+1}}\leq C h^4/\epsilon^3. \label{R2N}
\end{align}
\end{subequations}
\end{mytheo}

\begin{proof}
 $\bullet$  \textbf{Proof of \eqref{R0N}.}
Concern  the  linearized system of \eqref{ma0-charged-particle:B(x)} as follows
\begin{equation}\label{ma-linearized system}
\begin{aligned}
\dot{\tilde{x}}_n(s)=\dot{\tilde{v}}_n(s),\ \
\dot{\tilde{v}}_n(s)=\frac{1}{\epsilon}\tilde{B}(\epsilon x(t_n))\tilde{v}_n(s) +F({\tilde{x}}_n(s)),\ \
{\tilde{x}}_n(0)=x(t_n),\ \ {\tilde{v}}_n(0)=v(t_n),\ \ \ 0<s\leq h.
\end{aligned}
\end{equation}
Similarly as  the proof of Theorem \ref{order condition-B(x)},   the  errors between the  systems \eqref{ma0-charged-particle:B(x)} and \eqref{ma-linearized system} are bounded  as follows
\begin{equation*}%\label{R0N-local error}
\xi^{x}_{n}(h)=\mathcal{O}(h^3),\ \ \ \ \ \ \ \xi^{v}_{n}(h)=\mathcal{O}(h^3/\epsilon).
\end{equation*}
Then, we shall consider the global errors of the method \eqref{SG1O1} in the same way as Theorem \ref{order condition-B(x)}, and one gets
\begin{equation*}\label{R0N-error}
\begin{aligned}
&\norm{{e}^{x}_{n+1}}\leq \norm{{e}^{x}_{n}}+h\norm{{e}^{v}_{n}}+\norm{\varphi_n^x}+\norm{\xi^{x}_{n}(h)}+\norm{\gamma_{n}^{x}},\\
&\norm{{e}^{v}_{n+1}}\leq \norm{{e}^{v}_{n}}+\norm{\varphi_n^v}+\norm{\xi^{v}_{n}(h)}+\norm{\gamma_{n}^{v}}.
\end{aligned}
\end{equation*}
The notations $\varphi_n^x$, $\varphi_n^v$, $\gamma_{n}^{x}$  and $\gamma_{n}^{v}$ are totally the same as those in the proof of \eqref{R1} with some corresponding modifications in $B$.  They are bounded by
$$
\norm{\varphi_n^x}=\mathcal{O}(h^3/\epsilon), \ \  \norm{\varphi_n^v}=\mathcal{O}(h^3/\epsilon), \ \
\norm{\gamma_n^x}\leq h^2(\norm{e_n^x}+h\norm{e_n^v}),\ \  \norm{\gamma_n^v}\leq h(\norm{e_n^x}+h\norm{e_n^v}).$$
The above results lead to
\begin{equation*}
\begin{aligned}
\norm{{e}^{x}_{n+1}}+\norm{{e}^{v}_{n+1}}-\norm{{e}^{x}_{n}}-\norm{{e}^{v}_{n}} \leq h(\norm{{e}^{x}_{n}}+\norm{{e}^{v}_{n}})
+\norm{\varphi_n^x}+\norm{\varphi_n^v}+\norm{\xi_n^x(h)}+\norm{\xi_n^v(h)}.
\end{aligned}
\end{equation*}
%Then summing them up for $0 \leq n \leq m$ and inserting the errors \eqref{R0N-local error} and \eqref{R0N-error bounds-1}, we have
%\begin{equation*}
%\begin{aligned}
%\norm{{e}^{x}_{n+1}}+\norm{{e}^{v}_{n+1}}
%&\leq h \sum_{n=0}^{m}(\norm{{e}^{x}_{n}}+\norm{{e}^{v}_{n}})
%+\sum_{n=0}^{m}(\norm{\varphi_n^x}+\norm{\varphi_n^v}+\norm{\xi_n^x(h)}+\norm{\xi_n^v(h)})\\
%&\leq h \sum_{n=0}^{m}(\norm{{e}^{x}_{n}}+\norm{{e}^{v}_{n}}) +Ch.
%\end{aligned}
%\end{equation*}
%Finally by Gronwall's inequality gives
Therefore, we have
$$\norm{ e_{n+1}^{x}}+ \norm{ e_{n+1}^{v}} \leq Ch.$$

 $\bullet$  \textbf{Proof of \eqref{R1N}--\eqref{R2N}.}
The proof of  \eqref{R1N} and \eqref{R2N} can be obtained by the same progress with \eqref{R0N}, and so we
leave out the details for simplicity.
\end{proof}

 \subsection{Numerical tests}
 In what follows, we present two numerical experiments to show the behaviour of the derived methods. We still choose   BORIS, RKO2 and RKO4  for comparison which are given in Section \ref{sec: experiment}. {To compare the first order method SG1O1 with some existing methods, the implicit Euler method (of order one) is chosen and we refer to it as Euler.}
\noindent\vskip3mm \noindent\textbf{Problem 2. (General magnetic field)}
For the charged-particle dynamics \eqref{charged-particle sts-cons},  the  scalar potential $U(x)$ and non-homogeneous  magnetic field $\frac{1}{\epsilon}B(x)$ are given by (\cite{Lubich2017})
$$ \frac{1}{\epsilon}B(x)=\nabla \times \frac{1}{3\epsilon}\Big(-x_2\sqrt{x_{1}^2+x_{2}^2},x_1\sqrt{x_{1}^2+x_{2}^2},0\Big)^{T}
=\frac{1}{\epsilon}\Big(0,0,\sqrt{x_{1}^2+x_{2}^2}\Big)^{\intercal}, \ \ \ U(x)=\frac{1}{100\sqrt{x_{1}^2+x_{2}^2}},$$
where    initial values are chosen as
$x(0)=(0,1,0.1)^{\intercal},\ v(0)=(0.09,0.05,0.2)^{\intercal}.$
  The problem is  solved on $[0,1]$ with $h=1/2^k,$ where $k=3,\ldots,7$, and
 {Figs. \ref{fig:problem31}--\ref{fig:problem31new2} show the  results of global errors.} Then the system is integrated on the interval $ [0,100]$
with  a step size $h=\frac{1}{100}$ and see  Fig. \ref{fig:problem32} for the energy conservation.

\begin{figure}[t!]
\centering\tabcolsep=0.4mm
\begin{tabular}
[c]{ccc}%
\includegraphics[width=4.7cm,height=4.5cm]{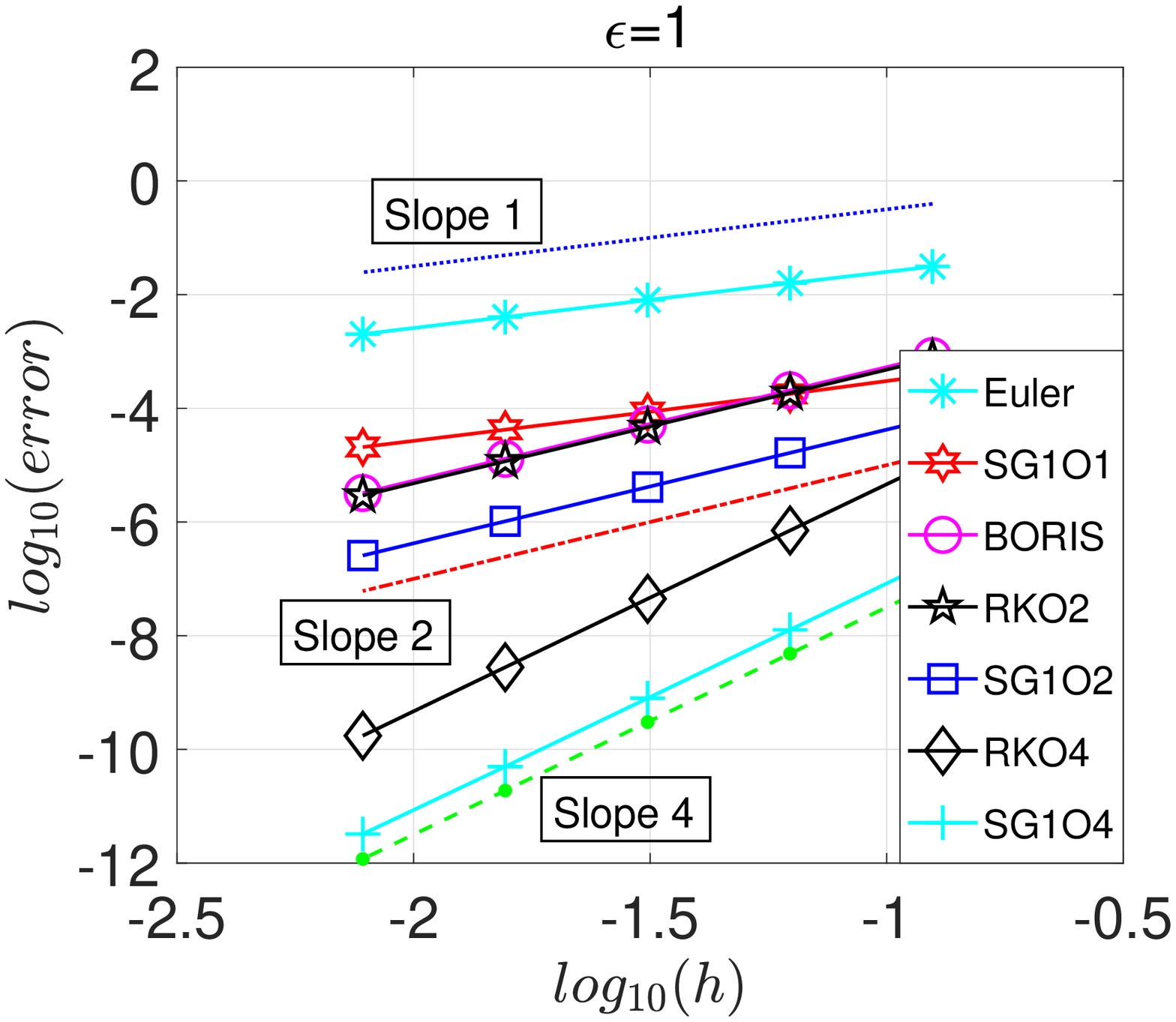} & \includegraphics[width=4.7cm,height=4.5cm]{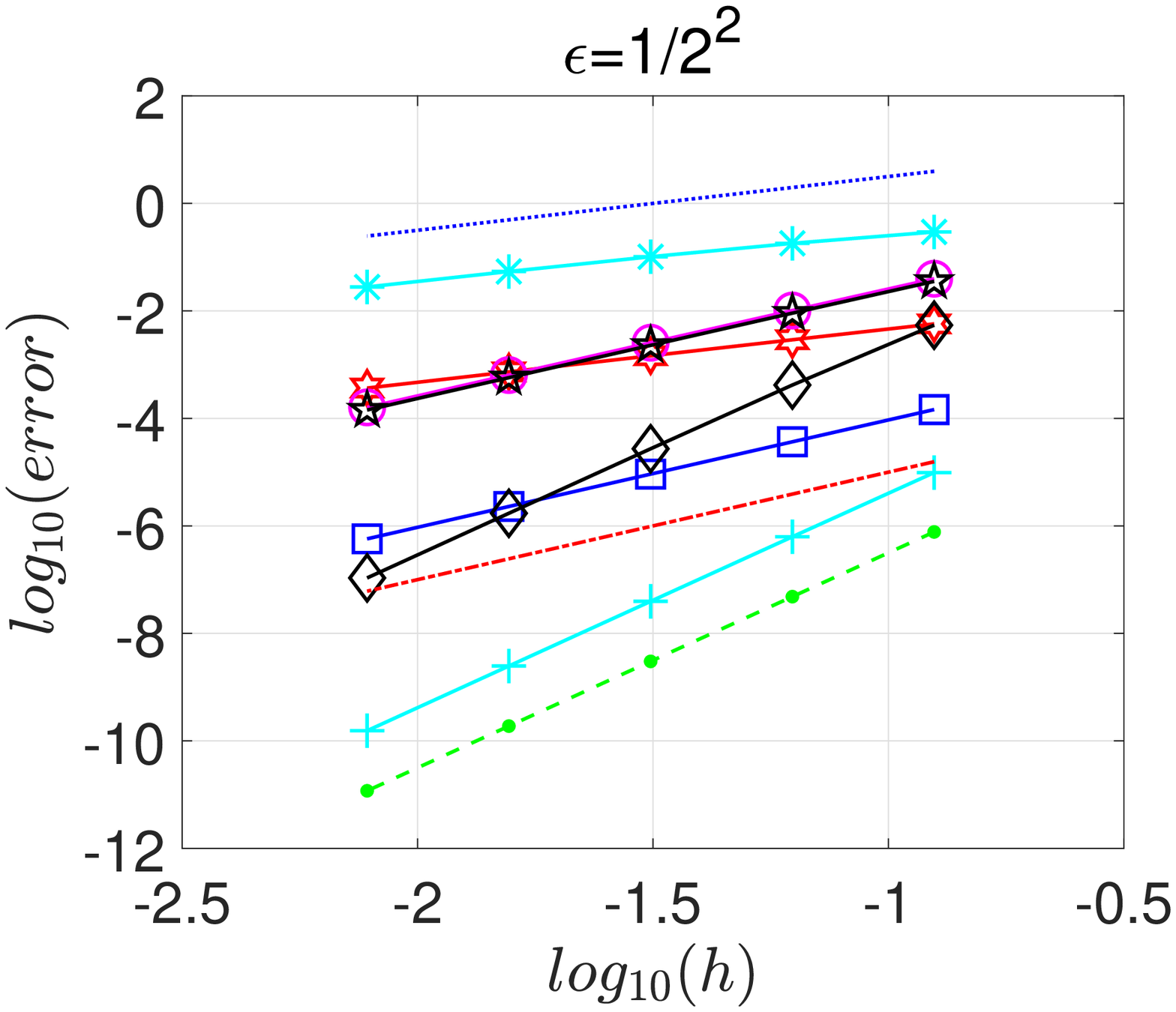}&\includegraphics[width=4.7cm,height=4.5cm]{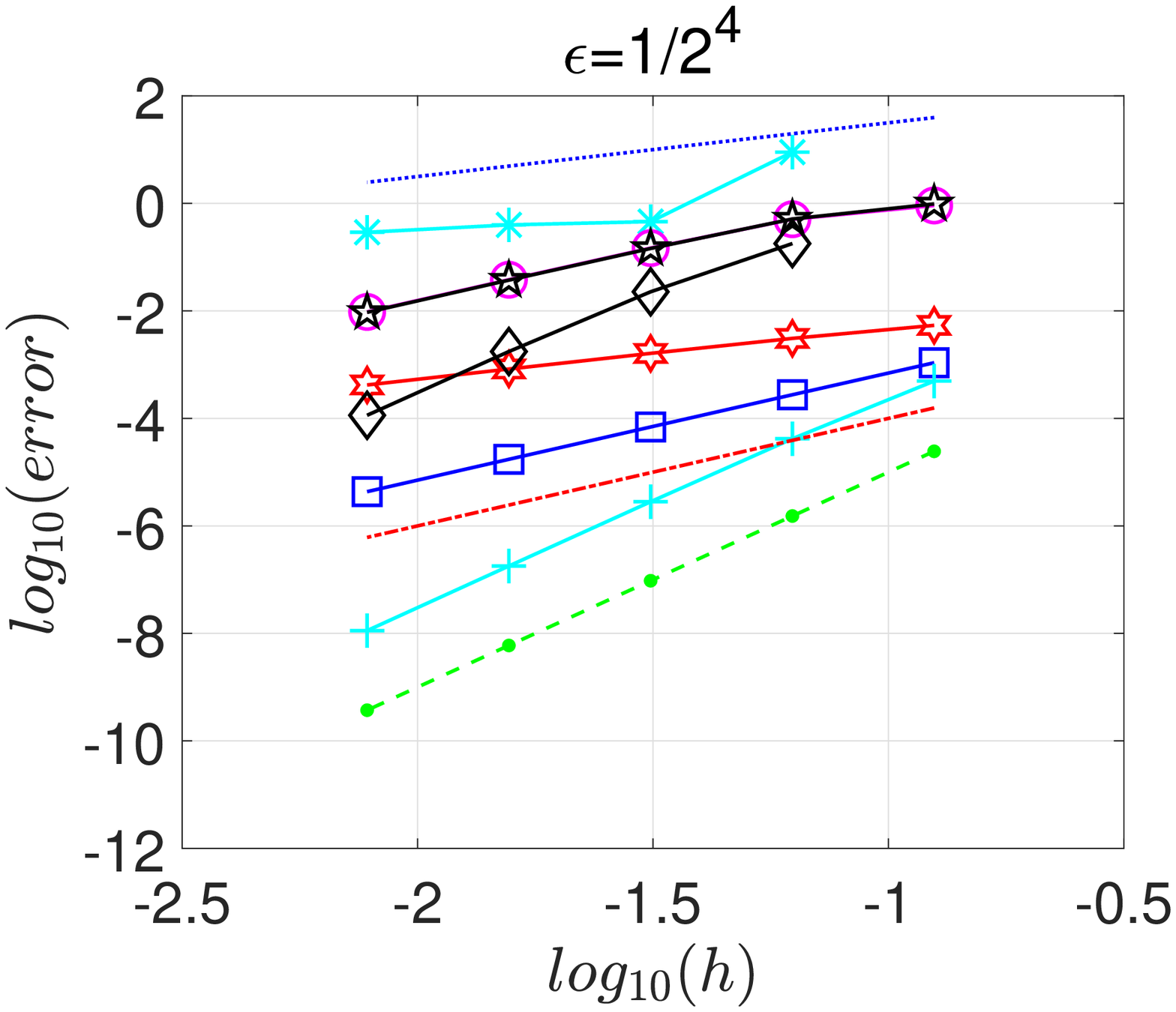}
\end{tabular}
\caption{Problem 2. The global errors $error:=\frac{\norm{x_{n}-x(t_n)}}{\norm{x(t_n)}}+\frac{ \norm{v_{n}-v(t_n)}}{\norm{v(t_n)}}$ with $t=1$ and $h=1/2^{k}$ for $k=3,4,\ldots,7$ under different $\epsilon$.}
\label{fig:problem31}
\end{figure}

\begin{figure}[t!]
\centering\tabcolsep=0.4mm
\begin{tabular}
[c]{cccc}%
 \includegraphics[width=4.7cm,height=4.5cm]{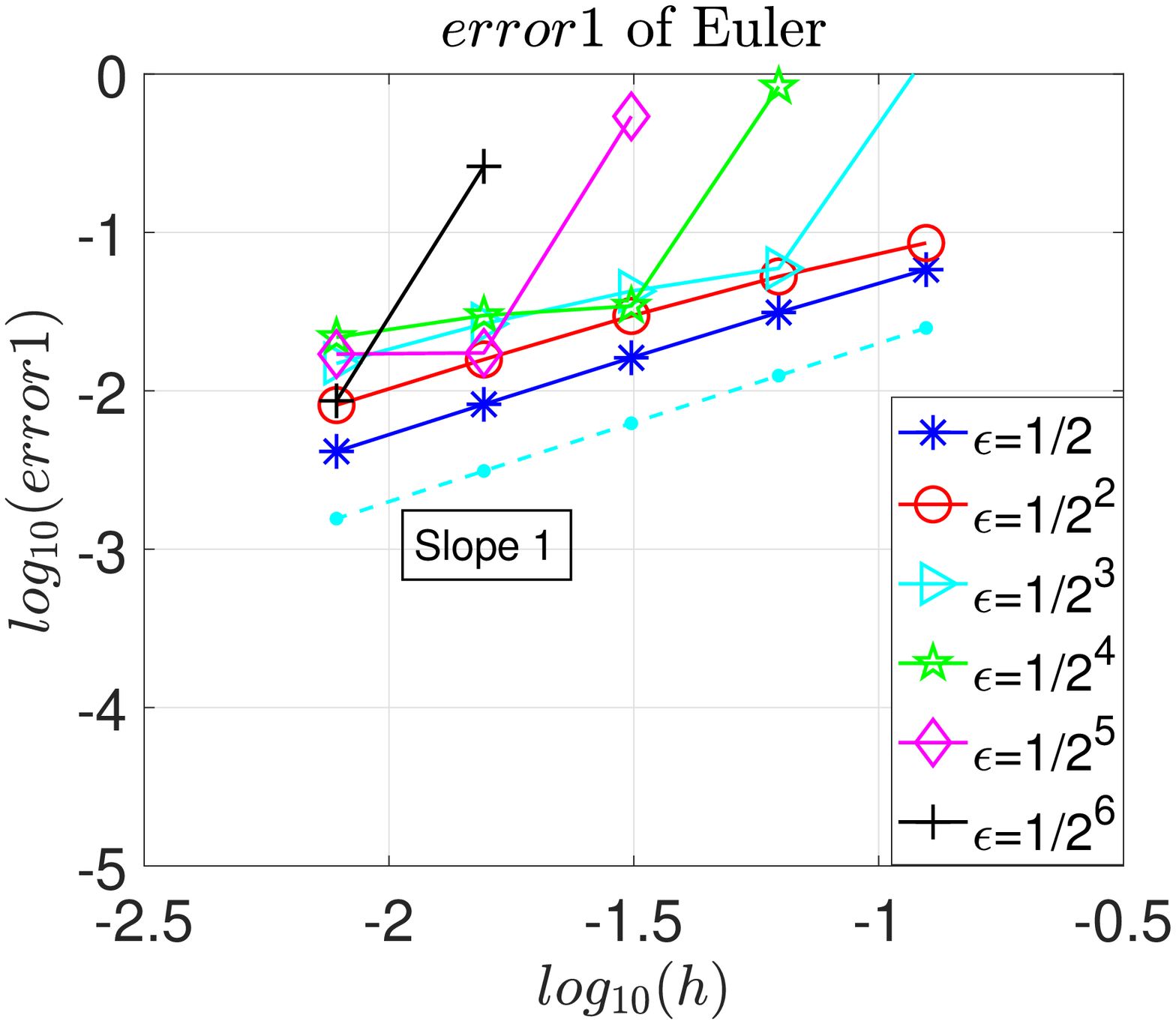} &\includegraphics[width=4.7cm,height=4.5cm]{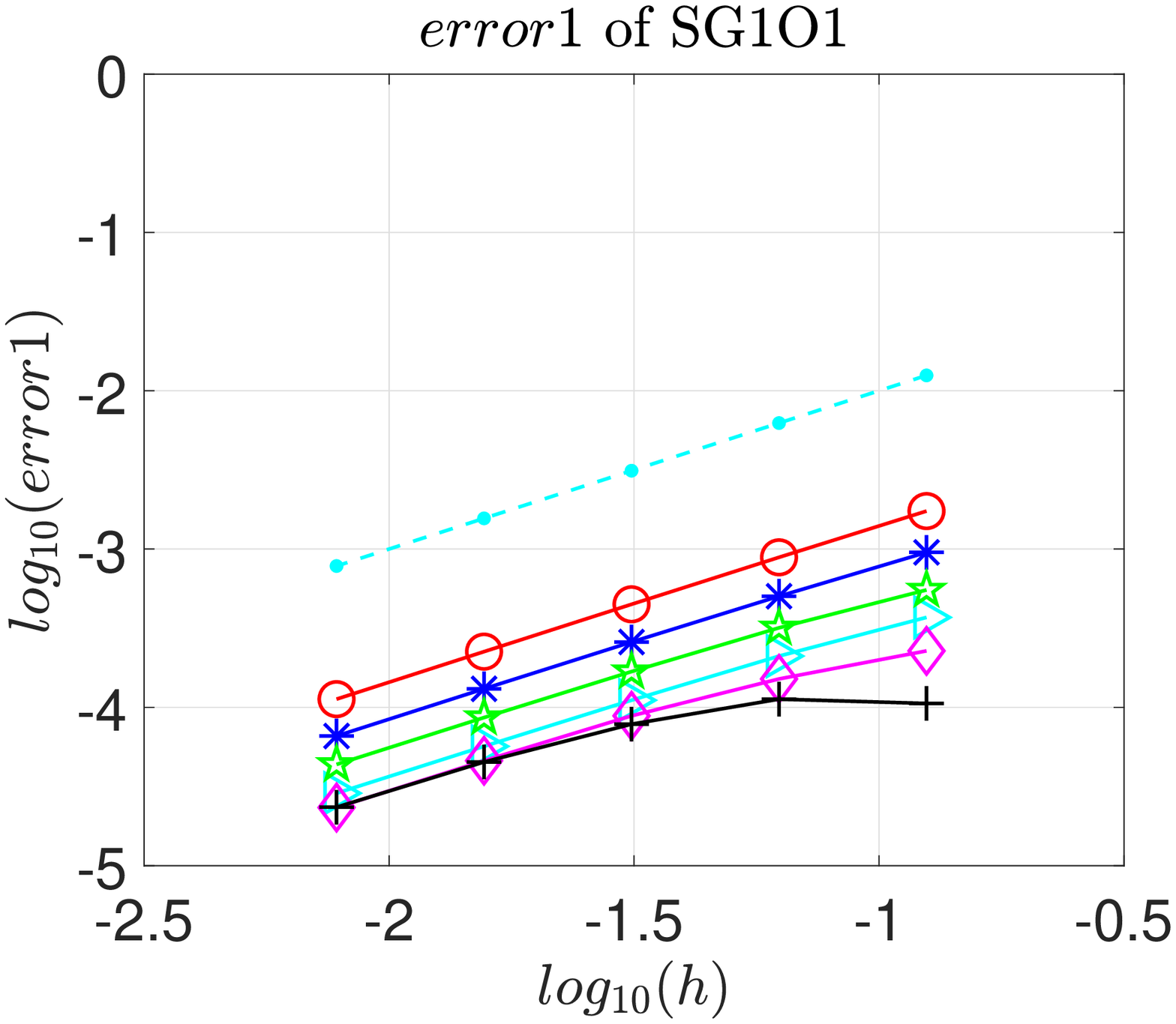}
\end{tabular}
\caption{Problem 2. The errors $error1:=\frac{\norm{x_{n}-x(t_n)}}{\norm{x(t_n)}}+\frac{ \eps \norm{v_{n}-v(t_n)}}{\norm{v(t_n)}}$ of first order methods (Euler and SG1O1) with $t=1$ and $h=1/2^{k}$ for $k=3,4,\ldots,7$ under different $\epsilon$.}
\label{fig:problem31new3}
\end{figure}

\begin{figure}[t!]
\centering\tabcolsep=0.4mm
\begin{tabular}
[c]{ccc}%
\includegraphics[width=4.7cm,height=4.5cm]{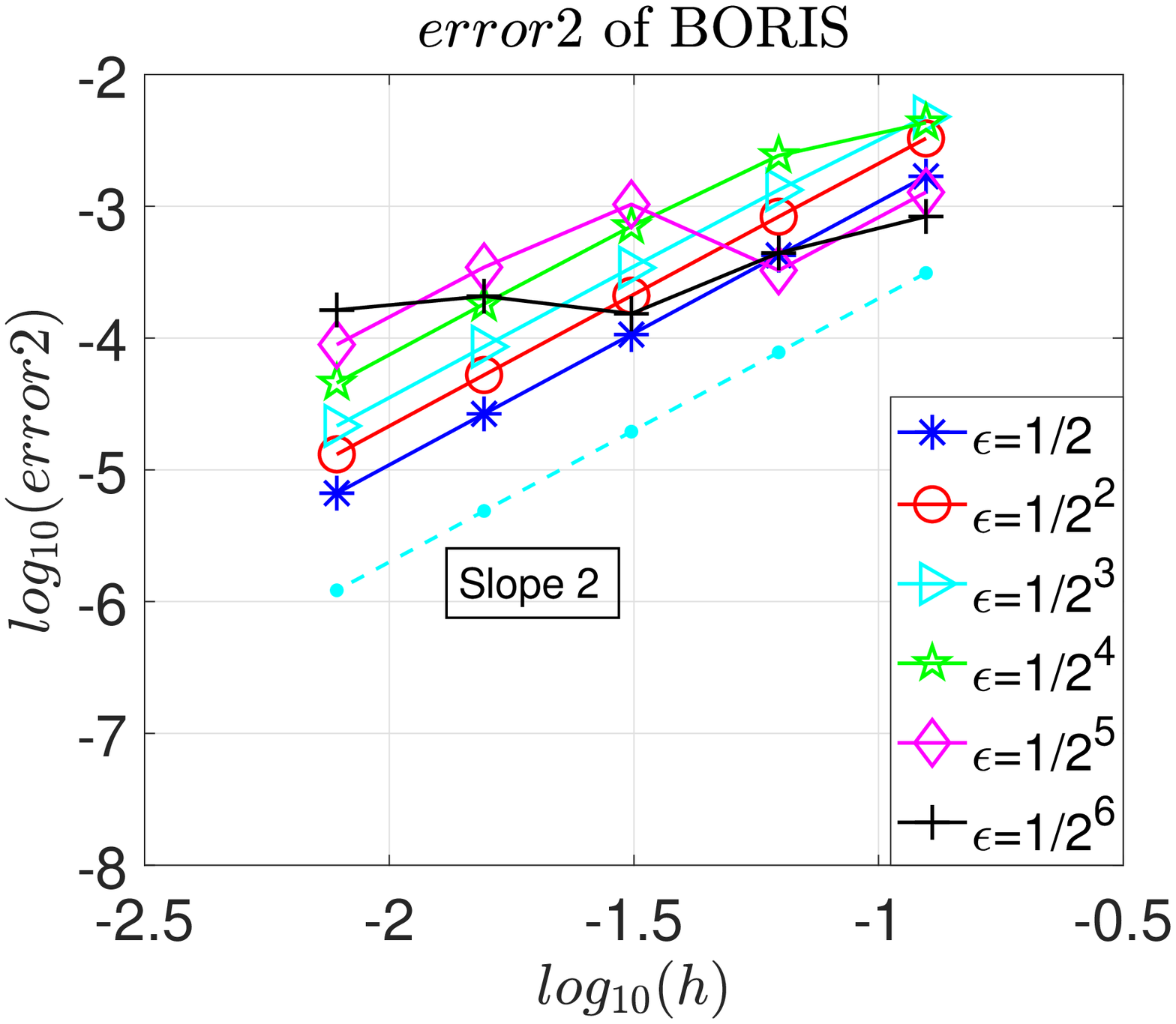} &\includegraphics[width=4.7cm,height=4.5cm]{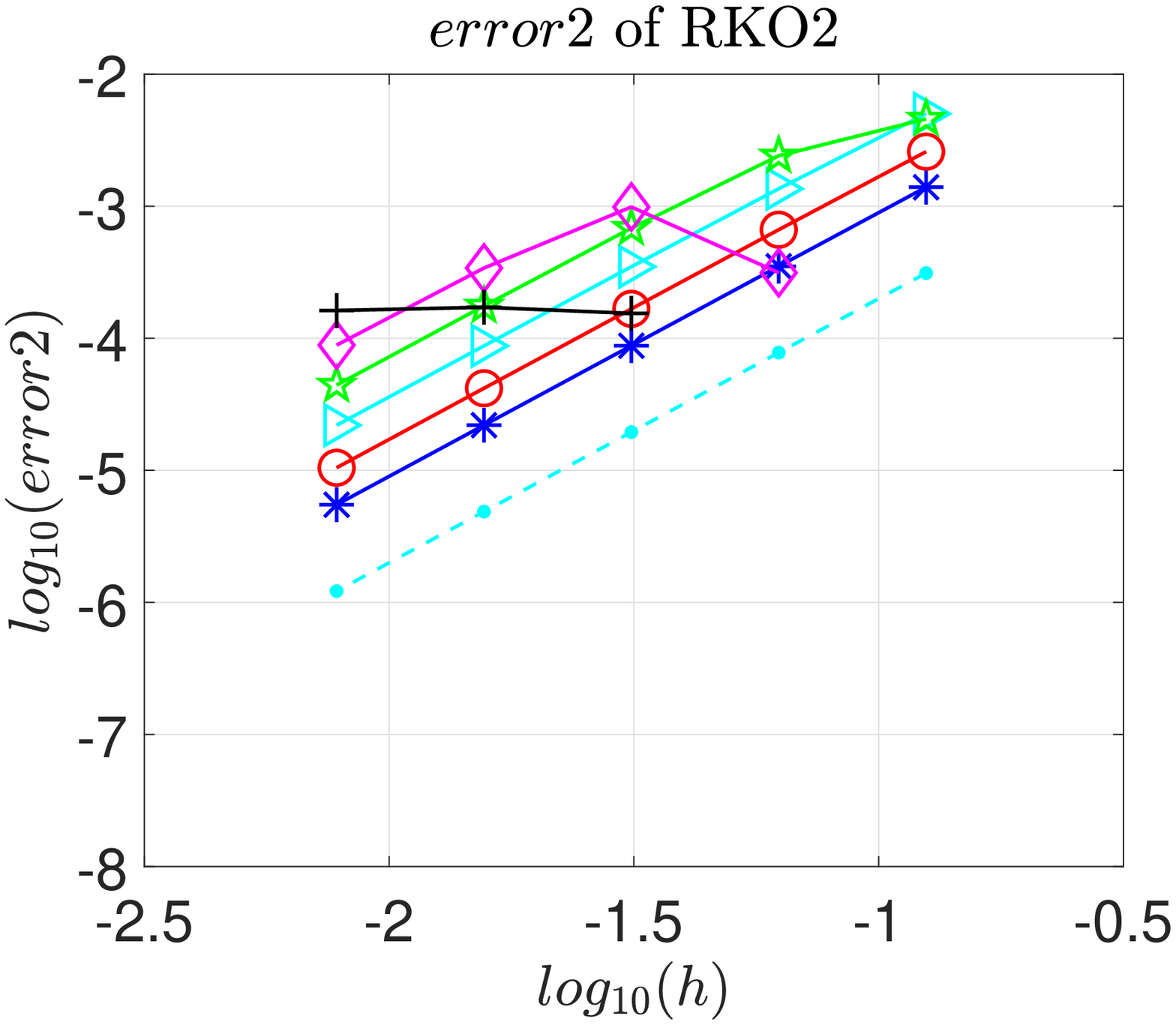} & \includegraphics[width=4.7cm,height=4.5cm]{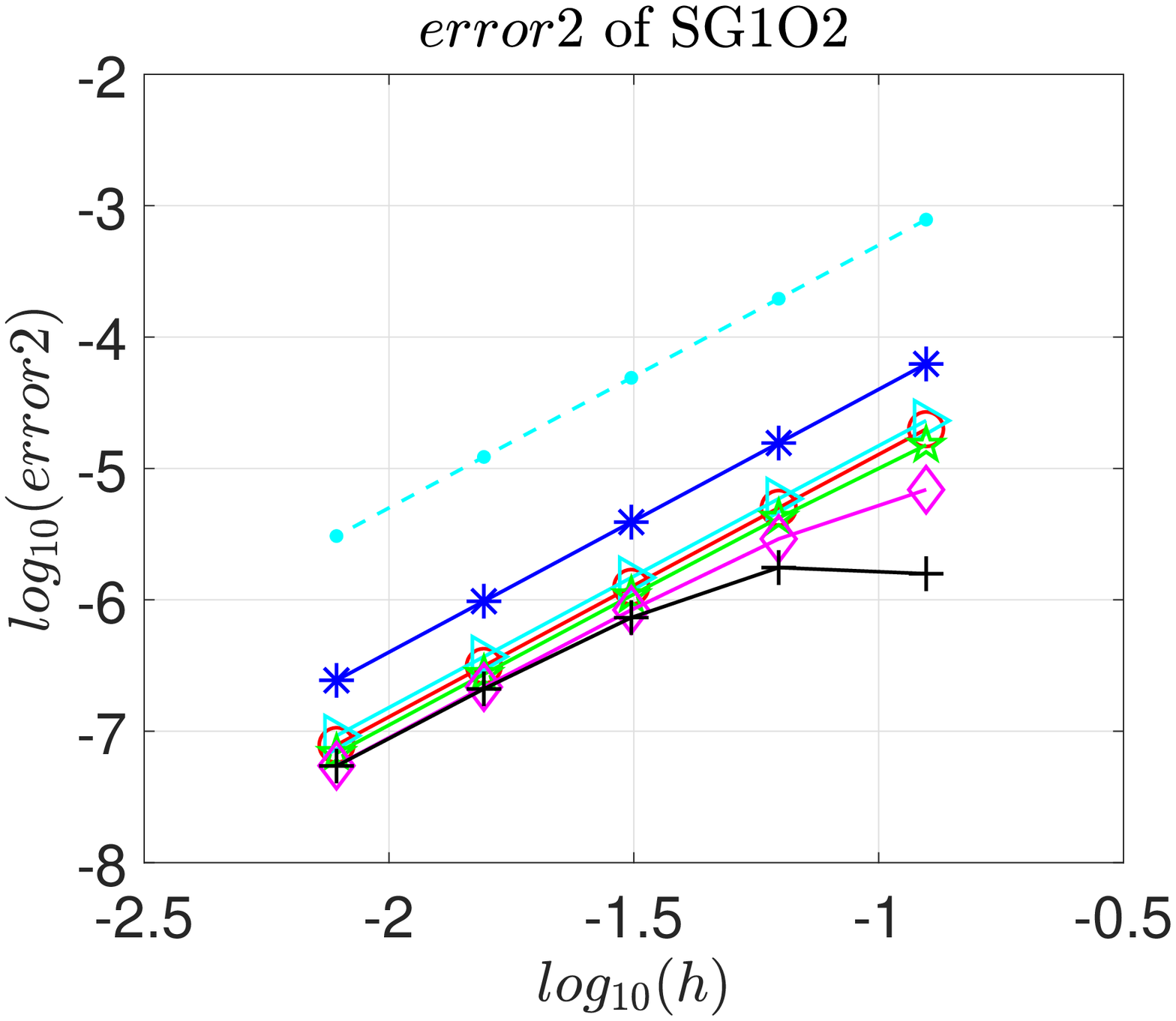}
\end{tabular}
\caption{Problem 2. The errors $error2:=\frac{\eps\norm{x_{n}-x(t_n)}}{\norm{x(t_n)}}+\frac{\eps^2 \norm{v_{n}-v(t_n)}}{\norm{v(t_n)}}$ of second order methods (BORIS, RKO2 and SG1O2) with $t=1$ and $h=1/2^{k}$ for $k=3,4,\ldots,7$ under different $\epsilon$. }
\label{fig:problem31new1}
\end{figure}

\begin{figure}[t!]
\centering\tabcolsep=0.4mm
\begin{tabular}
[c]{cccc}%
 \includegraphics[width=4.7cm,height=4.5cm]{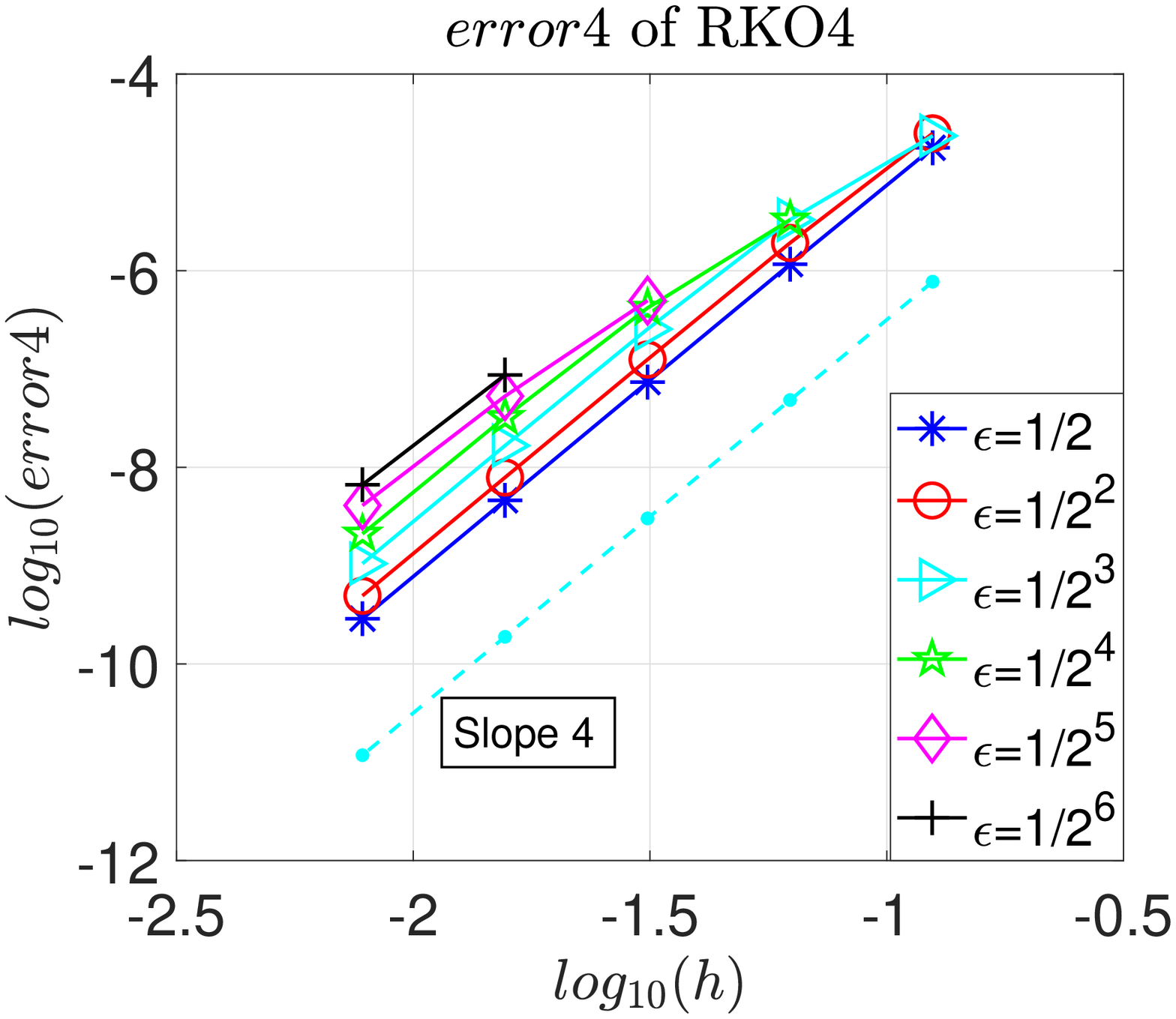} &\includegraphics[width=4.7cm,height=4.5cm]{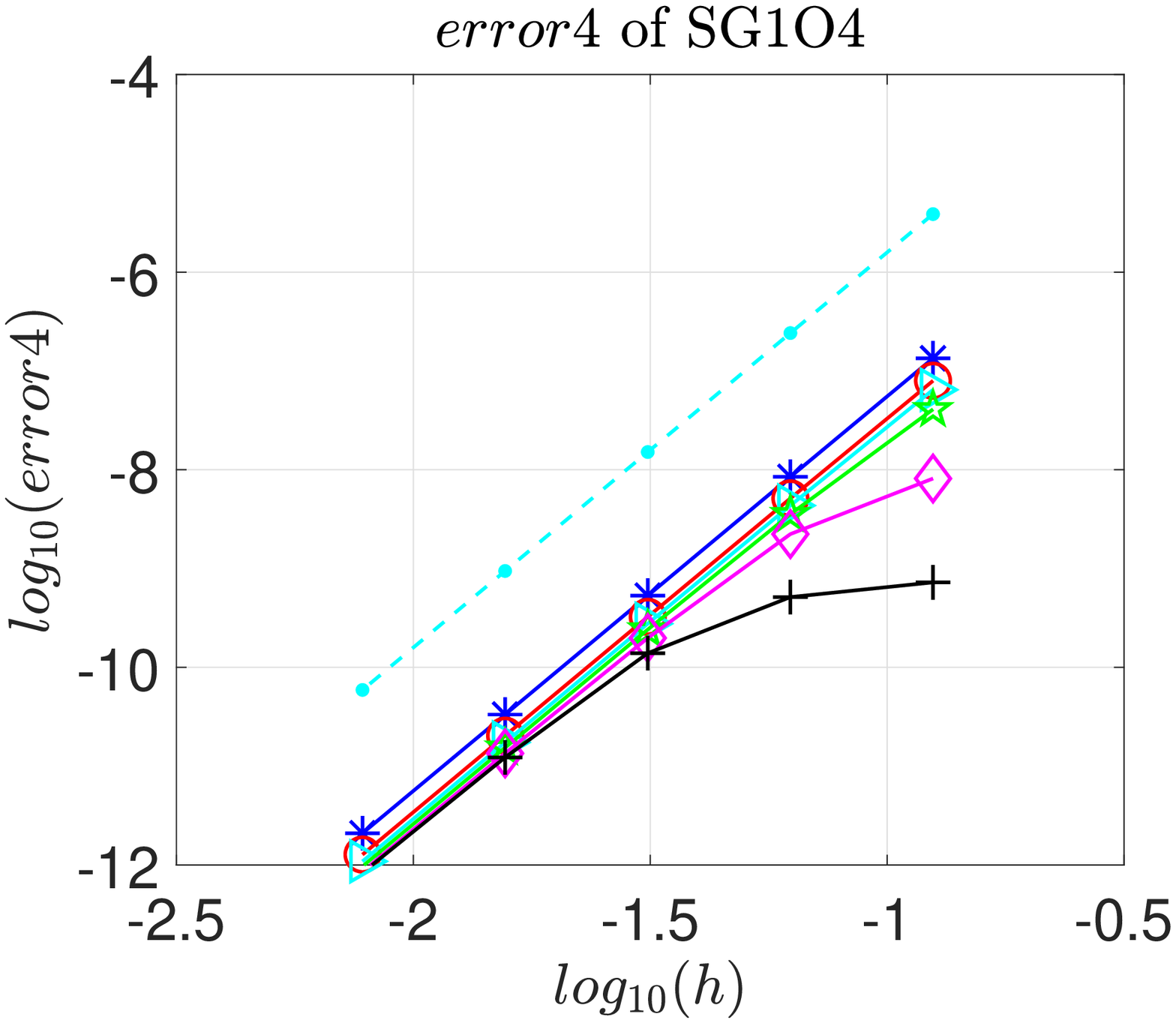}
\end{tabular}
\caption{Problem 2. The errors $error4:=\frac{\eps^3\norm{x_{n}-x(t_n)}}{\norm{x(t_n)}}+\frac{\eps^4 \norm{v_{n}-v(t_n)}}{\norm{v(t_n)}}$ of fourth order methods (RKO4 and SG1O4) with $t=1$ and $h=1/2^{k}$ for $k=3,4,\ldots,7$ under different $\epsilon$. }
\label{fig:problem31new2}
\end{figure}

\begin{figure}[t!]
\centering\tabcolsep=0.4mm
\begin{tabular}
[c]{ccc}%
\includegraphics[width=4.7cm,height=4.5cm]{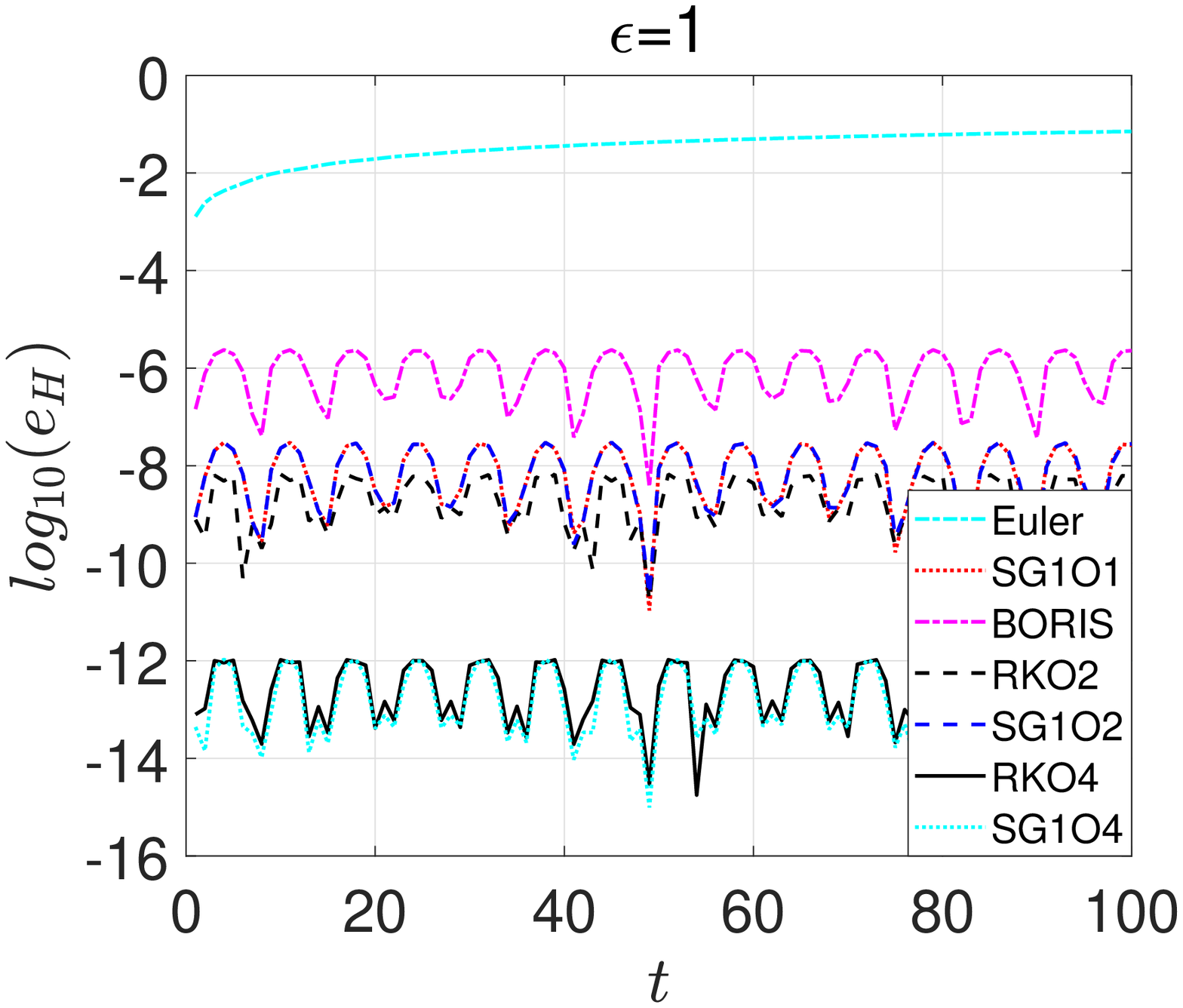} & \includegraphics[width=4.7cm,height=4.5cm]{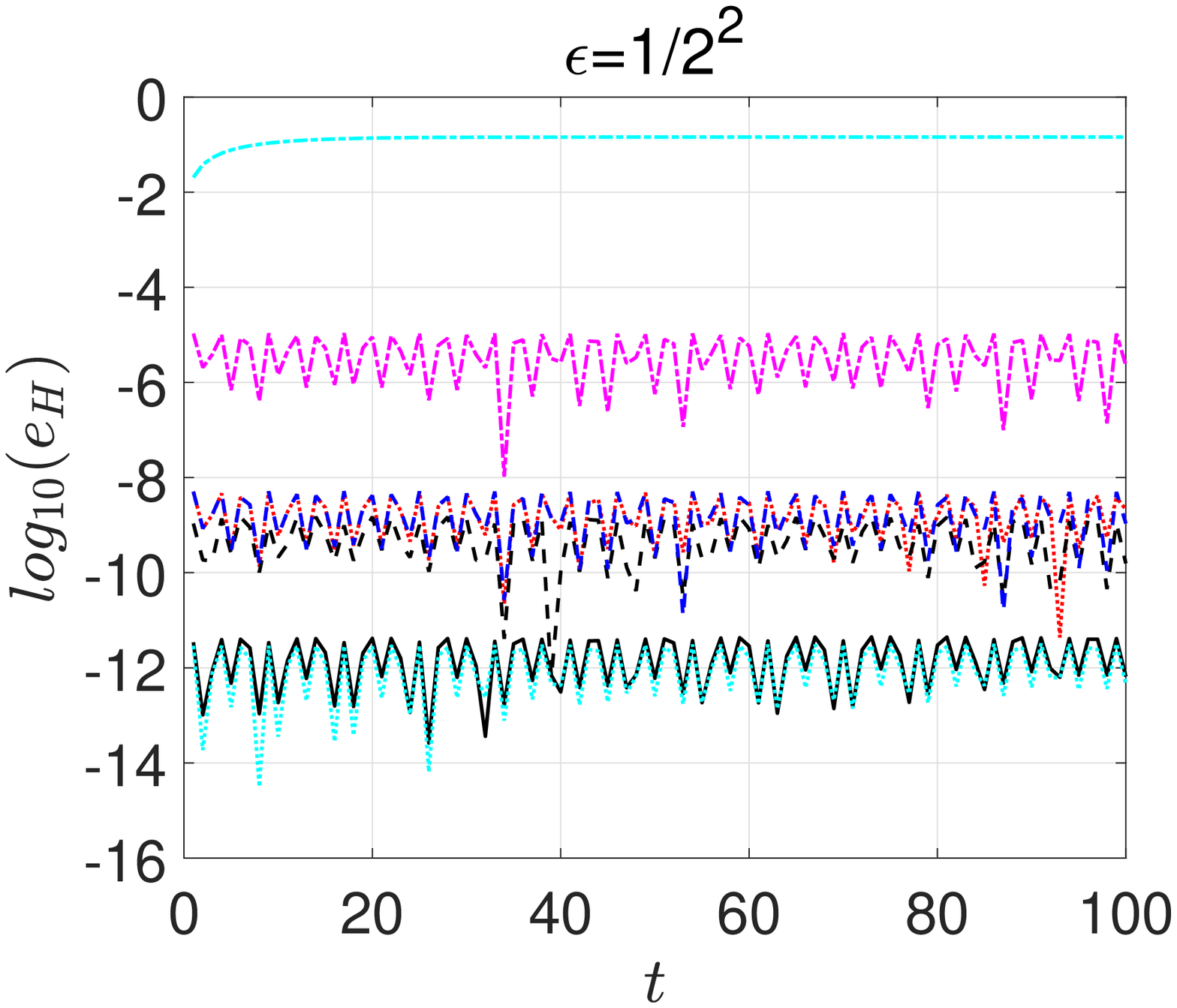} & \includegraphics[width=4.7cm,height=4.5cm]{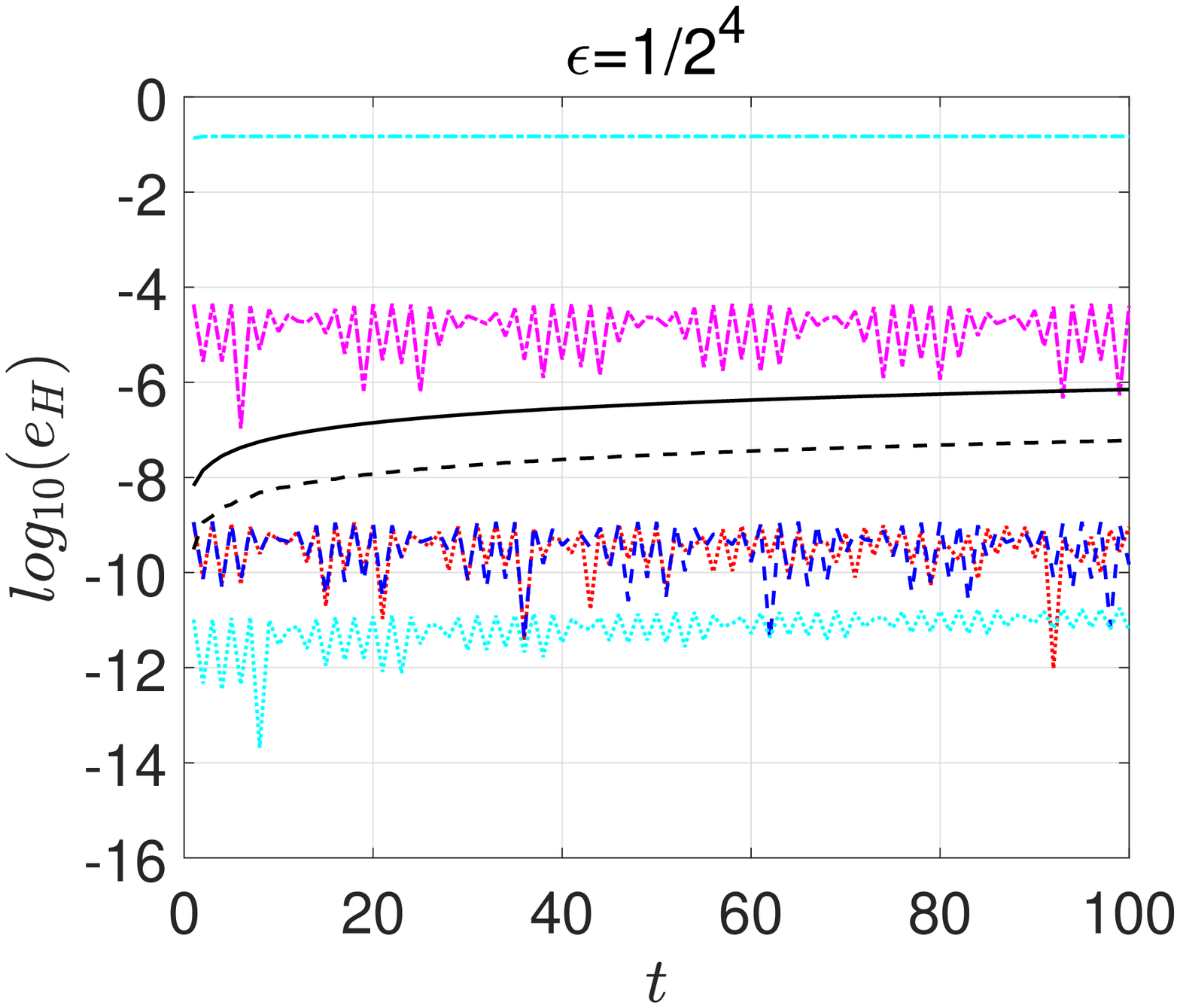}
\end{tabular}
\caption{Problem 2.  Evolution of the energy error $e_{H}:=\frac{|H(x_{n},v_n)-H(x_0,v_0)|}{|H(x_0,v_0)|}$ as function of time  $t=nh$.}
\label{fig:problem32}
\end{figure}

%\begin{figure}[t!]
%\centering\tabcolsep=0.4mm
%\begin{tabular}
%[c]{cccc}%
%\includegraphics[width=3.8cm,height=4.5cm]{Boris-3tra} &\includegraphics[width=3.8cm,height=4.5cm]{Euler-3tra} & \includegraphics[width=3.8cm,height=4.5cm]{SG1O1-3tra} & \includegraphics[width=3.8cm,height=4.5cm]{SG1O2-3tra}
%\end{tabular}
%\caption{Problem 1.  The trajectory of second order methods (BORIS, RKO2, SG1O1 and SG1O2) in [x y] space with $t=100$, $h=1/4$ and $\epsilon=0.01$. }
%\label{fig:problem11new3}
%\end{figure}

\noindent\vskip3mm \noindent\textbf{Problem 3. (Maximal ordering scaling)}
Consider  the charged-particle dynamics \eqref{charged-particle sts-cons} from \cite{WZ} in a magnetic field with the maximal ordering scaling
  $\frac{1}{\epsilon}B(\epsilon x)=\frac{1}{\epsilon}\left(
                   \begin{array}{c}
                     \cos(\epsilon x_2) \\
                      1+\sin(\epsilon x_3) \\
                    \cos(\epsilon x_1) \\
                   \end{array}
                 \right)+\left(
                   \begin{array}{c}
                     -x_1 \\
                      0 \\
                    x_3 \\
                   \end{array}
                 \right)$
 and  the scalar potential $U(x)=\frac{1}{\sqrt{x_{1}^2+x_{2}^2}}.$
  We take $x(0)=(1/3, 1/4, 1/2)^{\intercal}, \ v(0)=(2/5, 2/3, 1)^{\intercal}$ as the initial values.
 {Figs. \ref{fig:problem41}--\ref{fig:problem41new2} present the errors of all the methods} when solving this problem  on $[0,1]$ with different   $\epsilon$  and $h=1/2^{k}$, where $k=3,4,\ldots,7$. Finally,  the energy errors are shown in Fig.
\ref{fig:problem42} on  $[0,100]$   with $h=\frac{1}{20}$.

\begin{figure}[t!]
\centering\tabcolsep=0.4mm
\begin{tabular}
[c]{ccc}%
\includegraphics[width=4.7cm,height=4.5cm]{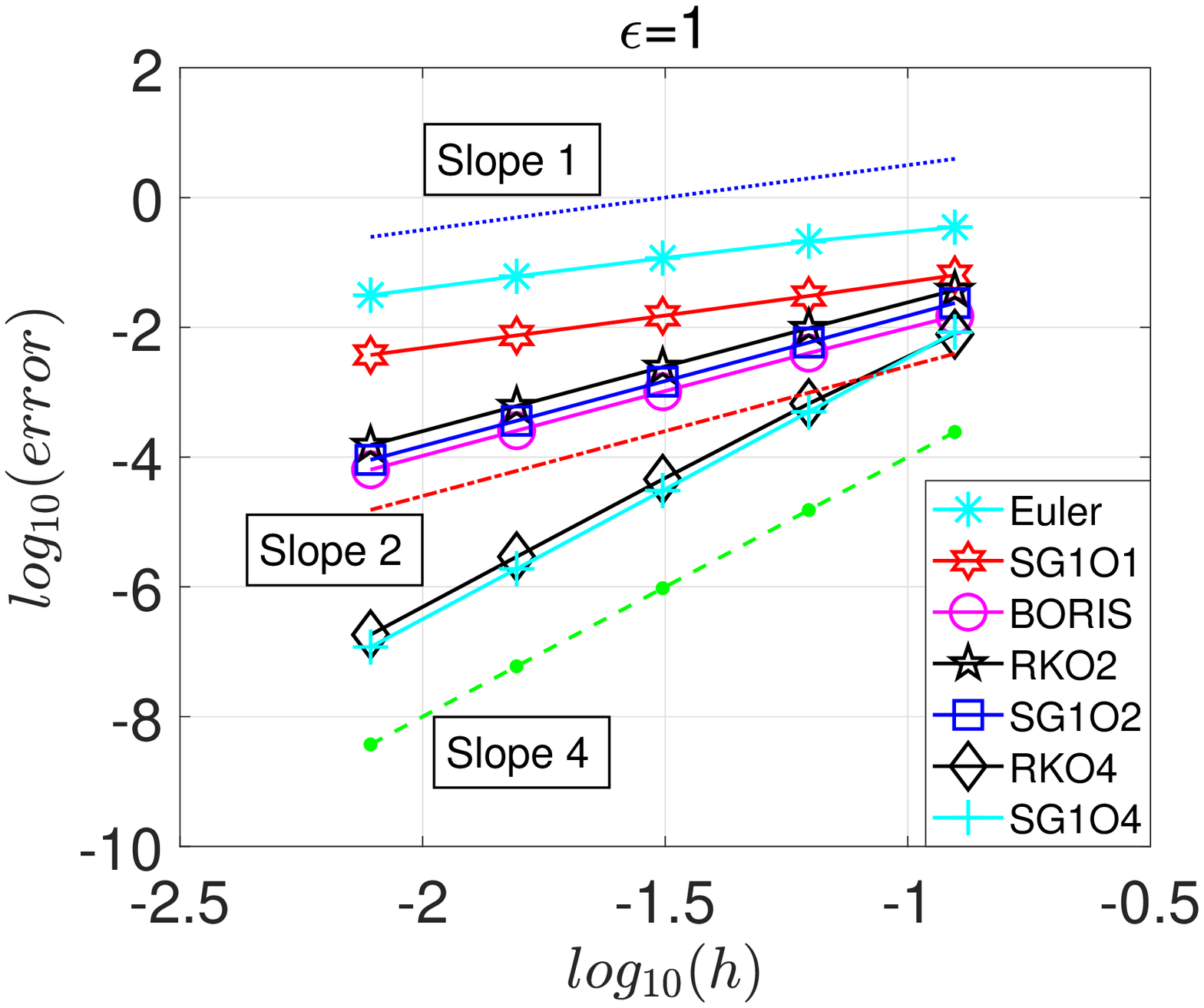} & \includegraphics[width=4.7cm,height=4.5cm]{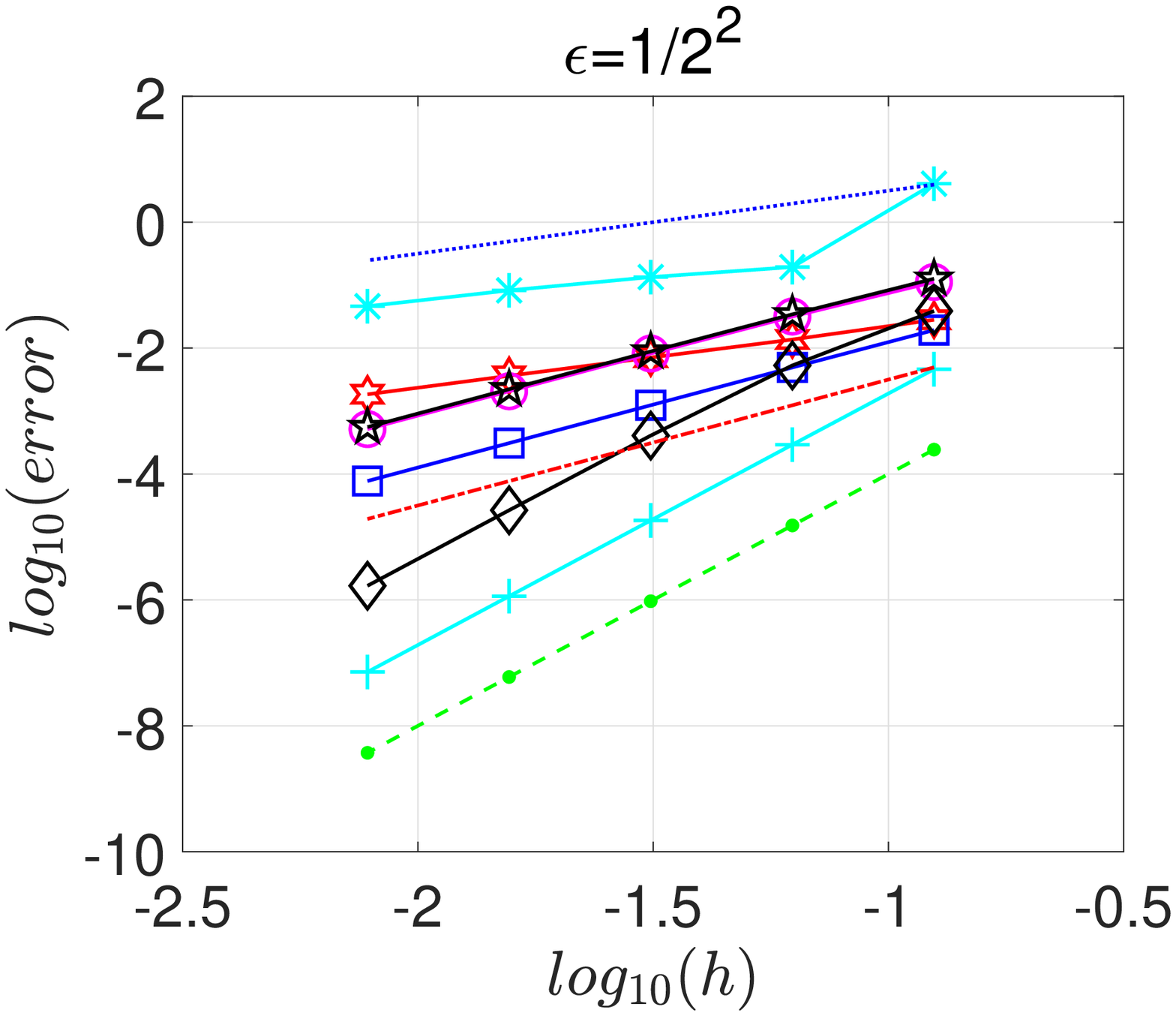} & \includegraphics[width=4.7cm,height=4.5cm]{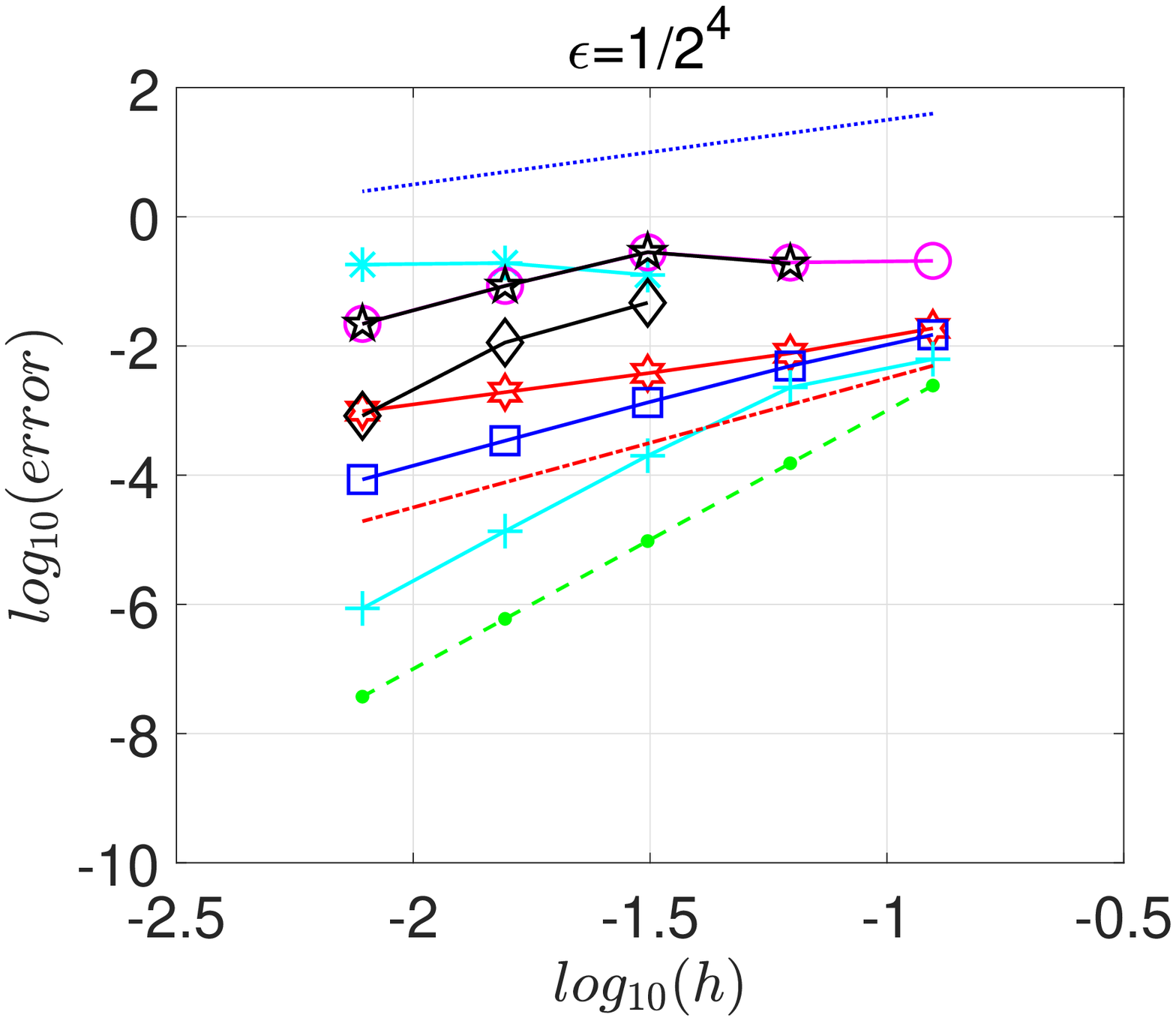}
\end{tabular}
\caption{Problem 3.  The global errors $error:=\frac{\norm{x_{n}-x(t_n)}}{\norm{x(t_n)}}+\frac{ \norm{v_{n}-v(t_n)}}{\norm{v(t_n)}}$ with $t=1$ and $h=1/2^{k}$ for $k=3,4,\ldots,7$ under different $\epsilon$. }
\label{fig:problem41}
\end{figure}

\begin{figure}[t!]
\centering\tabcolsep=0.4mm
\begin{tabular}
[c]{cccc}%
 \includegraphics[width=4.7cm,height=4.5cm]{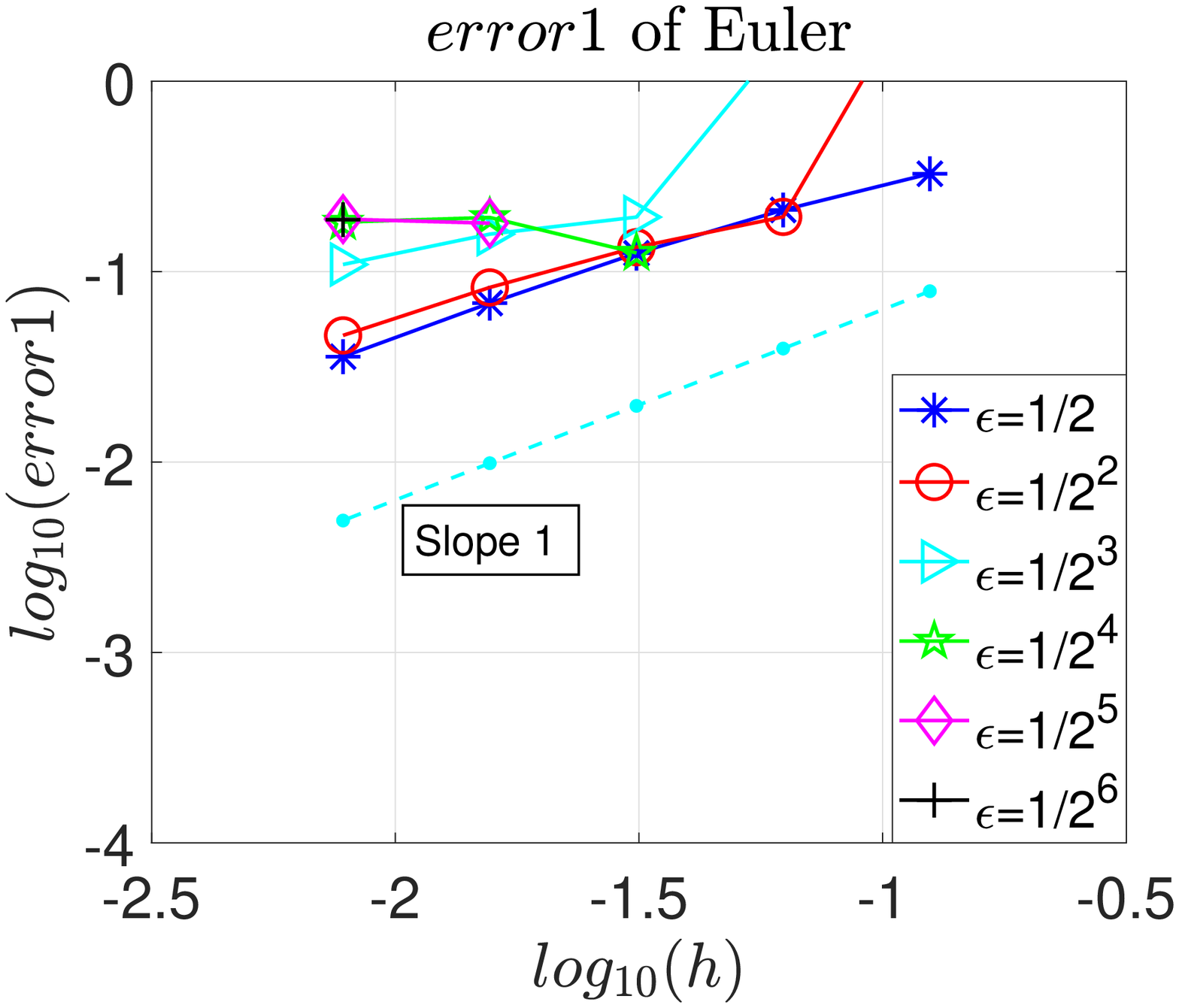} &\includegraphics[width=4.7cm,height=4.5cm]{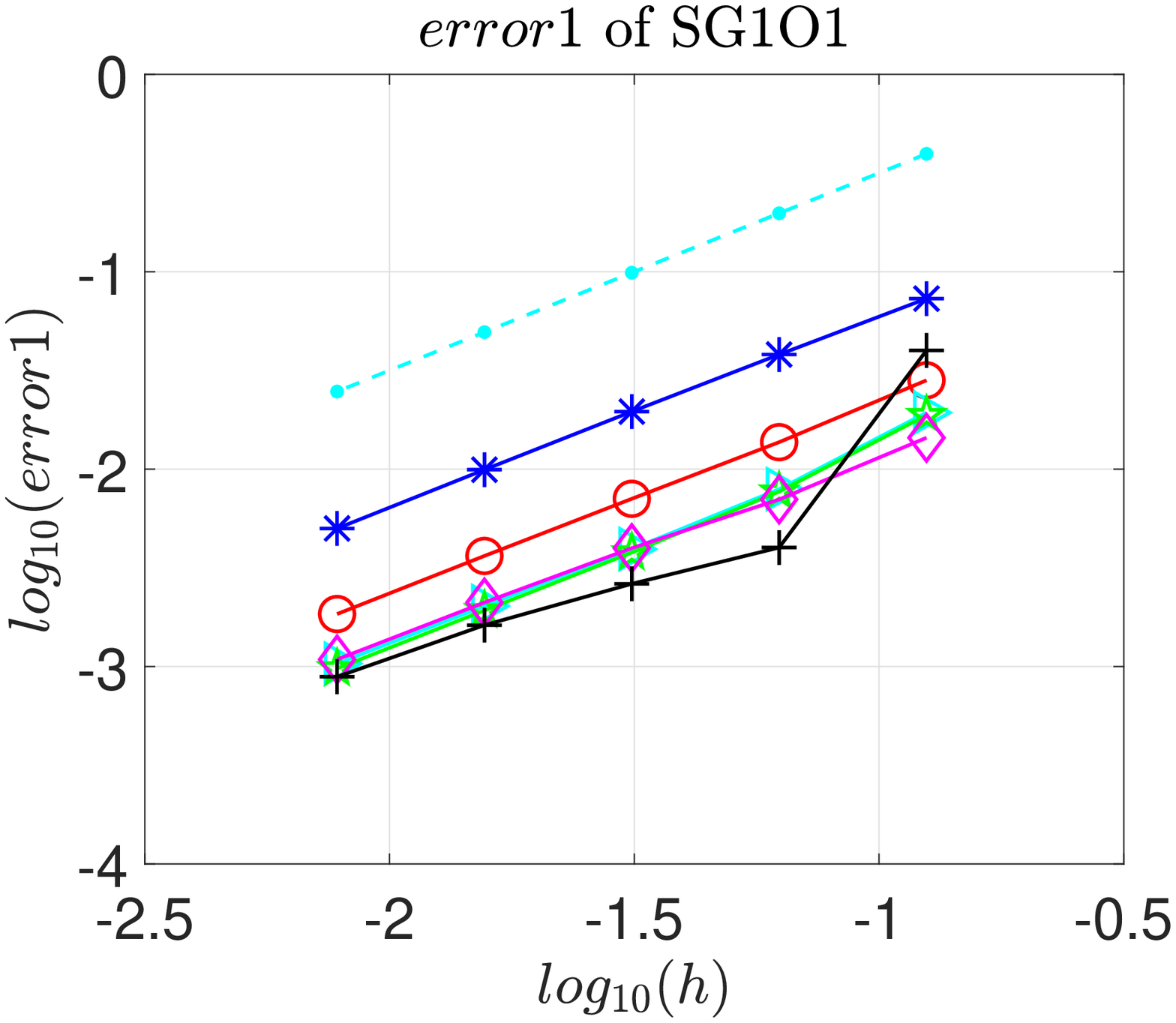}
\end{tabular}
\caption{Problem 3. The errors $error1:=\frac{\norm{x_{n}-x(t_n)}}{\norm{x(t_n)}}+\frac{ \norm{v_{n}-v(t_n)}}{\norm{v(t_n)}}$  of first order methods (Euler and SG1O1) with $t=1$ and $h=1/2^{k}$ for $k=3,4,\ldots,7$ under different $\epsilon$.}
\label{fig:problem41new3}
\end{figure}

\begin{figure}[t!]
\centering\tabcolsep=0.4mm
\begin{tabular}
[c]{ccc}%
\includegraphics[width=4.7cm,height=4.5cm]{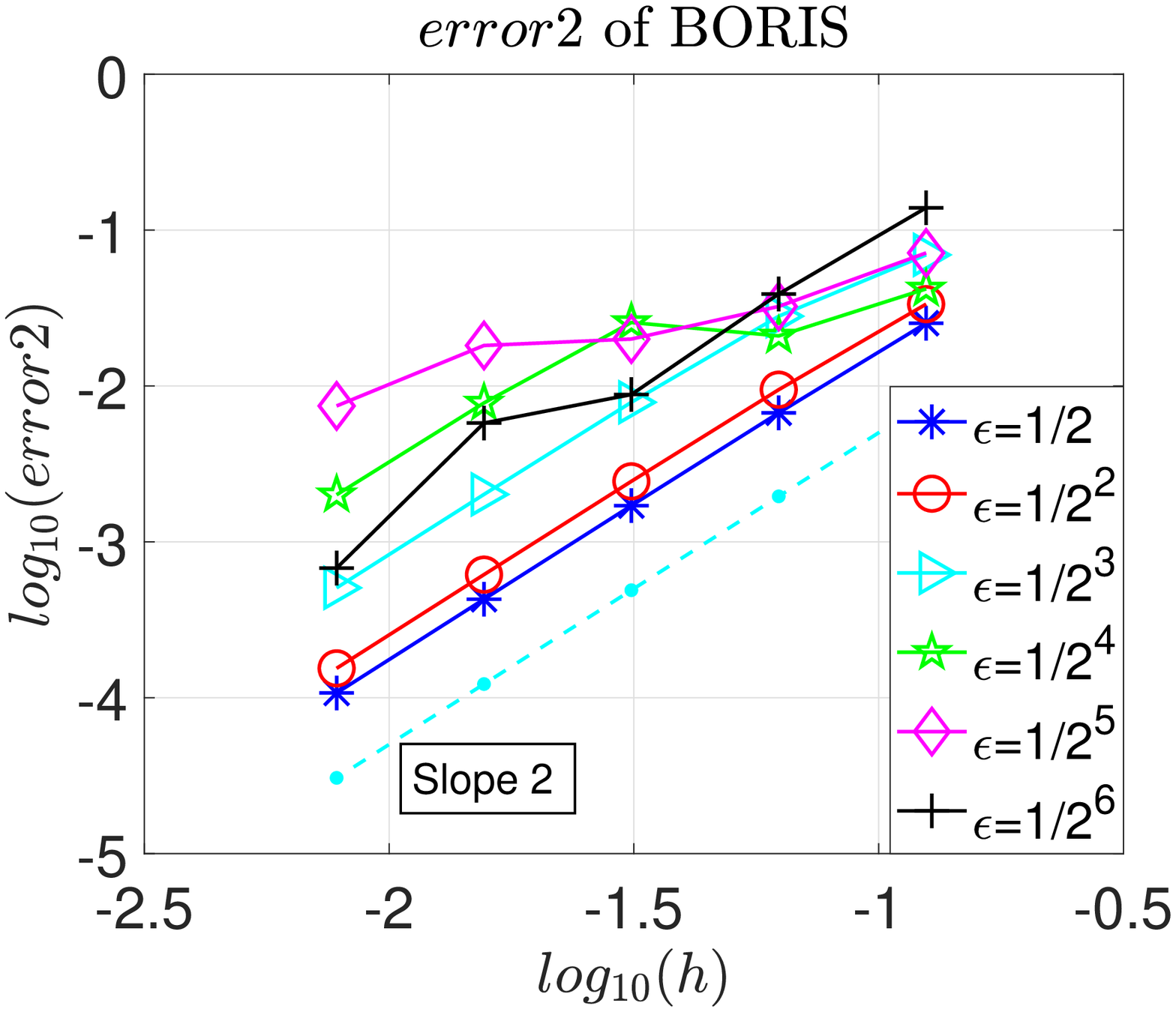} &\includegraphics[width=4.7cm,height=4.5cm]{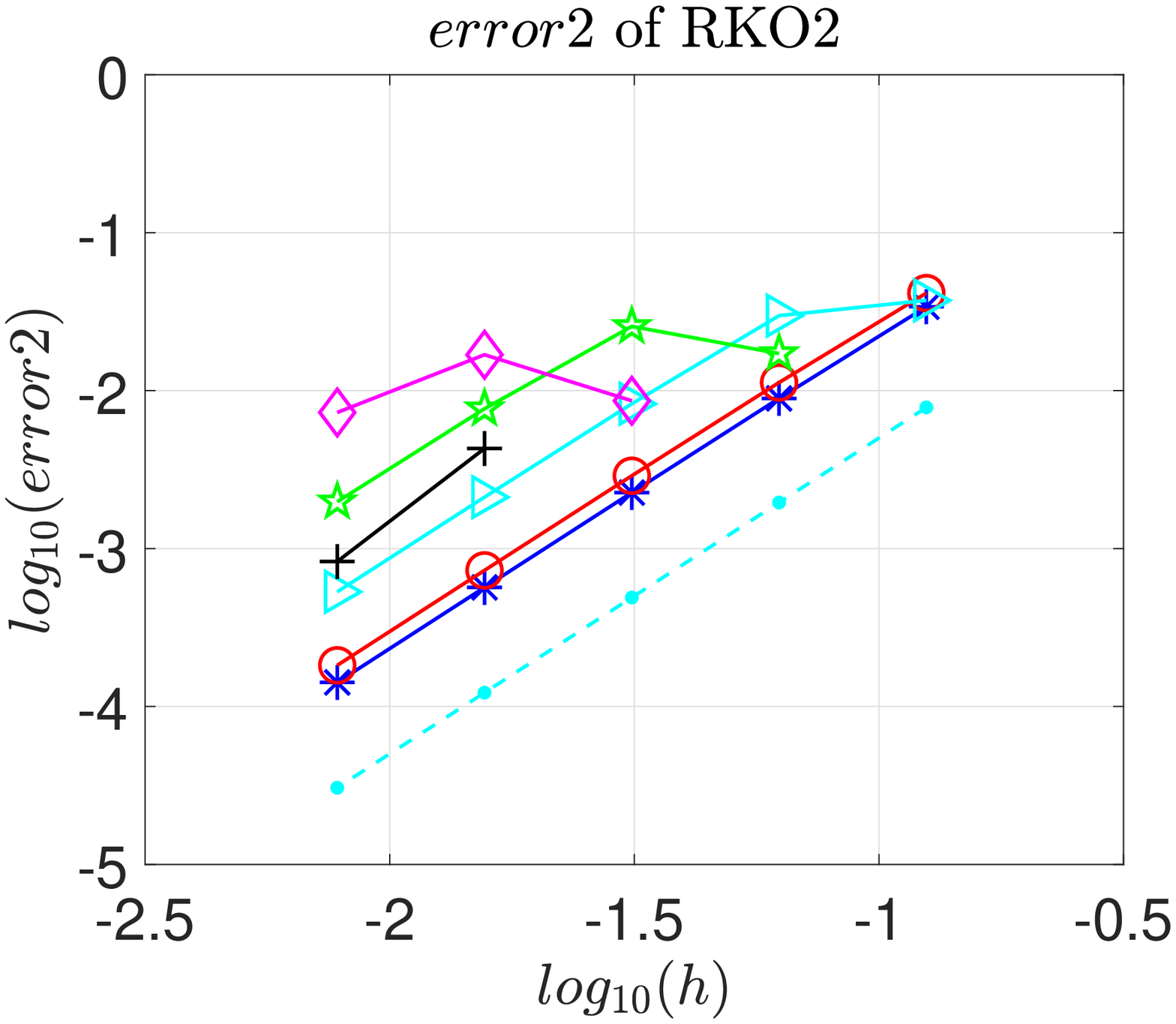} & \includegraphics[width=4.7cm,height=4.5cm]{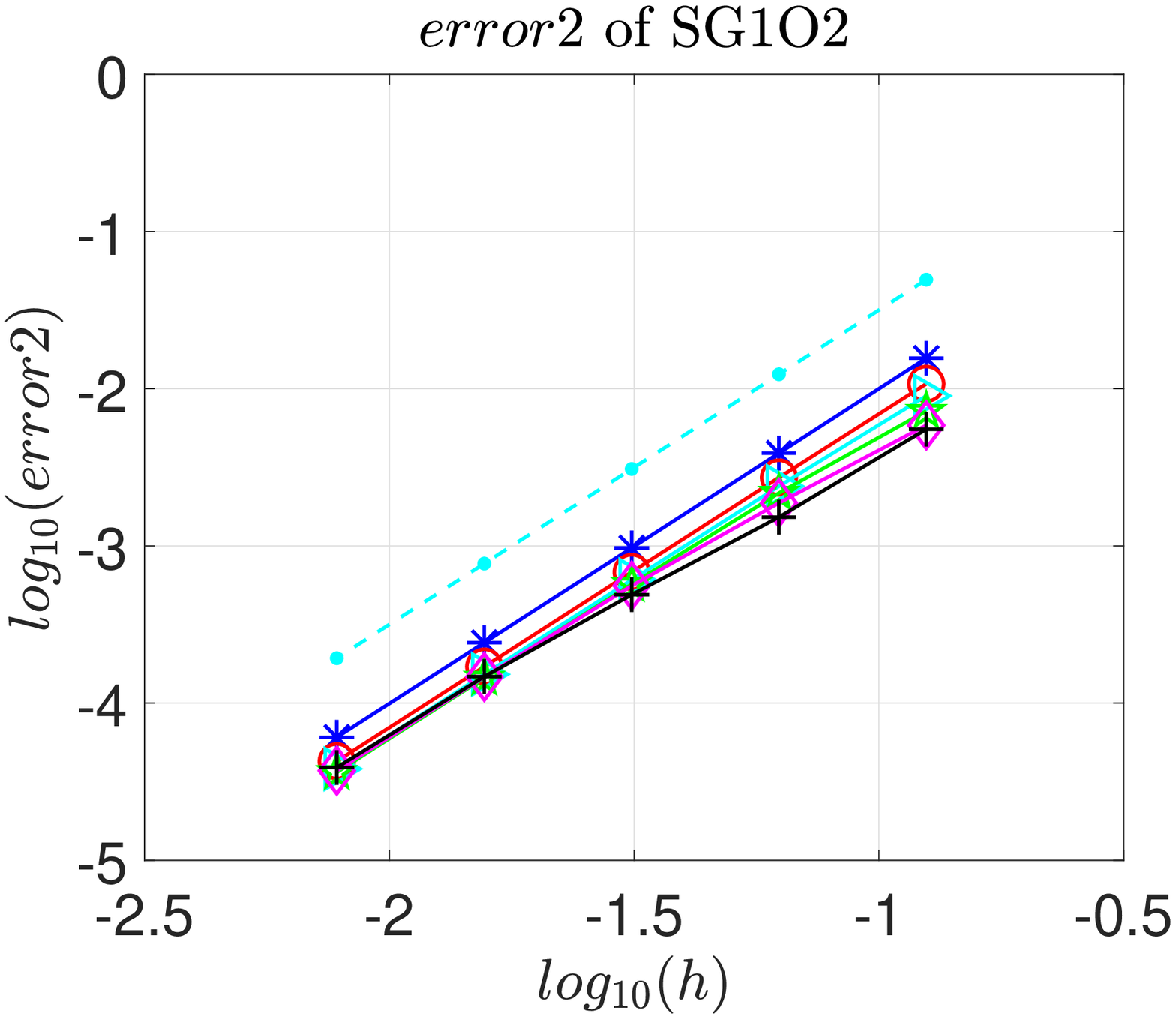}
\end{tabular}
\caption{Problem 3. The errors $error2:=\frac{\norm{x_{n}-x(t_n)}}{\norm{x(t_n)}}+\frac{\eps \norm{v_{n}-v(t_n)}}{\norm{v(t_n)}}$ of second order methods (BORIS, RKO2 and SG1O2) with $t=1$ and $h=1/2^{k}$ for $k=3,4,\ldots,7$ under different $\epsilon$. }
\label{fig:problem41new1}
\end{figure}

\begin{figure}[t!]
\centering\tabcolsep=0.4mm
\begin{tabular}
[c]{cccc}%
 \includegraphics[width=4.7cm,height=4.5cm]{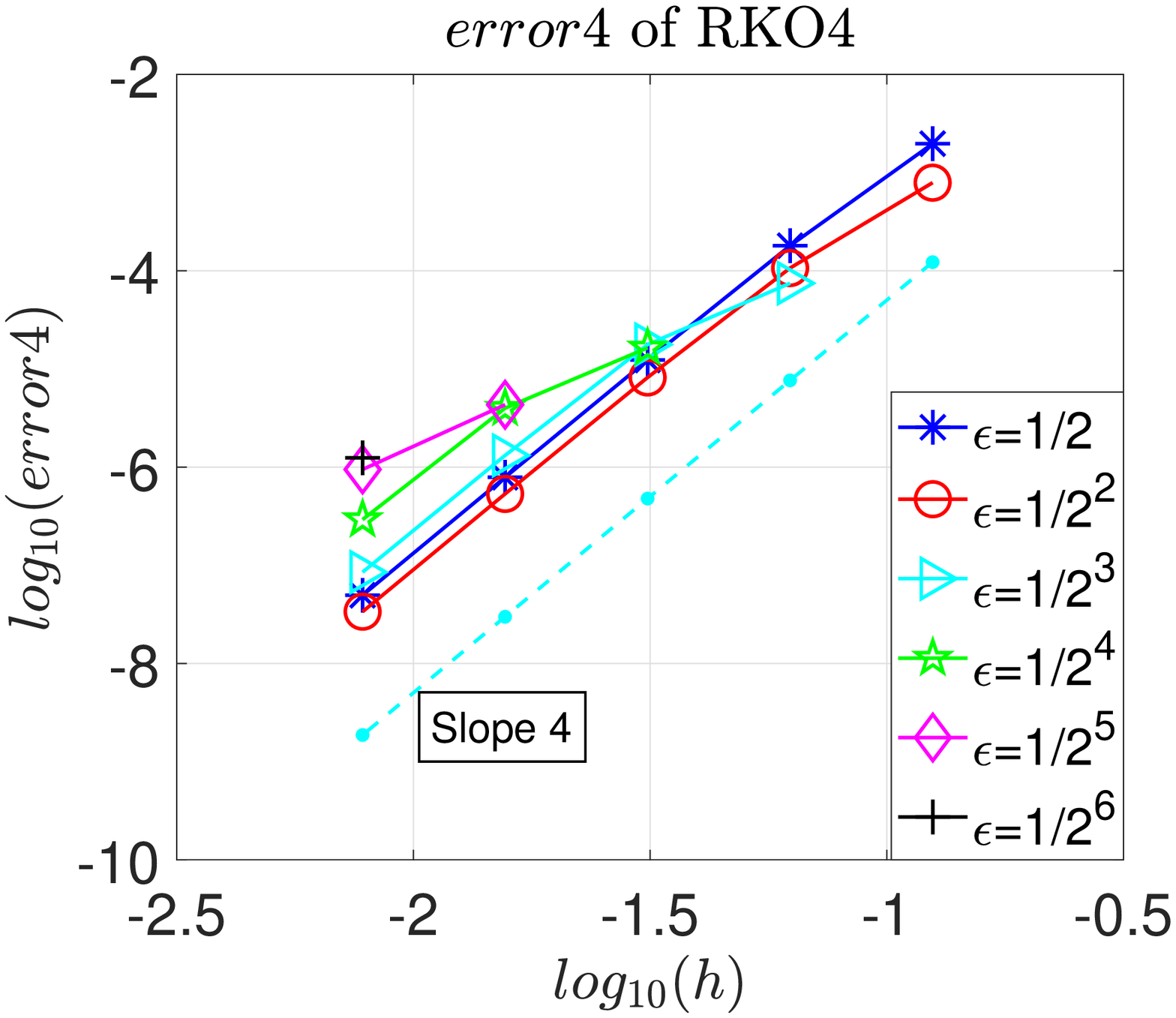} &\includegraphics[width=4.7cm,height=4.5cm]{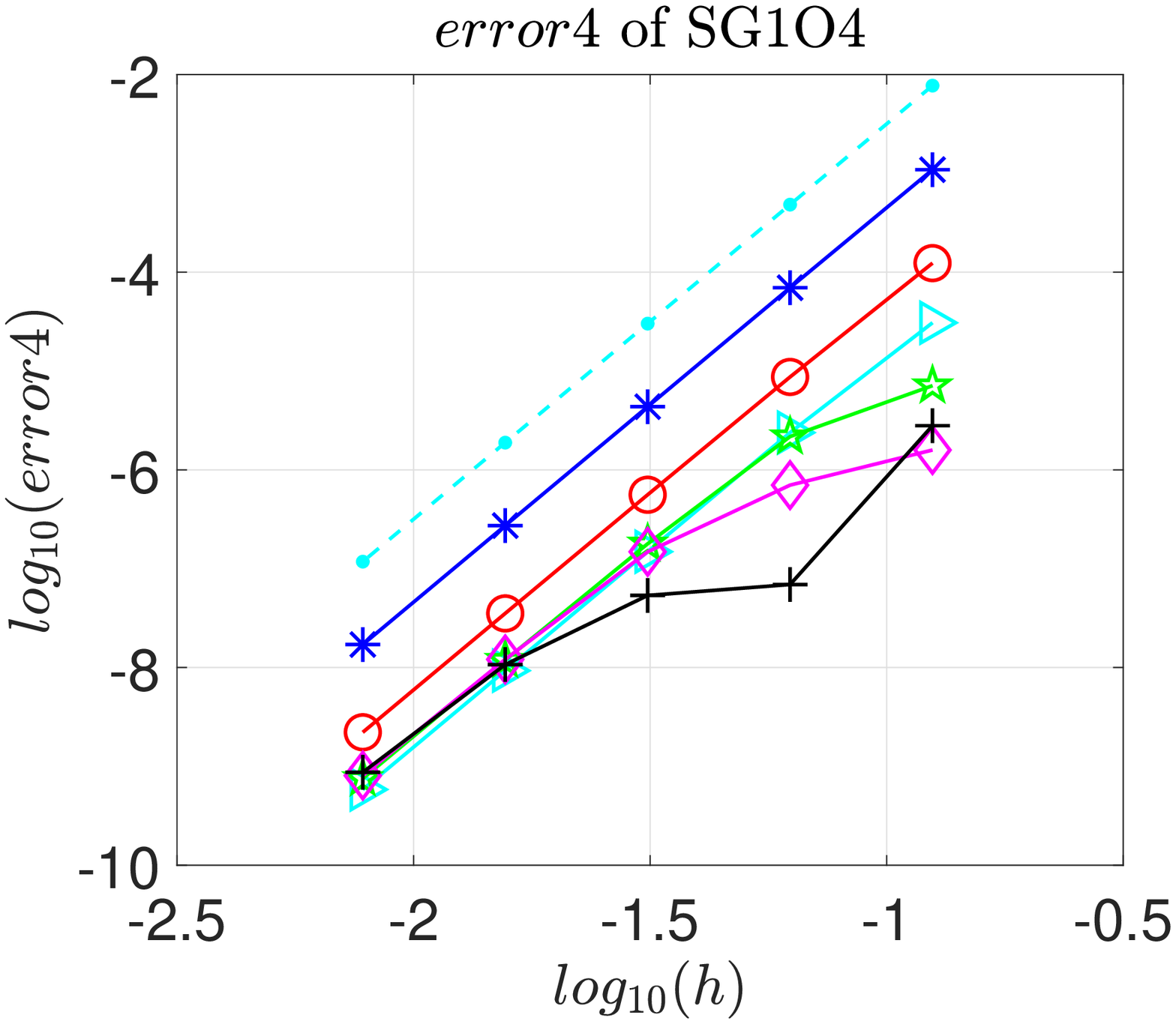}
\end{tabular}
\caption{Problem 3. The errors $error4:=\frac{\eps^2\norm{x_{n}-x(t_n)}}{\norm{x(t_n)}}+\frac{\eps^3 \norm{v_{n}-v(t_n)}}{\norm{v(t_n)}}$ of fourth order methods (RKO4 and SG1O4) with $t=1$ and $h=1/2^{k}$ for $k=3,4,\ldots,7$ under different $\epsilon$.}
\label{fig:problem41new2}
\end{figure}

\begin{figure}[t!]
\centering\tabcolsep=0.4mm
\begin{tabular}
[c]{ccc}%
\includegraphics[width=4.7cm,height=4.5cm]{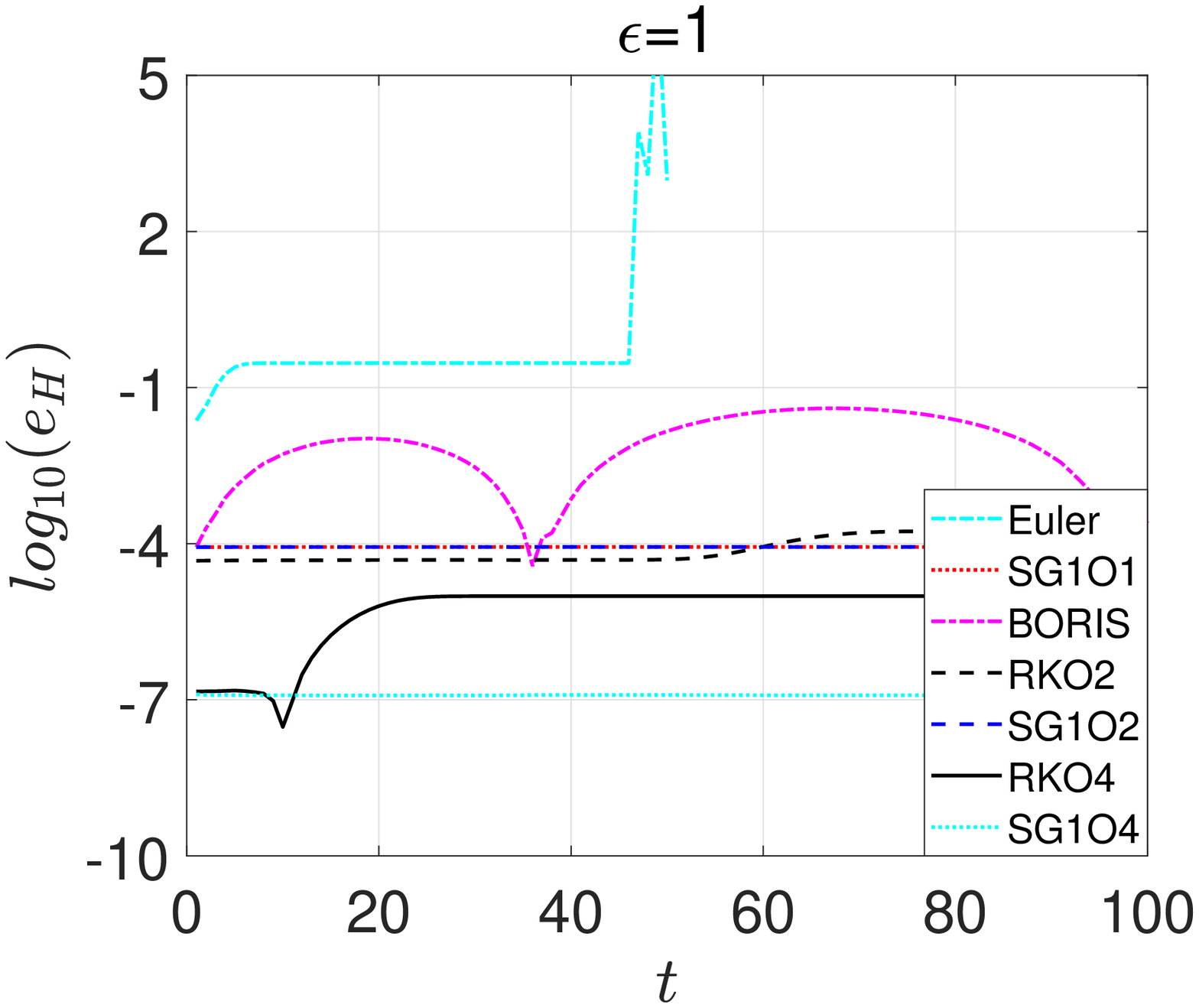} & \includegraphics[width=4.7cm,height=4.5cm]{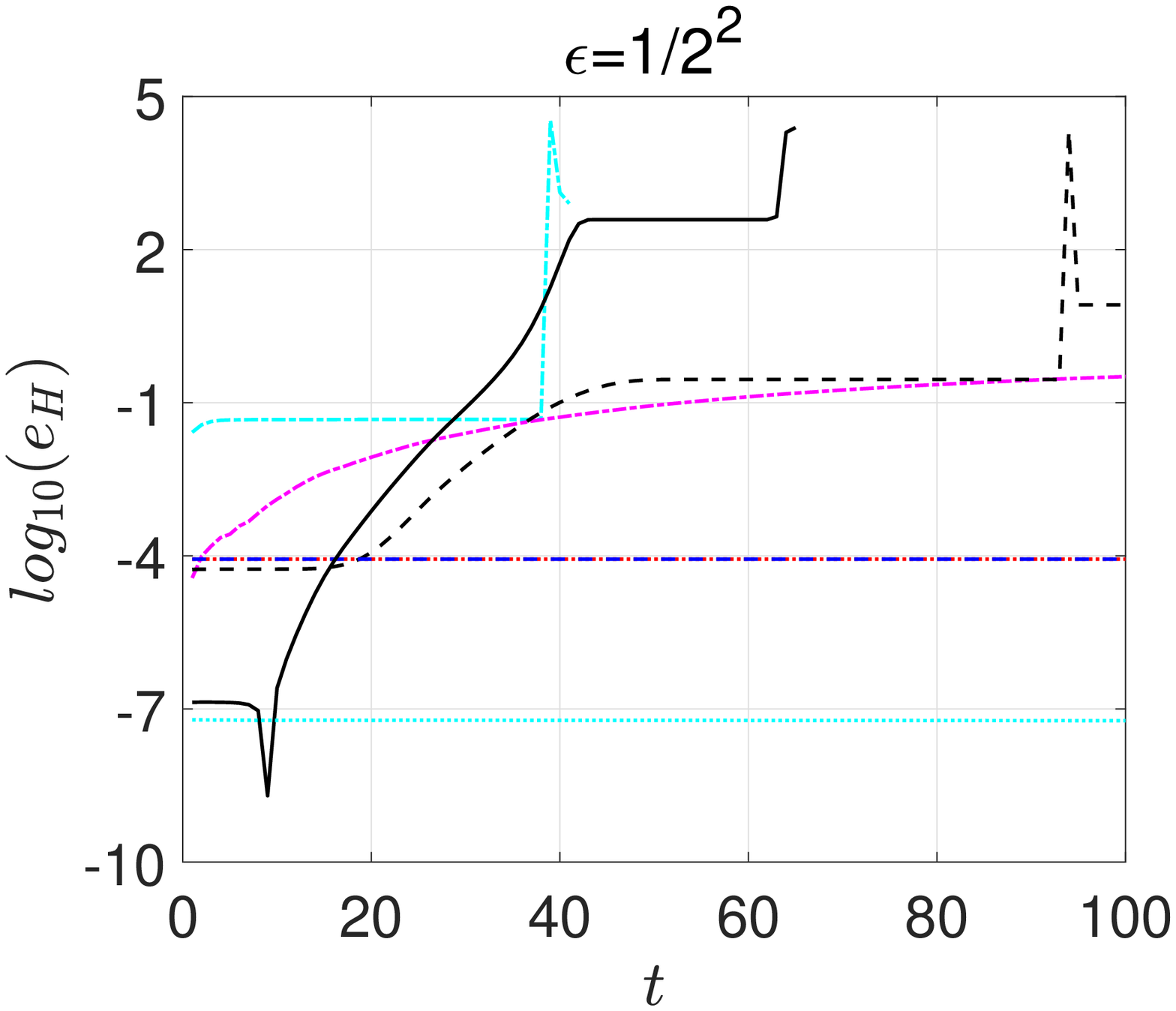} & \includegraphics[width=4.7cm,height=4.5cm]{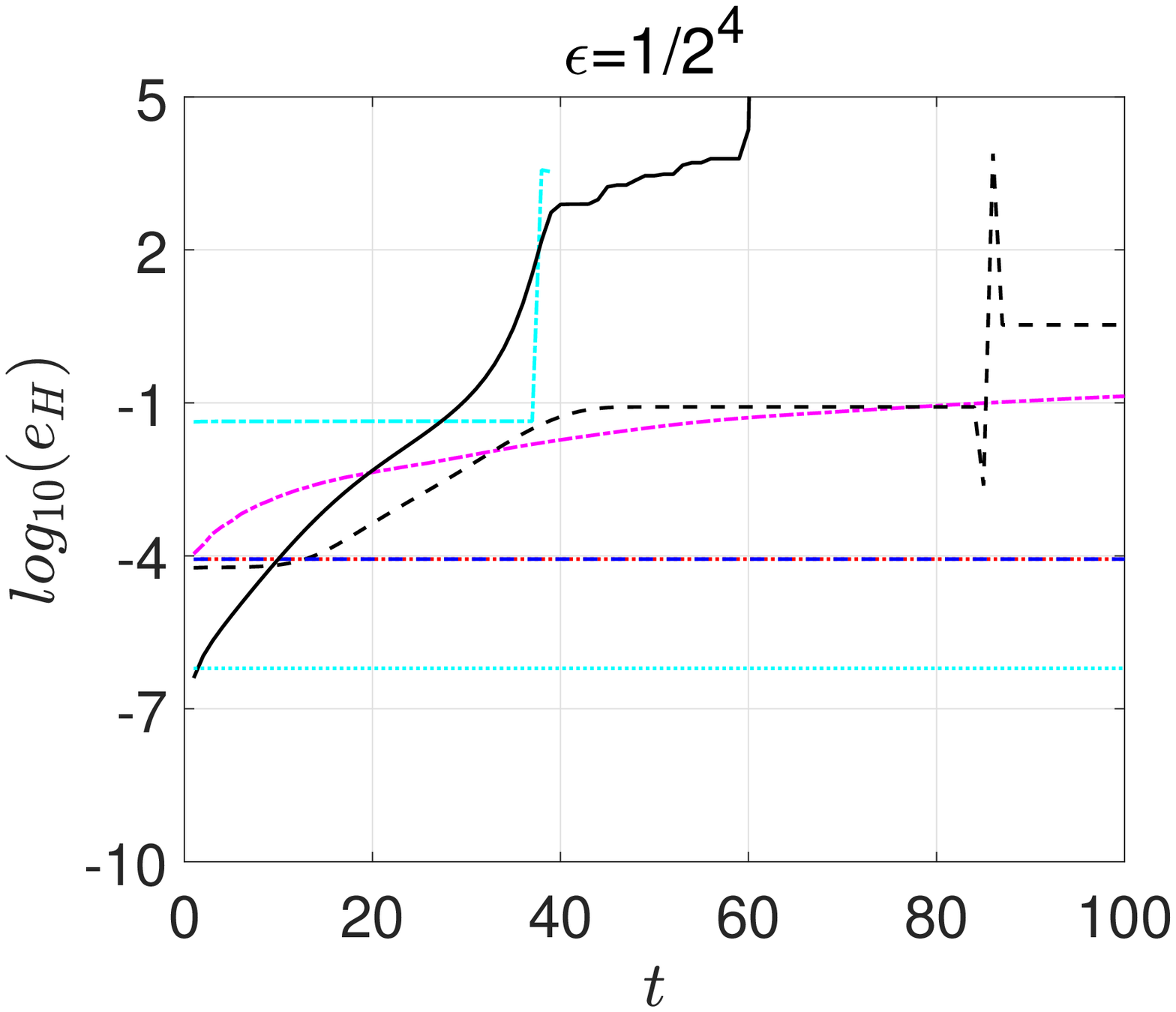}
\end{tabular}
\caption{Problem 3. Evolution of the energy error $e_{H}:=\frac{|H(x_{n},v_n)-H(x_0,v_0)|}{|H(x_0,v_0)|}$ as function of time  $t=nh$.}
\label{fig:problem42}
\end{figure}
Based on the above results, we can draw the following observations. For the accuracy, our methods agree with the results presented in Theorem \ref{order condition-B(x)} and  \ref{order conditionN} and are more accurate than the Euler, Boris and Runge-Kutta methods. Concerning the energy conservation, it can be seen that our methods have a long time conservation not only for normal magnetic fields but also for strong ones.

\section{Conclusions} \label{sec:conclusions}
{
In this paper,    symplectic  methods
 for solving the charged-particle dynamics (CPD) \eqref{charged-particle sts-cons} were presented and studied.
 By employing some transformations of the system and methods, we derived symplecticity conditions for
a novel  kind of adapted exponential methods, and based on which,
two  symplectic methods up to order four were constructed for solving
CPD in a strong and homogeneous magnetic field.
Rigorous error estimates were presented  and the proposed second order symplectic method was shown to have   uniform error bound in the position w.r.t. the strength of the magnetic field.
 Furthermore,   the extension of the obtained symplectic methods to
the case of non-homogeneous magnetic fields was discussed. Three novel algorithms up to order four were constructed
and one method was proved to have uniform accuracy in the position.
Some numerical
tests on homogeneous and non-homogeneous magnetic fields were  presented to confirm the theoretical results and  to demonstrate the numerical behaviour in accuracy and energy conservation.

Last but not least, we point out that the numerical results of Problems 1--3 demonstrate  a very good long time energy conservation for the methods presented in this paper. To theoretically prove this property, the strategies named as
  backward error analysis \cite{Hairer2017,Lubich2017,Hairer2002,wangwu2020} and modulated Fourier expansion \cite{Hairer2018,Hairer2020,Hairer2002,Lubich2020,WZ} will be employed  for CPD under   normal and strong magnetic fields, respectively.  The rigorous analysis on this topic will be considered in our next work.  Another issue for future exploration is the study of uniform higher-order symplectic  integrators. Besides, a combination of symplectic  integrators and the Particle-In-Cell (PIC) approximation and its analysis are of great interests for Vlasov equations   \cite{VP1,Zhao,VP4,VP5,VP6}.

  }

%\section*{Acknowledgement}
%The authors are sincerely thankful to three anonymous reviewers for
%their valuable suggestions, which help improve the presentation of
%the manuscript significantly.

\end{document}